\documentclass[11pt, reqno]{amsart}
\usepackage{amssymb,latexsym,amsmath,amsfonts,mathdots,enumitem}
\usepackage{latexsym}
\usepackage[mathscr]{eucal}
\usepackage{colortbl,xcolor}
\usepackage{lmodern}
\usepackage{sansmathaccent}
\usepackage[latin1]{inputenc}
\usepackage{tikz}
\usepackage{physics}
\usetikzlibrary{shapes,arrows}
\allowdisplaybreaks
\usetikzlibrary{matrix,calc,shapes,arrows,positioning}
\pdfmapfile{+sansmathaccent.map}

\numberwithin{equation}{section}
\theoremstyle{plain}
\newtheorem{exam}{Example}

\newtheorem{thm}{Theorem}[section]

\newtheorem{cor}[thm]{Corollary}
\newtheorem{lem}[thm]{Lemma}
\newtheorem{prop}[thm]{Proposition}
\theoremstyle{definition}
\newtheorem{defn}[thm]{Definition}
\newtheorem{rem}[thm]{Remark}

\numberwithin{equation}{section}

\def\beq{\begin{eqnarray}}
	\def\eeq{\end{eqnarray}}
\def\beqa{\begin{eqnarray*}}
	\def\eeqa{\end{eqnarray*}}

\def\beqn{\begin{equation}}
	\def\eeqn{\end{equation}}

\def\mg#1{}

\renewcommand{\epsilon}{\varepsilon}
\renewcommand{\phi}{\varphi}

\renewcommand{\bf}[1]{\textbf{#1}}
\renewcommand{\it}[1]{\textit{#1}}

\renewcommand{\sf}[1]{\textsf{#1}}

\numberwithin{equation}{section}
\allowdisplaybreaks[4] 

\setlist[enumerate]{font=\upshape,noitemsep, topsep=0pt} 
\setlist[itemize]{noitemsep, topsep=0pt}

\begin{document}
	
	\title{On the Dilation Theory and Canonical Decomposition of $\mathbf{\Theta}_n$-Contractions }
	\author{Aparna Gupta, \, Avijit Pal,  \, Bhaskar Paul}

\address[ A. Gupta]{Department of Mathematics, IIT Bhilai, 6th Lane Road, Jevra, Chhattisgarh 491002, India}
\email{aparnagupta@iitbhilai.ac.in}

\address[A. Pal]{Department of Mathematics, IIT Bhilai, 6th Lane Road, Jevra, Chhattisgarh 491002, India}
\email{A. Pal:avijit@iitbhilai.ac.in}

\address[B. Paul]{Department of Mathematics, IIT Bhilai, 6th Lane Road, Jevra, Chhattisgarh 491002, India}
\email{B. Paul:bhaskarpaul@iitbhilai.ac.in }
\subjclass[2010]{47A15, 47A20, 47A25, 47A45}
	
	\keywords{$\mathbf{\Theta}_n$-contraction, $\mathbf{\Theta}_n$-unitary, $\mathbf{\Theta}_n$-isometry, Conditional dilation, $\mu$-synthesis domains, Distinguish boundary, Completely non-unitary contraction}
	
	\maketitle
\begin{abstract}
This paper investigates the domain $\mathbf{\Theta}_n$ from an
operator-theoretic perspective. We establish several characterizations of
$\mathbf{\Theta}_n$-contractions, $\mathbf{\Theta}_n$-unitaries, and
$\mathbf{\Theta}_n$-isometries, and explore their connections with
$\Gamma_n$- and tetrablock operator tuples, as well as with the
corresponding operator classes associated with $\mathbf{\Theta}_{n+1}$.
We prove that every $\mathbf{\Theta}_n$-contraction admits a canonical
decomposition as the direct sum of a $\mathbf{\Theta}_n$-unitary and a
completely non-unitary $\mathbf{\Theta}_n$-contraction. We further develop
the conditional dilation theory for $\mathbf{\Theta}_n$-contractions and
establish necessary conditions for the existence of
$\mathbf{\Theta}_n$-isometric dilations. We show that the conditions in
Theorem~\ref{Conditional Dilation} are not, in general, sufficient by
constructing a $\mathbf{\Theta}_3$-contraction that admits a
$\mathbf{\Theta}_3$-isometric dilation although condition~(2) of
Theorem~\ref{Conditional Dilation} fails. We further study pure
$\mathbf{\Theta}_n$-isometric dilations of
$\mathbf{\Theta}_n$-contractions and establish a Hardy space
characterization of such dilations. We show that the resulting
$\mathbf{\Theta}_n$-isometric dilation is minimal. We also study isometric dilations of doubly commuting
$\mathbf{\Theta}_n$-contractions under suitable Hardy space
compatibility conditions. Finally, we identify a special class of
$\mathbf{\Theta}_2$-contractions that admits a
$\mathbf{\Theta}_2$-isometric dilation.
\end{abstract}
\section{Introduction}\label{Intro}
The \textit{symmetrized polydisc} has emerged as one of the central domains in multivariable operator theory. Since the pioneering work of Agler and Young \cite{ay,ay1,McCarthy 1,McCarthy}, it has provided a natural framework in which complex geometry, spectral set theory, dilation theory, function theory, and operator models interact. Over the last two decades, this interaction has led to remarkable developments, including canonical models, functional calculi, von Neumann type inequalities, and dilation theorems for $\Gamma_n$-contractions.
Motivated by these developments, several domains closely related to the symmetrized polydisc have recently been introduced in an attempt to extend this operator-theoretic framework. Among them is the family of generalized symmetrized domains $\mathbf{\Theta}_n$, introduced in \cite{Biswas 2,Ghosh}. These domains contain the symmetrized polydisc as a special case and share many of its geometric features. Their function theory has already begun to emerge, but very little is known from the perspective of operator theory.

Thus, the present paper extends several fundamental aspects of the operator theory of the symmetrized polydisc to the considerably broader family of generalized symmetrized domains. Our main objective is to develop the operator theory associated with $\mathbf{\Theta}_n$, including characterizations of $\mathbf{\Theta}_n$-unitaries and $\mathbf{\Theta}_n$-isometries, canonical decomposition theorems, fundamental operator tuples, and dilation theory. In this sense, the present work may be viewed as a natural continuation of the operator theory of the symmetrized polydisc.

Let $\mathbb D$ denote the open unit disc in $\mathbb C$. The \textit{symmetrization map}
$
\mathbf s=(s_1,\ldots,s_n):\mathbb C^n\to\mathbb C^n$
is defined by
\[
s_i(z_1,\ldots,z_n)
=
\sum_{1\le k_1<\cdots<k_i\le n}
z_{k_1}\cdots z_{k_i},
\qquad
1\le i\le n.
\]
Its image
$
G_n=\mathbf s(\mathbb D^n)$
is the symmetrized polydisc, while
$
\Gamma_n=\mathbf s(\overline{\mathbb D}^{\,n})$
denotes its closure.
A natural generalization of the symmetrization map was introduced in \cite{Ghosh}. Let $p$ be a fixed divisor of $m$, let $\theta_0=1$, and define
\[
\theta_i(z_1,\ldots,z_n)
=
s_i(z_1^m,\ldots,z_n^m),
\qquad
1\le i\le n-1,
\]
and
\[
\theta_n(z_1,\ldots,z_n)
=
(z_1\cdots z_n)^q,
\qquad
q=\frac{m}{p}.
\]
The images of $\mathbb D^n$ and $\overline{\mathbb D}^{\,n}$ under
$
\boldsymbol{\theta}=(\theta_1,\ldots,\theta_n)$
are denoted by $\mathbf{\Theta}_n$ and $\overline{\mathbf{\Theta}}_n$, respectively. Observe that
$
\mathbf{\Theta}_1=\mathbb D,$
and that
$
\mathbf{\Theta}_n=G_n$
when $m=p=1$. It was shown in \cite{Biswas 2,Ghosh} that
$
(\theta_1,\ldots,\theta_n)\in\overline{\mathbf{\Theta}}_n$
if and only if every zero of
\begin{equation}\label{P}
P(z)
=
z^n-\theta_1z^{n-1}
+\cdots
+(-1)^n\theta_n^p
\end{equation}
lies in $\overline{\mathbb D}$.

A fundamental tool for studying operator tuples associated with a complex domain is the notion of a spectral set. Let $\Omega\subset\mathbb C^n$ be compact and let $\mathcal O(\Omega)$ denote the algebra of functions that are holomorphic on a neighbourhood of $\Omega$. A commuting tuple
$
\mathbf T=(T_1,\ldots,T_n)$
is said to have $\Omega$ as a spectral set if
$
\sigma(\mathbf T)\subseteq\Omega$
and
\[
\|f(\mathbf T)\|
\le
\sup_{z\in\Omega}|f(z)|
\]
for every $f\in\mathcal O(\Omega)$.
The importance of spectral sets originates in the celebrated theorem of von Neumann, which asserts that the closed unit disc is a spectral set for every contraction.
\begin{thm}[von Neumann's inequality {\cite[Chapter~1, Corollary~1.2]{Nagy}}]
Let $T$ be a contraction. Then
\[
\|p(T)\|
\le
\sup_{|z|\le1}|p(z)|
\]
for every polynomial $p$.
\end{thm}
Closely related to this result is the dilation theorem of Sz.-Nagy, which states that every contraction admits a unitary power dilation.
\begin{thm}[Sz.-Nagy {\cite{paulsen}}]
Every contraction admits a unitary power dilation.
\end{thm}
These two theorems form the foundation of modern dilation theory and have inspired analogous theories for many domains, including the symmetrized polydisc, the tetrablock, and the pentablock. It is therefore natural to ask whether the generalized domains $\mathbf{\Theta}_n$ admit a similar operator theory.
Following \cite{Biswas 2}, we introduce the basic operator classes associated with $\mathbf{\Theta}_n$: $\mathbf{\Theta}_n$-contractions, $\mathbf{\Theta}_n$-unitaries, $\mathbf{\Theta}_n$-isometries, $\mathbf{\Theta}_n$-co-isometries, and pure $\mathbf{\Theta}_n$-isometries.
\begin{defn}\label{Defition 1}
\begin{enumerate}
\item  Let $\mathbf{T}=(T_1,\ldots,T_n)$ be a commuting $n$-tuple of bounded operators on a Hilbert space $\mathcal H$. We say that $\mathbf{T}$ is a \emph{$\mathbf{\Theta}_n$-contraction} if $\overline{\mathbf{\Theta}}_n$ is a spectral set for $\mathbf{T}.$
\item Let $\mathbf{N}=(N_1,\ldots,N_n)$ be a commuting $n$-tuple of bounded normal operators on a Hilbert space $\mathcal H$. We say that $\mathbf{N}$ is a \emph{$\mathbf{\Theta}_n$-unitary} if its joint spectrum is contained in the distinguished boundary $b\mathbf{\Theta}_n$.

\item Let $\mathbf{V}=(V_1,\ldots,V_n)$ be a commuting $n$-tuple of bounded operators on a Hilbert space $\mathcal H$. We say that $\mathbf{V}$ is a \emph{$\mathbf{\Theta}_n$-isometry} if there exist a Hilbert space $\mathcal K \supseteq \mathcal H$ and a $\mathbf{\Theta}_n$-unitary $\mathbf{N}=(N_1,\ldots,N_n)$ on $\mathcal K$ such that
$
V_i=N_i|_{\mathcal H}, 1\le i\le n.$
We say that $\mathbf{V}$ is a \emph{$\mathbf{\Theta}_n$-co-isometry} if
$
\mathbf{V}^*=(V_1^*,\ldots,V_n^*)$
is a $\mathbf{\Theta}_n$-isometry. A $\mathbf{\Theta}_n$-isometry $\mathbf{V}$ is said to be a \emph{pure $\mathbf{\Theta}_n$-isometry} if $V_n$ is a pure isometry.

\end{enumerate}
\end{defn}
A contraction $T$ on a Hilbert space $\mathcal H$ is said to be \emph{completely non-unitary} (briefly, \emph{c.n.u.}) if it has no non-trivial reducing subspace on which it acts as a unitary operator. The unilateral shift is the standard example of a c.n.u. contraction. 
This leads to the following definition. \begin{defn}\label{c.n.u.} A $\mathbf{\Theta}_n$-contraction $\mathbf T=(T_1,\ldots,T_n)$ is called a \emph{c.n.u. $\mathbf{\Theta}_n$-contraction} if $T_n$ is a completely non-unitary contraction. \end{defn} 
One of the principal tools in the study of $\Gamma_n$-contractions is the theory of fundamental operators \cite{S. Pal, A. Pal}. These operators encode the deviation of a $\Gamma_n$-contraction from being a $\Gamma_n$-isometry and play a central role in the construction of functional models, canonical decompositions, and dilation theorems. A natural question is whether an analogous theory exists for $\mathbf{\Theta}_n$-contractions. This question motivates the system of operator equations introduced below. To introduce these equations, let $T$ be a contraction on a Hilbert space $\mathcal H$. The associated defect operator and defect space are defined by $ D_T=(I-T^*T)^{1/2}, ~\mathcal D_T=\overline{\operatorname{Ran}}D_T, $
respectively. For a $\Gamma_n$-contraction $(S_1,\ldots,S_n)$, the fundamental equations are 
\begin{equation}\label{Gamma_n Fundamental} \begin{aligned} S_i-S_{n-i}^*S_n=D_{S_n}E_iD_{S_n},S_{n-i}-S_i^*S_n=D_{S_n}E_{n-i}D_{S_n}, \quad E_i,E_{n-i}\in\mathcal B(\mathcal D_{S_n}), \end{aligned} \end{equation} 
for $1\le i\le n-1$.  Motivated by this observation, we introduce the following system of operator equations associated with a $\mathbf{\Theta}_n$-contraction $\mathbf T=(T_1,\ldots,T_n)$.
\begin{equation}\label{Fundamental} \begin{aligned} T_i-T_{n-i}^*T_n^p &= \sum_{l=0}^{p-1} T_n^{*\,p-1-l} D_{T_n}A_l^{(i)}D_{T_n} T_n^{p-1}  + \sum_{l=0}^{p-2} T_n^{*\,p-2-l} D_{T_n}A_l^{(i)}D_{T_n} T_n^{p-2} +\cdots \\ & + \sum_{l=0}^{1} T_n^{*\,1-l} D_{T_n}A_l^{(i)}D_{T_n} T_n + D_{T_n}A_0^{(i)}D_{T_n}, \qquad 1\le i\le n-1. \end{aligned} \end{equation} 
We refer to \eqref{Fundamental} as the \emph{fundamental equations} of $\mathbf T$, and $$(A_0^{(i)},\ldots,A_{p-1}^{(i)}), ~~1\le i\le n-1, $$ are the  \emph{tuples  of fundamental operators} of $\mathbf T$. Observe that when $m=p=1$, these equations reduce precisely to the fundamental equations for $\Gamma_n$-contractions with $A_{0}^{(i)}=E_i, 1\leq i \leq (n-1)$. It was shown in  \cite[Theorem 4.4]{A. Pal} that the fundamental operators for $\Gamma_n$-contractions are unique.

Let $\mathcal E$ denotes a separable Hilbert space, and $\mathcal B(\mathcal E)$ denotes the algebra of bounded linear operators on $\mathcal E$. We write $H^2(\mathcal E)$ for the Hardy space of analytic $\mathcal E$-valued functions on the unit disc $\mathbb D$, and $L^2(\mathcal E)$ for the Hilbert space of square-integrable $\mathcal E$-valued functions on the unit circle $\mathbb T$. Furthermore, $H^\infty(\mathcal B(\mathcal E))$ and $L^\infty(\mathcal B(\mathcal E))$ denote the spaces of bounded analytic and bounded measurable $\mathcal B(\mathcal E)$-valued functions, respectively. For $\phi\in L^\infty(\mathcal B(\mathcal E))$, the Toeplitz operator with symbol $\phi$ is defined by \[ T_\phi f=P_+(\phi f), \qquad f\in H^2(\mathcal E), \] where $P_+:L^2(\mathcal E)\to H^2(\mathcal E)$ is the orthogonal projection. In particular, we denote $T_z=M_z$ as the unilateral shift and $T_{\bar z}=M_z^*$ is the backward shift on $H^2(\mathcal E)$. For convenience, throughout the paper we write
\[
k(i)=\binom{n-1}{i}+\binom{n-1}{n-i},
\quad
\gamma_i=\frac{n-i}{n},\quad 1\leq i \leq n-1.
\]

The paper is organized as follows. In Section~\ref{Char}, we establish
several characterizations and structural properties of
$\mathbf{\Theta}_n$-contractions. We also discuss the uniqueness issue
associated with the fundamental operators and emphasize that, for
$p\geq 2$, the uniqueness of the family
\[
\bigl(A_0^{(i)},A_1^{(i)},\ldots,A_{p-1}^{(i)}\bigr),
\quad 1\leq i\leq n-1,
\]
does not follow from the fundamental equation~\eqref{Fundamental} alone.
Section~\ref{Unitary, Isometry} is devoted to the study of
$\mathbf{\Theta}_n$-unitaries and $\mathbf{\Theta}_n$-isometries, together
with their connections with $\Gamma_n$-, $\mathbf{\Theta}_{n+1}$-, and
$\mathbb{E}$-unitaries and isometries. In Section~\ref{Decomposition}, we
establish a canonical decomposition theorem for
$\mathbf{\Theta}_n$-contractions, providing a structural description of
these operator tuples.
The final section develops the conditional dilation theory for
$\mathbf{\Theta}_n$-contractions and contains our main results on
$\mathbf{\Theta}_n$-isometric dilations. We first establish necessary
conditions for the existence of a $\mathbf{\Theta}_n$-isometric dilation.
We then show, by means of an explicit example, that these conditions are
not necessary in general. More precisely, we construct a
$\mathbf{\Theta}_3$-contraction that admits a
$\mathbf{\Theta}_3$-isometric dilation although condition~(2) of
Theorem~\ref{Conditional Dilation} is not satisfied. We further
investigate pure $\mathbf{\Theta}_n$-isometric dilations of
$\mathbf{\Theta}_n$-contractions and obtain a Hardy space
characterization of such dilations. We also show that the resulting
$\mathbf{\Theta}_n$-isometric dilation is minimal. In addition, we study
isometric dilations of doubly commuting
$\mathbf{\Theta}_n$-contractions under suitable Hardy space
compatibility conditions. Finally, we construct a special class of
$\mathbf{\Theta}_2$-contractions that admits a
$\mathbf{\Theta}_2$-isometric dilation. These results demonstrate that the
existence of $\mathbf{\Theta}_n$-isometric dilations extends beyond the
scope of the conditional dilation framework and provides further
structural insight into the dilation theory of
$\mathbf{\Theta}_n$-contractions.

\section{Characterizations of $\mathbf{\Theta}_n$-Contractions}\label{Char}

The generalized symmetrized domains $\mathbf{\Theta}_n$, the symmetrized polydisc, and the tetrablock are closely related (see \cite[Lemmas~2.1--2.8]{Paul}). These relationships provide valuable insights into the operator theory associated with these domains. In this section, we establish several characterizations of $\mathbf{\Theta}_n$-contractions and investigate their connections with $\Gamma_n$-contractions, $\mathbf{\Theta}_{n+1}$-contractions, and $\mathbb{E}$-contractions. We also introduce certain operator-valued functions on $\mathcal H$ that play a key role in the proofs of the main results. We begin with a useful characterization of the symmetrized polydisc. The following lemma is a slight extension of \cite[Lemma~2.3]{S. Pal}. Since its proof is identical to that of the equivalence $(1)\Leftrightarrow(3)$ in \cite[Lemma~2.3]{S. Pal}, we omit it.

\begin{lem}\label{Lem Gamma_n}
Let $(s_1,\dots,s_n)\in\mathbb C^n$. Then the following are equivalent:
\begin{enumerate}
\item $(s_1,\dots,s_n)\in\Gamma_n$.

\item $(\alpha s_1,\dots,\alpha^n s_n)\in\Gamma_n$ for every $\alpha\in\overline{\mathbb D}$.
\end{enumerate}
\end{lem}	
\begin{lem}\label{Lem 0}
Let $(\theta_1,\ldots,\theta_n)\in\overline{\mathbf{\Theta}}_n$. Then
\[
\left(\alpha^i\frac{\theta_i}{k(i)},
      \alpha^{n-i}\frac{\theta_{n-i}}{k(i)},
      \alpha^n\theta_n^p\right)
      \in\overline{\mathbb E}
\]
for every $\alpha\in\overline{\mathbb D}$ and $1\le i\le n-1$.
\end{lem}
\begin{proof}
Since $(\theta_1,\ldots,\theta_n)\in\overline{\mathbf{\Theta}}_n$, it follows from \cite[Lemma~2.2]{Paul} that
$
(\theta_1,\ldots,\theta_{n-1},\theta_n^p)\in\Gamma_n.$
Applying Lemma~\ref{Lem Gamma_n}, we obtain
$
(\alpha\theta_1,\ldots,\alpha^{n-1}\theta_{n-1},
\alpha^n\theta_n^p)\in\Gamma_n$
for every $\alpha\in\overline{\mathbb D}$. The conclusion now follows immediately from \cite[Lemma~2.8]{Paul}.
\end{proof}
\begin{rem}
It follows from the characterization of the closed tetrablock
$\overline{\mathbb E}$ (see \cite[Theorem~2.4]{Abouhajar}) that the rational functions
\begin{equation}\label{Psi, Upsilon}
\Psi(z,x_1,x_2,x_3)
=
\frac{x_1-zx_3}{1-zx_2},
\quad
\Upsilon(z,x_1,x_2,x_3)
=
\frac{x_2-zx_3}{1-zx_1},
\end{equation}
are contractive on $\overline{\mathbb E}\times\mathbb D$. Consequently, Lemma~\ref{Lem 0} implies that
\[
\left|
\frac{\alpha^i\theta_i-zk(i)\alpha^n\theta_n^p}
{k(i)-z\alpha^{n-i}\theta_{n-i}}
\right|
\le1
\quad 
\text{and}\quad
\left|
\frac{\alpha^{n-i}\theta_{n-i}-zk(i)\alpha^n\theta_n^p}
{k(i)-z\alpha^i\theta_i}
\right|
\le1,
\]
for every $(\theta_1,\ldots,\theta_n)\in\overline{\mathbf{\Theta}}_n$,
$\alpha\in\overline{\mathbb D}$, $z\in\mathbb D$, and
$1\le i\le n-1$.
Since $\Psi$ and $\Upsilon$ are continuous on
$\overline{\mathbb E}\times\overline{\mathbb D}$, the above inequalities remain valid for every $z\in\overline{\mathbb D}$. In particular, taking $z=1$, we obtain
\[
\left|
\frac{\alpha^i\theta_i-k(i)\alpha^n\theta_n^p}
{k(i)-\alpha^{n-i}\theta_{n-i}}
\right|
\le1
\quad
\text{and}\quad
\left|
\frac{\alpha^{n-i}\theta_{n-i}-k(i)\alpha^n\theta_n^p}
{k(i)-\alpha^i\theta_i}
\right|
\le1.
\]
\end{rem}
The following result is an immediate consequence of the definition of a
$\mathbf{\Theta}_n$-contraction, and hence its proof is omitted.	
\begin{prop}\label{Estimate T_i, T_n}
Let $(T_1,\ldots,T_n)$ be a $\mathbf{\Theta}_n$-contraction. Then
\[
\|T_i\|\le k(i), \quad 1\le i\le n-1, \quad \text{and}\quad
\|T_n\|\le1.
\]
\end{prop}
We next introduce two operator-valued functions that play an important role throughout the paper. Let $(T_1,\ldots,T_n)$ be a commuting $n$-tuple of bounded operators on a Hilbert space $\mathcal H$. For each $1\le i\le n-1$, define
\begin{equation}\label{Phi_1}
\begin{aligned}
\Phi^{(i)}_1(\alpha^iT_i,\alpha^{n-i}T_{n-i},\alpha^nT_n^p)
&=k(i)^2\bigl(I-|\alpha|^{2n}T_n^{*p}T_n^p\bigr)
+\bigl(|\alpha|^{2i}T_i^*T_i-|\alpha|^{2(n-i)}T_{n-i}^*T_{n-i}\bigr)\\
&\quad-k(i)\alpha^i\bigl(T_i-|\alpha|^{2(n-i)}T_{n-i}^*T_n^p\bigr)-k(i)\overline{\alpha}^{\,i}\bigl(T_i^*-|\alpha|^{2(n-i)}T_{n-i}T_n^{*p}\bigr),
\end{aligned}
\end{equation}
and
\begin{equation}\label{Phi_2}
\begin{aligned}
\Phi^{(i)}_2(\alpha^iT_i,\alpha^{n-i}T_{n-i},\alpha^nT_n^p)
&=k(i)^2\bigl(I-|\alpha|^{2n}T_n^{*p}T_n^p\bigr)
+\bigl(|\alpha|^{2(n-i)}T_{n-i}^*T_{n-i}-|\alpha|^{2i}T_i^*T_i\bigr)\\
&\quad-k(i)\alpha^{\,n-i}\bigl(T_{n-i}-|\alpha|^{2i}T_i^*T_n^p\bigr)-k(i)\overline{\alpha}^{\,n-i}\bigl(T_{n-i}^*-|\alpha|^{2i}T_iT_n^{*p}\bigr).
\end{aligned}
\end{equation}
The following lemma establishes a fundamental connection between $\mathbf{\Theta}_n$-contractions and $\Gamma_n$-contractions.
\begin{lem}\label{Lem 1}
Let $\mathbf T=(T_1,\ldots,T_n)$ be a commuting $n$-tuple of bounded operators on a Hilbert space $\mathcal H$. If $\mathbf T$ is a $\mathbf{\Theta}_n$-contraction, then
$
(T_1,\ldots,T_{n-1},T_n^p)$
is a $\Gamma_n$-contraction.
\end{lem}
\begin{proof}
Since $\mathbf T$ is a $\mathbf{\Theta}_n$-contraction, $\overline{\mathbf{\Theta}}_n$ is a spectral set for $\mathbf T$. Define
\(
\pi_p:\mathbb C^n\longrightarrow\mathbb C^n\) by
\[\pi_p(z_1,\ldots,z_n)
=(z_1,\ldots,z_{n-1},z_n^p).
\]
By \cite[Lemma~2.2]{Paul},
$
\pi_p(\overline{\mathbf{\Theta}}_n)\subseteq\Gamma_n.$
Therefore, for every polynomial $\widetilde p$ in $n$ variables,
\[
\begin{aligned}
\|\widetilde p(T_1,\ldots,T_{n-1},T_n^p)\|
&=\|(\widetilde p\circ\pi_p)(T_1,\ldots,T_n)\|\\
&\le
\|\widetilde p\circ\pi_p\|_{\infty,\overline{\mathbf{\Theta}}_n}\\
&=
\|\widetilde p\|_{\infty,\pi_p(\overline{\mathbf{\Theta}}_n)}\\
&\le
\|\widetilde p\|_{\infty,\Gamma_n}.
\end{aligned}
\]
Hence $\Gamma_n$ is a spectral set for $
(T_1,\ldots,T_{n-1},T_n^p),$
and consequently this tuple is a $\Gamma_n$-contraction.
\end{proof}	
The next result shows how a $\mathbf{\Theta}_{n+1}$-contraction can be constructed from a given $\mathbf{\Theta}_n$-contraction.
\begin{lem}\label{Lem 2}
Let $\mathbf{T}=(T_1,\ldots,T_n)$ be a commuting $n$-tuple of bounded operators on a Hilbert space $\mathcal H$. If $\mathbf T$ is a $\mathbf{\Theta}_n$-contraction, then
\[
(\alpha^mI+T_1,\,
\alpha^mT_1+T_2,\,
\ldots,\,
\alpha^mT_{n-1}+T_n^p,\,
\alpha^{m/p}T_n)
\]
is a $\mathbf{\Theta}_{n+1}$-contraction for every $\alpha\in\overline{\mathbb D}$.
\end{lem}

\begin{proof}
Fix $\alpha\in\overline{\mathbb D}$. Since $\mathbf T$ is a $\mathbf{\Theta}_n$-contraction, $\overline{\mathbf{\Theta}}_n$ is a spectral set for $\mathbf T$. Define
$
\pi_\alpha:\mathbb C^n\longrightarrow\mathbb C^{n+1}$
by
\[
\pi_\alpha(z_1,\ldots,z_n)
=
(\alpha^m+z_1,\,
\alpha^mz_1+z_2,\,
\ldots,\,
\alpha^mz_{n-1}+z_n^p,\,
\alpha^{m/p}z_n).
\]
By \cite[Lemma~2.5]{Paul},
$
\pi_\alpha(\overline{\mathbf{\Theta}}_n)
\subseteq
\overline{\mathbf{\Theta}}_{n+1}.$
Hence, for every polynomial $q$ in $n+1$ variables,
\[
\begin{aligned}
&\|q(\alpha^mI+T_1,\alpha^mT_1+T_2,\ldots,
\alpha^mT_{n-1}+T_n^p,\alpha^{m/p}T_n)\|  \\
&\qquad
=
\|(q\circ\pi_\alpha)(T_1,\ldots,T_n)\|   \\
&\qquad
\le
\|q\circ\pi_\alpha\|_{\infty,\overline{\mathbf{\Theta}}_n}
=
\|q\|_{\infty,\pi_\alpha(\overline{\mathbf{\Theta}}_n)}
\le
\|q\|_{\infty,\overline{\mathbf{\Theta}}_{n+1}}.
\end{aligned}
\]
Therefore,
$
(\alpha^mI+T_1,\,
\alpha^mT_1+T_2,\,
\ldots,\,
\alpha^mT_{n-1}+T_n^p,\,
\alpha^{m/p}T_n)$
is a $\mathbf{\Theta}_{n+1}$-contraction.
\end{proof}
The following lemma establishes a connection between $\mathbf{\Theta}_n$-contractions and $\mathbb E$-contractions.
\begin{lem}\label{Lem 3}
Let $\mathbf T=(T_1,\ldots,T_n)$ be a commuting $n$-tuple of bounded operators on a Hilbert space $\mathcal H$. If $\mathbf T$ is a $\mathbf{\Theta}_n$-contraction, then
$
\left(
\frac{T_i}{k(i)},
\frac{T_{n-i}}{k(i)},
T_n^p
\right)$
is an $\mathbb E$-contraction for  $1\le i\le n-1$.
\end{lem}
\begin{proof}
Since $\mathbf T$ is a $\mathbf{\Theta}_n$-contraction, $\overline{\mathbf{\Theta}}_n$ is a spectral set for $\mathbf T$. Define
$
\pi_i:\mathbb C^n\longrightarrow\mathbb C^3$
by
\[
\pi_i(z_1,\ldots,z_n)
=
\left(
\frac{z_i}{k(i)},
\frac{z_{n-i}}{k(i)},
z_n^p
\right),
\quad
1\le i\le n-1.
\]
By \cite[Lemma~2.8]{Paul},
$
\pi_i(\overline{\mathbf{\Theta}}_n)
\subseteq
\overline{\mathbb E}.$
Therefore, for every polynomial $q$ in three variables,
\[
\begin{aligned}
\left\|
q\left(
\frac{T_i}{k(i)},
\frac{T_{n-i}}{k(i)},
T_n^p
\right)
\right\|
&=
\|(q\circ\pi_i)(T_1,\ldots,T_n)\|  \\
&\le
\|q\circ\pi_i\|_{\infty,\overline{\mathbf{\Theta}}_n}  \\
&=
\|q\|_{\infty,\pi_i(\overline{\mathbf{\Theta}}_n)}  \\
&\le
\|q\|_{\infty,\overline{\mathbb E}}.
\end{aligned}
\]
Hence
$
\left(
\frac{T_i}{k(i)},
\frac{T_{n-i}}{k(i)},
T_n^p
\right)$
is an $\mathbb E$-contraction.
\end{proof}
The following lemma is an immediate consequence of Lemma~\ref{Lem 0} and Lemma~\ref{Lem 3}, and hence its proof is omitted.
\begin{lem}\label{Lem 4}
Let $\mathbf T=(T_1,\ldots,T_n)$ be a commuting $n$-tuple of bounded operators on a Hilbert space $\mathcal H$. If $\mathbf T$ is a $\mathbf{\Theta}_n$-contraction, then
$
\left(
\alpha^i\frac{T_i}{k(i)},
\alpha^{n-i}\frac{T_{n-i}}{k(i)},
\alpha^nT_n^p
\right)$
is an $\mathbb E$-contraction for every $\alpha\in\overline{\mathbb D}$ and every $1\le i\le n-1$.
\end{lem}

The next proposition plays a fundamental role in the characterization of $\mathbf{\Theta}_n$-unitaries and $\mathbf{\Theta}_n$-isometries, as well as in the canonical decomposition of a $\mathbf{\Theta}_n$-contraction into the direct sum of a $\mathbf{\Theta}_n$-unitary and a completely non-unitary $\mathbf{\Theta}_n$-contraction.

\begin{prop}\label{Prop 1}
Let $\mathbf T=(T_1,\ldots,T_n)$ be a $\mathbf{\Theta}_n$-contraction. Then, for every $\alpha\in\overline{\mathbb D}$ and  $1\le i\le n-1$,
\[
\Phi^{(i)}_1(\alpha^iT_i,\alpha^{n-i}T_{n-i},\alpha^nT_n^p)\ge0
\quad \text{and}\quad
\Phi^{(i)}_2(\alpha^iT_i,\alpha^{n-i}T_{n-i},\alpha^nT_n^p)\ge0.
\]
\end{prop}

\begin{proof}
By Lemma~\ref{Lem 4},
$
\left(
\alpha^i\frac{T_i}{k(i)},
\alpha^{n-i}\frac{T_{n-i}}{k(i)},
\alpha^nT_n^p
\right)$
is an $\mathbb E$-contraction for every $\alpha\in\overline{\mathbb D}$ and $1\le i\le n-1$. Therefore, by the defining property of the rational function $\Psi$ in \eqref{Psi, Upsilon},
\[
\left\|
\Psi\left(
z,
\alpha^i\frac{T_i}{k(i)},
\alpha^{n-i}\frac{T_{n-i}}{k(i)},
\alpha^nT_n^p
\right)
\right\|
\le1,
\quad z\in\mathbb D.
\]
A straightforward computation shows that this inequality is equivalent to
\[
\Phi^{(i)}_1
(\alpha^iT_i,
z\alpha^{n-i}T_{n-i},
z\alpha^nT_n^p)
\ge0,
\quad z\in\mathbb D.
\]
Since $\Phi^{(i)}_1$ depends continuously on $z$, the inequality remains valid for every $z\in\overline{\mathbb D}$. Taking $z=1$, we obtain
\[
\Phi^{(i)}_1
(\alpha^iT_i,
\alpha^{n-i}T_{n-i},
\alpha^nT_n^p)
\ge0.
\]
Applying the same argument to the rational function $\Upsilon$ yields
\[
\Phi^{(i)}_2
(\alpha^iT_i,
\alpha^{n-i}T_{n-i},
\alpha^nT_n^p)
\ge0.
\]
This completes the proof.
\end{proof}
The following theorem provides an equivalent characterization of
$\mathbf{\Theta}_n$-contractions in terms of the von Neumann inequality.
\begin{thm}\label{Theta_n Contraction Char}
Let $\mathbf T=(T_1,\ldots,T_n)$ be a commuting $n$-tuple of bounded operators on a Hilbert space $\mathcal H$. Then the following are equivalent.
\begin{enumerate}
\item $\mathbf T$ is a $\mathbf{\Theta}_n$-contraction.

\item For every polynomial $p$ in $n$ variables,
\[
\|p(T_1,\ldots,T_n)\|
\le
\|p\|_{\infty,\overline{\mathbf{\Theta}}_n}.
\]
\end{enumerate}
\end{thm}

\begin{proof}
The implication $(1)\Rightarrow(2)$ follows immediately from the definition of a $\mathbf{\Theta}_n$-contraction.

Conversely, assume that {\rm(2)} holds. We first show that $
\sigma(\mathbf T)\subseteq\overline{\mathbf{\Theta}}_n.$
Suppose, to the contrary, that there exists
$
x=(x_1,\ldots,x_n)\in\sigma(\mathbf T)
\setminus
\overline{\mathbf{\Theta}}_n.$
Since $\overline{\mathbf{\Theta}}_n$ is polynomially convex, the Oka--Weil separation theorem yields a polynomial $p$ such that
$
|p(x)|
>
\|p\|_{\infty,\overline{\mathbf{\Theta}}_n}.$
By the spectral mapping theorem,
\[
\sigma\bigl(p(T_1,\ldots,T_n)\bigr)
=
p\bigl(\sigma(\mathbf T)\bigr).
\]
Hence,
\[
r\bigl(p(T_1,\ldots,T_n)\bigr)
\ge
|p(x)|
>
\|p\|_{\infty,\overline{\mathbf{\Theta}}_n},
\]
where $r(\cdot)$ denotes the spectral radius. Since
\[
r\bigl(p(T_1,\ldots,T_n)\bigr)
\le
\|p(T_1,\ldots,T_n)\|,
\]
this contradicts {\rm(2)}. Therefore,
$
\sigma(\mathbf T)\subseteq\overline{\mathbf{\Theta}}_n.$
Finally, since $\overline{\mathbf{\Theta}}_n$ is polynomially convex, the Oka--Weil theorem \cite[Theorem~5.1]{Gamelin} implies that every function holomorphic on a neighbourhood of $\overline{\mathbf{\Theta}}_n$ can be uniformly approximated on $\overline{\mathbf{\Theta}}_n$ by polynomials. Hence, by the holomorphic functional calculus for commuting operator tuples \cite[Theorem~9.9]{Vasilescu}, the inequality in {\rm(2)} extends from polynomials to all functions holomorphic on a neighbourhood of $\overline{\mathbf{\Theta}}_n$. Consequently, $\overline{\mathbf{\Theta}}_n$ is a spectral set for $\mathbf T$, and therefore $\mathbf T$ is a $\mathbf{\Theta}_n$-contraction.
\end{proof}	
	
\subsection{Non-uniqueness of the fundamental operator family}
We next discuss the uniqueness issue for the family of fundamental
operators associated with a $\mathbf{\Theta}_n$-contraction. We emphasize
that, for $p\geq 2$, the uniqueness of the family
\[
\bigl(A_0^{(i)},A_1^{(i)},\ldots,A_{p-1}^{(i)}\bigr),
\quad 1\leq i\leq n-1,
\]
does not follow from the fundamental equation
\eqref{Fundamental} alone.
Indeed, the structure of \eqref{Fundamental} shows that the operators
$A_l^{(i)}$ occur through defect-operator expressions of the form
\(
D_{T_n}A_l^{(i)}D_{T_n},
\)
multiplied on the left and right by powers of $T_n^*$ and $T_n$,
respectively. Consequently, the single operator identity
\eqref{Fundamental} does not, in general, separate the individual
coefficients
\[
A_0^{(i)},A_1^{(i)},\ldots,A_{p-1}^{(i)}.
\]
Consider the zero $n$-tuple
\(
\mathbf T=(0,\ldots,0)
\)
on a non-zero Hilbert space $\mathcal H$. Since
\(
T_n=0,
\)
we have
\(
D_{T_n}
=
(I-T_n^*T_n)^{1/2}
=
I_{\mathcal H}.
\)
Moreover,
\(
T_n^k=0,
T_n^{*k}=0, k\geq1.
\)
The left-hand side of \eqref{Fundamental} is
\(
T_i-T_{n-i}^*T_n^p=0.
\)
On the right-hand side, every term containing a positive power of
$T_n$ or $T_n^*$ vanishes. The only term which survives is the last
term. Consequently, \eqref{Fundamental} reduces to
\[
0=D_{T_n}A_0^{(i)}D_{T_n}.
\]
Since $D_{T_n}=I_{\mathcal H}$, this gives
\(
A_0^{(i)}=0.
\)
On the other hand, the operators
\(
A_1^{(i)},A_2^{(i)},\ldots,A_{p-1}^{(i)}
\)
do not occur in the resulting equation. Hence they can be chosen
arbitrarily without changing the validity of \eqref{Fundamental}.
For example, let
\(
B_1^{(i)},\ldots,B_{p-1}^{(i)}
\in\mathcal B(\mathcal H)
\)
be arbitrary operators. Then
\(
\bigl(
0,B_1^{(i)},\ldots,B_{p-1}^{(i)}
\bigr)
\)
satisfies the fundamental equation for every
$1\leq i\leq n-1$. Thus the fundamental equation admits infinitely many
different families of coefficients.
Therefore, for $p\geq2$, uniqueness of
\(
\bigl(A_0^{(i)},A_1^{(i)},\ldots,A_{p-1}^{(i)}\bigr)
\)
cannot be deduced from \eqref{Fundamental} alone.
The situation is different when $p=1$. In that case the fundamental
equation reduces to
\[
T_i-T_{n-i}^*T_n
=
D_{T_n}A_0^{(i)}D_{T_n}.
\tag{4.8}\label{eq:p-one-fundamental}
\]
Here there is only one unknown fundamental operator. Consequently, its
uniqueness follows from the density of
\(
\operatorname{Ran}D_{T_n}
\)
in the defect space
\(
\mathcal D_{T_n}
=
\overline{\operatorname{Ran}}D_{T_n}
\)  \cite[Theorem 4.4]{A. Pal}. We therefore omit the proof.

\begin{prop}\label{prop:p-one-uniqueness}
Suppose that $p=1$. Then, for each
$1\leq i\leq n-1$, the operator $A_0^{(i)}$ satisfying
\eqref{eq:p-one-fundamental} is uniquely determined on
$\mathcal D_{T_n}$. 
\end{prop}


The preceding discussion shows that, for $p\geq2$, a uniqueness theorem
would require additional information beyond \eqref{Fundamental}. 
\begin{lem}\label{lem:defect-uniqueness}
Let $D\in\mathcal B(\mathcal H)$ be a positive operator and put
\(
\mathcal D=\overline{\operatorname{Ran}}D.
\)
If $A,B\in\mathcal B(\mathcal D)$ satisfy
\(
DAD=DBD,
\)
then
\(
A=B.
\)
\end{lem}

\begin{proof}
Put $X=A-B$. The hypothesis gives
\(
DXD=0.
\)
Hence, for every $x,y\in\mathcal H$,
\[
\langle XDx,Dy\rangle
=
\langle DXD x,y\rangle
=
0.
\]
Since
\(
\operatorname{Ran}D
\)
is dense in $\mathcal D$, it follows that
\[
\langle Xu,v\rangle=0,
\quad
u,v\in\operatorname{Ran}D.
\]
By continuity, the same equality holds for all $u,v\in\mathcal D$.
Consequently $X=0$, and hence $A=B$.
\end{proof}
\begin{thm}\label{thm:conditional-uniqueness}
Let $\mathbf T$ be a $\mathbf{\Theta}_n$-contraction and let
\(
\mathcal D_{T_n}
=
\overline{\operatorname{Ran}}D_{T_n}.
\)
Suppose that, for some $r\geq0$, there exist operators
\[
R_0^{(i)},\ldots,R_r^{(i)}
\in\mathcal B(\mathcal H),
\quad
1\leq i\leq n-1,
\]
which depend only on $\mathbf T$, such that the fundamental
operators satisfy the coefficient-wise identities
\begin{equation}\label{eq:coefficient-separated}
D_{T_n}A_l^{(i)}D_{T_n}
=
R_l^{(i)},
\quad
0\leq l\leq r.
\end{equation}
Then the operators
\(
A_0^{(i)},\ldots,A_r^{(i)}
\)
are uniquely determined on $\mathcal D_{T_n}$.
\end{thm}

\begin{proof}
Suppose that
\(
A_l^{(i)}
~\text{and}~
\widetilde A_l^{(i)}
\)
both satisfy \eqref{eq:coefficient-separated}. Then
\[
D_{T_n}
\bigl(A_l^{(i)}-\widetilde A_l^{(i)}\bigr)
D_{T_n}=0
\]
for every $0\leq l\leq r$. By Lemma~\ref{lem:defect-uniqueness},
\(
A_l^{(i)}=\widetilde A_l^{(i)},
0\leq l\leq r.
\)
Hence the entire family is uniquely determined.
\end{proof}
\begin{rem}\label{rem:status-fundamental-operators}
The preceding discussion shows that uniqueness of the fundamental
operators is sensitive to the precise form of the fundamental
equations. In the case $p=1$, the fundamental equation directly
determines the fundamental operator through a single defect-space
identity, and uniqueness follows from the density of
$\operatorname{Ran}D_{T_n}$ in $\mathcal D_{T_n}$. For $p\geq2$, the
equation \eqref{Fundamental}, in its present form, does not
separate the coefficients sufficiently to yield uniqueness. In
particular, the zero tuple provides an immediate obstruction to such
a uniqueness assertion. Therefore, a uniqueness theorem for the
general $\mathbf{\Theta}_n$-case requires either a refined system of
fundamental equations or additional structural conditions that
separately determine the coefficient operators.
\end{rem}

\section{$\mathbf{\Theta}_n$-Unitaries and $\mathbf{\Theta}_n$-Isometries}\label{Unitary, Isometry}

In this section, we characterize $\mathbf{\Theta}_n$-unitaries and $\mathbf{\Theta}_n$-isometries. We also establish their relationships with $\Gamma_n$-unitaries (respectively, $\Gamma_n$-isometries), $\mathbf{\Theta}_{n+1}$-unitaries (respectively, $\mathbf{\Theta}_{n+1}$-isometries), and $\mathbb E$-unitaries (respectively, $\mathbb E$-isometries).

\subsection{Characterization of $\mathbf{\Theta}_n$-Unitaries}

\begin{thm}\label{Thm 1}
Let $\mathbf N=(N_1,\ldots,N_n)$ be a commuting $n$-tuple of bounded operators on a Hilbert space $\mathcal H$. Then the following statements are equivalent:

\begin{enumerate}

\item
$\mathbf N$ is a $\mathbf{\Theta}_n$-unitary.

\item
There exist commuting unitary operators $U_1,\ldots,U_n$ on $\mathcal H$ such that
\[
N_i=
\sum_{1\le k_1<\cdots<k_i\le n}
U_{k_1}\cdots U_{k_i},
\quad
1\le i\le n-1,\quad \text{and} \quad
N_n=(U_1\cdots U_n)^q.
\]

\item
$N_n$ is unitary, $
(\gamma_1N_1,\ldots,\gamma_{n-1}N_{n-1})$
is a $\Gamma_{n-1}$-contraction, and
$
N_i=N_{n-i}^*N_n^p,~
1\le i\le n-1.$

\item
$\mathbf N$ is a $\mathbf{\Theta}_n$-contraction and $N_n$ is unitary.

\item
$N_n$ is unitary, and there exists a $\Gamma_{n-1}$-unitary
$
(R_1,\ldots,R_{n-1})$
on $\mathcal H$ such that
$
R_1,\ldots,R_{n-1},N_n$
commute and
\[
N_i
=
R_i+R_{n-i}^*N_n^p,
\quad
1\le i\le n-1.
\]

\item
$N_n$ is unitary, $
(\gamma_1N_1,\ldots,\gamma_{n-1}N_{n-1})$
is a $\Gamma_{n-1}$-contraction, and for every $\alpha\in\mathbb T$,
\[
\Phi^{(i)}_1
(\alpha^iN_i,\alpha^{n-i}N_{n-i},\alpha^nN_n^p)=0\quad \text{and}\quad
\Phi^{(i)}_2
(\alpha^iN_i,\alpha^{n-i}N_{n-i},\alpha^nN_n^p)=0,
\]
for $1\le i\le n-1$.

\end{enumerate}
\end{thm}
\begin{proof}
The equivalence $(1)\Leftrightarrow(2)\Leftrightarrow(3)$ follows from \cite[Theorem~3.2]{Biswas 2}. We prove
\[
(1)\Rightarrow(4)\Rightarrow(3)\Rightarrow(6)\Rightarrow(3)
\quad\text{and}\quad
(2)\Rightarrow(5)\Rightarrow(2).
\]

\medskip

$(1)\Rightarrow(4)$:
Let $\mathbf N=(N_1,\ldots,N_n)$ be a $\mathbf{\Theta}_n$-unitary. Then $N_1,\ldots,N_n$ are commuting normal operators and
$
\sigma(\mathbf N)\subseteq b\mathbf{\Theta}_n.$
Hence, for every polynomial $p$ in $n$ variables, $p(\mathbf N)$ is normal. Therefore,
\begin{equation}\label{theta}
\begin{aligned}
\|p(\mathbf N)\|
&=r\bigl(p(\mathbf N)\bigr)\\
&=\sup\{|p(z)|:z\in\sigma(\mathbf N)\}\\
&\le
\sup\{|p(z)|:z\in b\mathbf{\Theta}_n\}\\
&\le
\sup\{|p(z)|:z\in\overline{\mathbf{\Theta}}_n\}\\
&=\|p\|_{\infty,\overline{\mathbf{\Theta}}_n},
\end{aligned}
\end{equation}
where the second equality follows from the spectral theorem for commuting normal operators.
Thus, $\overline{\mathbf{\Theta}}_n$ is a spectral set for $\mathbf N$, and hence $\mathbf N$ is a $\mathbf{\Theta}_n$-contraction. Moreover, by the equivalence of {\rm(1)} and {\rm(3)}, $N_n$ is unitary. Therefore, {\rm(4)} holds.

\medskip

$(4)\Rightarrow(3)$:
Suppose that $\mathbf N=(N_1,\ldots,N_n)$ is a $\mathbf{\Theta}_n$-contraction and that $N_n$ is unitary. By \cite[Lemma~2.7]{Biswas 2},
$
(\gamma_1N_1,\ldots,\gamma_{n-1}N_{n-1})$
is a $\Gamma_{n-1}$-contraction. Moreover, Lemma~\ref{Lem 3} implies that
$
\left(
\frac{N_i}{k(i)},
\frac{N_{n-i}}{k(i)},
N_n^p
\right)$
is an $\mathbb E$-contraction for  $1\le i\le n-1$. Since $N_n$ is unitary, so is $N_n^p$. Therefore, by \cite[Theorem~5.4]{Bhattacharyya}, the above triple is an $\mathbb E$-unitary. Consequently,
$
N_i=N_{n-i}^*N_n^p,~
1\le i\le n-1.$
Hence, {\rm(3)} holds.

\medskip

$(3)\Rightarrow(6)$ is immediate from the definitions of $\Phi^{(i)}_1$ and $\Phi^{(i)}_2.$  We prove the converse.

\medskip

$(6)\Rightarrow(3)$:
Assume that {\rm(6)} holds. Adding the identities
\[
\Phi^{(i)}_1
(\alpha^iN_i,\alpha^{n-i}N_{n-i},\alpha^nN_n^p)=0
\quad \text{and}\quad
\Phi^{(i)}_2
(\beta^iN_i,\beta^{n-i}N_{n-i},\beta^nN_n^p)=0,
\]
where $\alpha,\beta\in\mathbb T$, gives
\begin{equation}\label{E1-new}
\begin{aligned}
k(i)(I-N_n^{*p}N_n^p)
&-\operatorname{Re}\!\left(
\alpha^{\,n-i}(N_{n-i}-N_i^*N_n^p)
\right)-\operatorname{Re}\!\left(
\beta^{\,i}(N_i-N_{n-i}^*N_n^p)
\right)=0.
\end{aligned}
\end{equation}
Choosing $\alpha$ and $\beta$ so that
$
\alpha^{\,n-i}=\pm1,
\beta^i=\pm1,$
and comparing the resulting identities, we obtain
\[
\operatorname{Re}(N_i-N_{n-i}^*N_n^p)=0\quad \text{and}\quad
\operatorname{Re}(N_{n-i}-N_i^*N_n^p)=0.
\]
Replacing $\alpha^{\,n-i}$ and $\beta^i$ by $\pm i$ similarly yields
\[
\operatorname{Im}(N_i-N_{n-i}^*N_n^p)=0\quad \text{and}\quad
\operatorname{Im}(N_{n-i}-N_i^*N_n^p)=0.
\]
Hence, $
N_i=N_{n-i}^*N_n^p,
N_{n-i}=N_i^*N_n^p,
1\le i\le n-1.$
Since the operators commute,
\[
N_iN_i^*N_n^p
=
N_iN_{n-i}
=
N_{n-i}N_i
=
N_i^*N_n^pN_i
=
N_i^*N_iN_n^p.
\]
Multiplying both sides on the right by $N_n^{*p}$ gives
$
N_iN_i^*=N_i^*N_i,$
so each $N_i$ is normal.
Let $\mathcal A=C^*(I,N_1,\ldots,N_n)$. Since $\mathcal A$ is a commutative unital $C^*$-algebra, the Gelfand transform identifies each $N_i$ with the corresponding coordinate function on the maximal ideal space. The relations
\[
N_i=N_{n-i}^*N_n^p,
\quad
N_n^*N_n=I,
\]
therefore hold pointwise on the joint spectrum, and the joint spectrum of
$
(\gamma_1N_1,\ldots,\gamma_{n-1}N_{n-1})$
is contained in $\Gamma_{n-1}$. It follows from \cite[Theorem~2.5]{Biswas 2} that
$
\sigma(\mathbf N)\subseteq b\mathbf{\Theta}_n.$
Since the operators $N_1,\ldots,N_n$ are commuting and normal, $\mathbf N$ is a $\mathbf{\Theta}_n$-unitary. Hence, {\rm(3)} holds.

\medskip

$(2)\Rightarrow(5)$:
Suppose that {\rm(2)} holds. For $1\le i\le n-1$, define
\[
R_i=
\sum_{1\le k_1<\cdots<k_i\le n-1}
U_{k_1}\cdots U_{k_i}.
\]
Then $(R_1,\ldots,R_{n-1})$ is a $\Gamma_{n-1}$-unitary by \cite[Theorem~3.2]{Biswas 2}. Moreover,
\[
\begin{aligned}
N_i
&=
\sum_{1\le k_1<\cdots<k_i\le n}
U_{k_1}\cdots U_{k_i} \\
&=
\sum_{1\le k_1<\cdots<k_i\le n-1}
U_{k_1}\cdots U_{k_i}
+
\sum_{1\le k_1<\cdots<k_{i-1}\le n-1}
U_{k_1}\cdots U_{k_{i-1}}U_n \\
&=
R_i+R_{n-i}^*N_n^p,
\end{aligned}
\]
where we use the identities
\[
R_{n-i}^*=R_{i-1}(U_1\cdots U_{n-1})^*\quad \text{and}\quad
N_n^p=U_1\cdots U_n.
\]
Hence, {\rm(5)} holds.

\medskip

$(5)\Rightarrow(2)$:
Assume that {\rm(5)} holds. Since $(R_1,\ldots,R_{n-1})$ is a $\Gamma_{n-1}$-unitary, by \cite[Theorem~3.2]{Biswas 2} there exist commuting unitary operators
$
U_1,\ldots,U_{n-1}$
such that
\[
R_i=
\sum_{1\le k_1<\cdots<k_i\le n-1}
U_{k_1}\cdots U_{k_i},
\quad
1\le i\le n-1.
\]
Define
\[
U_n=R_{n-1}^*N_n^p.
\]
Since both $R_{n-1}$ and $N_n$ are unitary, $U_n$ is unitary. Furthermore, $U_n$ commutes with $U_1,\ldots,U_{n-1}$ because $R_{n-1}$ and $N_n$ commute.
Now,
\[
N_i
=
R_i+R_{n-i}^*N_n^p
=
R_i+R_{i-1}U_n,
\]
which is precisely the $i$th elementary symmetric polynomial in
$
U_1,\ldots,U_n.$
Finally,
\[
(U_1\cdots U_n)^q
=
(R_{n-1}U_n)^q
=
N_n.
\]
Therefore,
\[
N_i=
\sum_{1\le k_1<\cdots<k_i\le n}
U_{k_1}\cdots U_{k_i},
\quad
1\le i\le n-1,\quad \text{and}\quad
N_n=(U_1\cdots U_n)^q.
\]
Hence, {\rm(2)} holds. This completes the proof.
\end{proof}
We next establish the relationship between $\mathbf{\Theta}_n$-unitaries, $\Gamma_n$-unitaries, $\mathbf{\Theta}_{n+1}$-unitaries, and $\mathbb E$-unitaries.
\begin{thm}\label{Thm 2}
Let $\mathbf N=(N_1,\ldots,N_n)$ be a commuting $n$-tuple of bounded operators on a Hilbert space $\mathcal H$. The following statements are equivalent:

\begin{enumerate}
\item
$\mathbf N$ is a $\mathbf{\Theta}_n$-unitary.

\item
$N_n$ is normal and
$
(N_1,\ldots,N_{n-1},N_n^p)$
is a $\Gamma_n$-unitary.
\end{enumerate}
\end{thm}

\begin{proof}
$(1)\Rightarrow(2)$:
Suppose that $\mathbf N=(N_1,\ldots,N_n)$ is a $\mathbf{\Theta}_n$-unitary. By Theorem~\ref{Thm 1}, $\mathbf N$ is a $\mathbf{\Theta}_n$-contraction and $N_n$ is unitary. Hence $N_n^p$ is also unitary. Moreover, Lemma~\ref{Lem 1} implies that
$
(N_1,\ldots,N_{n-1},N_n^p)$
is a $\Gamma_n$-contraction. Therefore, by \cite[Theorem~5.1]{A. Pal},
$
(N_1,\ldots,N_{n-1},N_n^p)$
is a $\Gamma_n$-unitary.

\medskip

$(2)\Rightarrow(1)$:
Assume that $N_n$ is normal and
$
(N_1,\ldots,N_{n-1},N_n^p)$
is a $\Gamma_n$-unitary. By \cite[Theorem~4.2]{Biswas},
$N_n^p$ is unitary,
$
(\gamma_1N_1,\ldots,\gamma_{n-1}N_{n-1})$
is a $\Gamma_{n-1}$-contraction, and
\[
N_i=N_{n-i}^*N_n^p,
\quad
1\le i\le n-1.
\]
Since $N_n$ is normal and $N_n^p$ is unitary,
$
\sigma(N_n^p)\subseteq\mathbb T.$
By the spectral mapping theorem,
\[
\sigma(N_n^p)
=
\{\lambda^p:\lambda\in\sigma(N_n)\}.
\]
Hence $|\lambda|=1$ for every $\lambda\in\sigma(N_n)$, and therefore $N_n$ is unitary. The conclusion now follows from Theorem~\ref{Thm 1}. This completes the proof.
\end{proof}
The following theorem shows how every $\mathbf{\Theta}_n$-unitary gives rise to a family of $\mathbf{\Theta}_{n+1}$-unitaries.
\begin{thm}\label{Thm 3}
Let $\mathbf N=(N_1,\ldots,N_n)$ be a commuting $n$-tuple of bounded operators on a Hilbert space $\mathcal H$. If $\mathbf N$ is a $\mathbf{\Theta}_n$-unitary, then
\[
(\alpha^mI+N_1,\,
\alpha^mN_1+N_2,\,
\ldots,\,
\alpha^mN_{n-1}+N_n^p,\,
\alpha^{m/p}N_n)
\]
is a $\mathbf{\Theta}_{n+1}$-unitary for every $\alpha\in\mathbb T$.
\end{thm}

\begin{proof}
Let $\mathbf N=(N_1,\ldots,N_n)$ be a $\mathbf{\Theta}_n$-unitary. By Theorem~\ref{Thm 1}, $\mathbf N$ is a $\mathbf{\Theta}_n$-contraction and $N_n$ is unitary. Hence, by Lemma~\ref{Lem 2},
\[
(\alpha^mI+N_1,\,
\alpha^mN_1+N_2,\,
\ldots,\,
\alpha^mN_{n-1}+N_n^p,\,
\alpha^{m/p}N_n)
\]
is a $\mathbf{\Theta}_{n+1}$-contraction for every $\alpha\in\overline{\mathbb D}$.
Now let $\alpha\in\mathbb T$. Since $N_n$ is unitary and $|\alpha|=1$, the operator $\alpha^{m/p}N_n$ is also unitary. Therefore, Theorem~\ref{Thm 1} implies that
\[
(\alpha^mI+N_1,\,
\alpha^mN_1+N_2,\,
\ldots,\,
\alpha^mN_{n-1}+N_n^p,\,
\alpha^{m/p}N_n)
\]
is a $\mathbf{\Theta}_{n+1}$-unitary.
\end{proof}	
The following theorem provides an operator-theoretic analogue of \cite[Proposition~2.11]{Paul}.
\begin{thm}\label{Thm 4}
Let $\mathbf N=(N_1,\ldots,N_n)$ be a commuting $n$-tuple of bounded operators on a Hilbert space $\mathcal H$. The following statements are equivalent:

\begin{enumerate}

\item
$\mathbf N$ is a $\mathbf{\Theta}_n$-unitary.

\item
$N_n$ is normal,
$
\left(
\frac{N_i}{k(i)},
\frac{N_{n-i}}{k(i)},
N_n^p
\right)$
is a $\mathbb E$-unitary for $1\le i\le n-1$, and
$
(\gamma_1N_1,\ldots,\gamma_{n-1}N_{n-1})$
is a $\Gamma_{n-1}$-contraction.

\end{enumerate}
\end{thm}

\begin{proof}
$(1)\Rightarrow(2)$:
Suppose that $\mathbf N$ is a $\mathbf{\Theta}_n$-unitary. By Theorem~\ref{Thm 1},
$
(\gamma_1N_1,\ldots,\gamma_{n-1}N_{n-1})$
is a $\Gamma_{n-1}$-contraction and $N_n$ is unitary. In particular, $N_n$ is normal. Moreover, Theorem~\ref{Thm 1} also implies that $\mathbf N$ is a $\mathbf{\Theta}_n$-contraction. Hence, by Lemma~\ref{Lem 3},
$
\left(
\frac{N_i}{k(i)},
\frac{N_{n-i}}{k(i)},
N_n^p
\right)$
is an $\mathbb E$-contraction for  $1\le i\le n-1$. Since $N_n^p$ is unitary, \cite[Theorem~5.4]{Bhattacharyya} yields that
$
\left(
\frac{N_i}{k(i)},
\frac{N_{n-i}}{k(i)},
N_n^p
\right)$
is an $\mathbb E$-unitary for  $1\le i\le n-1$.

\medskip

$(2)\Rightarrow(1)$:
Assume that {\rm(2)} holds. Since
$
\left(
\frac{N_i}{k(i)},
\frac{N_{n-i}}{k(i)},
N_n^p
\right)$
is an $\mathbb E$-unitary, \cite[Theorem~5.4]{Bhattacharyya} implies that
\[
N_i=N_{n-i}^*N_n^p,
\quad
1\le i\le n-1,
\]
and that $N_n^p$ is unitary. As $N_n$ is normal, it follows from the spectral mapping theorem (equivalently, by the argument in the proof of Theorem~\ref{Thm 2}) that $N_n$ is unitary. Together with the assumption that
$
(\gamma_1N_1,\ldots,\gamma_{n-1}N_{n-1})$
is a $\Gamma_{n-1}$-contraction, Theorem~\ref{Thm 1} now implies that $\mathbf N$ is a $\mathbf{\Theta}_n$-unitary.
\end{proof}	
	
	\subsection{Characterization of $\mathbf{\Theta}_n$-Isometry}
	
	In the following theorem, we characterize $\mathbf{\Theta}_n$-isometry.
\begin{thm}\label{Thm 5}
Let $\mathbf V=(V_1,\ldots,V_n)$ be a commuting $n$-tuple of bounded operators on a Hilbert space $\mathcal H$. Then the following statements are equivalent:

\begin{enumerate}

\item
$\mathbf V$ is a $\mathbf{\Theta}_n$-isometry.

\item
$V_n$ is an isometry,
$
(\gamma_1V_1,\ldots,\gamma_{n-1}V_{n-1})$
is a $\Gamma_{n-1}$-contraction, and
$
V_i=V_{n-i}^*V_n^p,~
1\le i\le n-1.$

\item
There exists an orthogonal decomposition
$
\mathcal H=\mathcal H_1\oplus\mathcal H_2$
into common reducing subspaces for
$V_1,\ldots,V_n$
such that
$
(V_1|_{\mathcal H_1},\ldots,V_n|_{\mathcal H_1})$
is a $\mathbf{\Theta}_n$-unitary and
$
(V_1|_{\mathcal H_2},\ldots,V_n|_{\mathcal H_2})$
is a pure $\mathbf{\Theta}_n$-isometry.

\item
$\mathbf V$ is a $\mathbf{\Theta}_n$-contraction and
$V_n$ is an isometry.

\item
$V_n$ is a contraction,
$
(\gamma_1V_1,\ldots,\gamma_{n-1}V_{n-1})$
is a $\Gamma_{n-1}$-contraction, and for every
$\alpha\in\mathbb T$,
\[
\Phi^{(i)}_1
(\alpha^iV_i,\alpha^{n-i}V_{n-i},\alpha^nV_n^p)=0,
\quad \text{and}\quad
\Phi^{(i)}_2
(\alpha^iV_i,\alpha^{n-i}V_{n-i},\alpha^nV_n^p)=0,
\]
for 
$1\le i\le n-1$. 

Moreover, if
$
r(V_i)<k(i),~
1\le i\le n-1,$
then the above conditions are also equivalent to:

\item
$V_n$ is a contraction,
$
(\gamma_1V_1,\ldots,\gamma_{n-1}V_{n-1})$
is a $\Gamma_{n-1}$-contraction, and
\[
(k(i)\alpha^nV_n^p-V_{n-i})
(k(i)I-\alpha^iV_i)^{-1}
\quad \text{and}\quad
(k(i)\alpha^nV_n^p-V_i)
(k(i)I-\alpha^{\,n-i}V_{n-i})^{-1}
\]
are isometries for every
$\alpha\in\mathbb T$
and 
$1\le i\le n-1$.

\end{enumerate}
\end{thm}	
	
	\begin{proof}
		The equivalence of {\rm(1)}, {\rm(2)}, and {\rm(3)} is established in \cite[Theorem~3.10]{Biswas 2}. We prove the implications
\[
(1)\Rightarrow(4)\Rightarrow(5)\Rightarrow(2)
\quad\text{and}\quad
(5)\Rightarrow(6)\Rightarrow(5).
\]
		
		\medskip

$(1)\Rightarrow(4)$:
Suppose that $\mathbf V=(V_1,\ldots,V_n)$ is a $\mathbf{\Theta}_n$-isometry. Then, by definition, there exist a Hilbert space $\mathcal K\supseteq\mathcal H$ and a $\mathbf{\Theta}_n$-unitary
$
\mathbf N=(N_1,\ldots,N_n)$
on $\mathcal K$ such that
\[
N_i|_{\mathcal H}=V_i,
\quad
1\le i\le n.
\]
Let $p$ be a polynomial in $n$ complex variables. Since $\mathcal H$ is invariant for each $N_i$,
$
p(\mathbf V)
=
p(\mathbf N)|_{\mathcal H},$
and therefore
\[
\|p(\mathbf V)\|
=
\|p(\mathbf N)|_{\mathcal H}\|
\le
\|p(\mathbf N)\|
\le
\|p\|_{\infty,\overline{\mathbf{\Theta}}_n},
\]
where the last inequality follows from \eqref{theta}. Hence $\mathbf V$ is a $\mathbf{\Theta}_n$-contraction. Moreover, since $\mathbf N$ is a $\mathbf{\Theta}_n$-unitary, Theorem~\ref{Thm 2} implies that $N_n$ is unitary. Consequently,
$
V_n=N_n|_{\mathcal H}$
is an isometry. Thus $\mathbf V$ is a $\mathbf{\Theta}_n$-contraction with $V_n$ an isometry, proving {\rm(4)}.	

\medskip
$(4)\Rightarrow(5)$:
Assume that {\rm(4)} holds. Then $\mathbf V$ is a $\mathbf{\Theta}_n$-contraction and $V_n$ is an isometry. By \cite[Lemma~2.7]{Biswas 2},
$
(\gamma_1V_1,\ldots,\gamma_{n-1}V_{n-1})$
is a $\Gamma_{n-1}$-contraction. Moreover, Lemma~\ref{Lem 3} implies that
$
\left(
\frac{V_i}{k(i)},
\frac{V_{n-i}}{k(i)},
V_n^p
\right)$
is an $\mathbb E$-contraction for  $1\le i\le n-1$. Since $V_n$ is an isometry, $V_n^p$ is also an isometry. Hence, by \cite[Theorem~5.7]{Bhattacharyya},
\[
V_i=V_{n-i}^*V_n^p,
\qquad
1\le i\le n-1.
\]
Consequently,
\[
\Phi^{(i)}_1(\alpha^iV_i,\alpha^{n-i}V_{n-i},\alpha^nV_n^p)=0
\quad \text{and}\quad
\Phi^{(i)}_2(\alpha^iV_i,\alpha^{n-i}V_{n-i},\alpha^nV_n^p)=0
\]
for every $\alpha\in\mathbb T$ and  $1\le i\le n-1$. Therefore, {\rm(5)} holds.
	
\medskip

$(5)\Rightarrow(2)$:
The proof that
$
V_i=V_{n-i}^*V_n^p,
1\le i\le n-1,$
and that $V_n^p$ is an isometry is identical to that of $(6)\Rightarrow(3)$ in Theorem~\ref{Thm 1}. It therefore remains to show that $V_n$ is an isometry.
Since $V_n$ is a contraction and $V_n^p$ is an isometry, for every $x\in\mathcal H$,
\[
\|x\|
=
\|V_n^px\|
\le
\|V_n^{p-1}x\|
\le
\cdots
\le
\|V_nx\|
\le
\|x\|.
\]
Hence all the above inequalities are equalities. In particular,
\[
\|V_nx\|=\|x\|,
\qquad
x\in\mathcal H,
\]
which shows that $V_n$ is an isometry. Therefore, {\rm(2)} holds.		

\medskip

The equivalence of {\rm(5)} and {\rm(6)} follows immediately from the preceding arguments. This completes the proof.	\end{proof}
	
	\begin{thm}\label{Thm 6}
Let $\mathbf{V}=(V_1,\ldots,V_n)$ be a commuting $n$-tuple of bounded operators on a Hilbert space $\mathcal H$. Then the following assertions are equivalent:
\begin{enumerate}
\item $\mathbf{V}$ is a $\mathbf{\Theta}_n$-isometry.

\item $V_n$ is a contraction and
$
(V_1,\ldots,V_{n-1},V_n^p)$
is a $\Gamma_n$-isometry.
\end{enumerate}
\end{thm}
	
\begin{proof}

$(1)\Rightarrow(2):$ Let $\mathbf{V}=(V_1,\ldots,V_n)$ be a $\mathbf{\Theta}_n$-isometry. By Theorem~\ref{Thm 5}, $\mathbf{V}$ is a
$\mathbf{\Theta}_n$-contraction and $V_n$ is an isometry. Since every
isometry is a contraction, $V_n$ is a contraction. Moreover, by
Lemma~\ref{Lem 1},
$
(V_1,\ldots,V_{n-1},V_n^p)$
is a $\Gamma_n$-contraction. As $V_n$ is an isometry, so is $V_n^p$.
Applying Theorem~\ref{Thm 5} once again, we conclude that
$
(V_1,\ldots,V_{n-1},V_n^p)$
is a $\Gamma_n$-isometry. This proves $(1)\Rightarrow(2)$.

\medskip

$(2)\Rightarrow(1):$ Suppose that $V_n$ is a contraction and $
(V_1,\ldots,V_{n-1},V_n^p)$
is a $\Gamma_n$-isometry. By \cite[Theorem~4.12]{Biswas},
$V_n^p$ is an isometry,
$(\gamma_1V_1,\ldots,\gamma_{n-1}V_{n-1})$
is a $\Gamma_{n-1}$-contraction, and
\[
V_i=V_{n-i}^*V_n^p,\quad 1\le i\le n-1.
\]
Therefore, by \cite[Theorem~3.10]{Biswas 2}, it is enough to show that
$V_n$ is an isometry in order to conclude that
$\mathbf{V}$ is a $\mathbf{\Theta}_n$-isometry. The proof that $V_n$ is
an isometry is identical to the argument used in the implication
$(5)\Rightarrow(2)$ of Theorem~\ref{Thm 5}. Hence,
$\mathbf{V}$ is a $\mathbf{\Theta}_n$-isometry.
\end{proof}	

\begin{thm}\label{Thm 7}
Let $\mathbf{V}=(V_1,\ldots,V_n)$ be a commuting $n$-tuple of bounded operators on a Hilbert space $\mathcal H$. If $\mathbf{V}$ is a $\mathbf{\Theta}_n$-isometry, then
\[
(\alpha^mI+V_1,\,
\alpha^mV_1+V_2,\,
\ldots,\,
\alpha^mV_{n-1}+V_n^p,\,
\alpha^{m/p}V_n)
\]
is a $\mathbf{\Theta}_{n+1}$-isometry for every $\alpha\in\mathbb T$.
\end{thm}

\begin{proof}
Suppose that $\mathbf{V}$ is a $\mathbf{\Theta}_n$-isometry. By Theorem~\ref{Thm 5}, $\mathbf{V}$ is a $\mathbf{\Theta}_n$-contraction and $V_n$ is an isometry. Hence, by Lemma~\ref{Lem 2},
\[
(\alpha^mI+V_1,\,
\alpha^mV_1+V_2,\,
\ldots,\,
\alpha^mV_{n-1}+V_n^p,\,
\alpha^{m/p}V_n)
\]
is a $\mathbf{\Theta}_{n+1}$-contraction for every $\alpha\in\overline{\mathbb D}$. Since $|\alpha|=1$, the operator $\alpha^{m/p}V_n$ is also an isometry. Therefore, another application of Theorem~\ref{Thm 5} shows that
\[
(\alpha^mI+V_1,\,
\alpha^mV_1+V_2,\,
\ldots,\,
\alpha^mV_{n-1}+V_n^p,\,
\alpha^{m/p}V_n)
\]
is a $\mathbf{\Theta}_{n+1}$-isometry for every $\alpha\in\mathbb T$.
\end{proof}	
	
The following theorem finds a connection of $\mathbf{\Theta}_n$-isometry with $\mathbb{E}$-isometry.
	
	\begin{thm}\label{Thm 8}
		Let $\mathbf{V} = (V_1, \dots, V_n)$ be a commuting $n$-tuple of bounded operators on a Hilbert space $\mathcal{H}$. Then the following assertions are equivalent:		\begin{enumerate}
			\item $\mathbf{V}$ is a $\mathbf{\Theta}_n$-isometry.
			
\item $V_n$ is a contraction,
$
\left(
\frac{V_i}{k(i)},
\frac{V_{n-i}}{k(i)},
V_n^p
\right)$
is a $\mathbb E$-isometry for  $1\le i\le n-1$, and
$
(\gamma_1V_1,\ldots,\gamma_{n-1}V_{n-1})$
is a $\Gamma_{n-1}$-contraction.
		\end{enumerate}
	\end{thm}
	
\begin{proof}

$(1)\Rightarrow(2):$ Suppose that $\mathbf{V}$ is a $\mathbf{\Theta}_n$-isometry.
By Theorem~\ref{Thm 5},
$\mathbf{V}$ is a $\mathbf{\Theta}_n$-contraction,
$(\gamma_1V_1,\ldots,\gamma_{n-1}V_{n-1})$
is a $\Gamma_{n-1}$-contraction,
$V_n$ is an isometry, and
$
V_i=V_{n-i}^*V_n^p,~
1\le i\le n-1.$
Since $V_n$ is an isometry, so is $V_n^p$.
Therefore, Lemma~\ref{Lem 3} implies that
$
\left(
\frac{V_i}{k(i)},
\frac{V_{n-i}}{k(i)},
V_n^p
\right)$
is an $\mathbb E$-contraction.
As $V_n^p$ is an isometry, \cite[Theorem~5.4]{Bhattacharyya} now yields that
$
\left(
\frac{V_i}{k(i)},
\frac{V_{n-i}}{k(i)},
V_n^p
\right)$
is an $\mathbb E$-isometry.
Hence (2) follows.

\medskip

$(2)\Rightarrow(1):$
Suppose that $
\left(
\frac{V_i}{k(i)},
\frac{V_{n-i}}{k(i)},
V_n^p
\right)$
is an $\mathbb E$-isometry for  $1\le i\le n-1$, and that
$(\gamma_1V_1,\ldots,\gamma_{n-1}V_{n-1})$
is a $\Gamma_{n-1}$-contraction.
By \cite[Theorem~5.7]{Bhattacharyya},
$V_n^p$ is an isometry,
\[
V_i=V_{n-i}^*V_n^p \quad \text{and}\quad
\|V_i\|\le k(i),
\quad
1\le i\le n-1.
\]
The proof that $V_n$ is an isometry is identical to that of the implication
$(5)\Rightarrow(2)$ of Theorem~\ref{Thm 6}. Consequently,
Theorem~\ref{Thm 5} implies that
$\mathbf{V}$ is a $\mathbf{\Theta}_n$-isometry.
\end{proof}	
	
\section{Canonical Decomposition of $\mathbf{\Theta}_n$-Contractions}\label{Decomposition}

Every contraction $T$ on a Hilbert space $\mathcal H$ admits a canonical orthogonal decomposition
$
\mathcal H=\mathcal H_1\oplus\mathcal H_2,$
where $T|_{\mathcal H_1}$ is unitary and $T|_{\mathcal H_2}$ is a completely non-unitary (c.n.u.) contraction. This decomposition, commonly referred to as the \emph{canonical decomposition} (or \emph{orthogonal decomposition}), is a fundamental result in the theory of contractions; (see \cite[Theorem~1.12.1]{Bhat} and \cite[Theorem~3.2]{Nagy}). The main objective of this section is to establish an analogous decomposition for $\mathbf{\Theta}_n$-contractions.
We begin by recalling the following lemma from \cite[Lemma~2.9]{Roy}, which will play a key role in the proof of the decomposition theorem.

\begin{lem}\label{Numerical Radius}
The numerical radius of an operator $X$ is not greater than one if and only if
\[
\operatorname{Re}(zX)\le I,
\quad
z\in\mathbb T.
\]
\end{lem}
We also recall the following result from \cite[Proposition~1.3.2]{Bhatia}.

\begin{lem}\label{bhatia}
Let
$
M=
\begin{bmatrix}
A & X\\
X^* & B
\end{bmatrix},$
where $A\ge0$ and $B\ge0$. Then $M\ge0$ if and only if
\[
X=A^{1/2}KB^{1/2}
\]
for some contraction $K$.
\end{lem}
Let $\mathbf T=(T_1,\ldots,T_n)$ be a $\mathbf{\Theta}_n$-contraction on a Hilbert space $\mathcal H$. Since $T_n$ is a contraction, it admits the canonical decomposition described above. The following theorem shows that this decomposition simultaneously reduces every operator $T_i$, $1\le i\le n-1$. Consequently, the canonical decomposition of a $\mathbf{\Theta}_n$-contraction is completely determined by the canonical decomposition of its last component $T_n$.

\begin{thm}\label{Decomposition of T}
Let $\mathbf{T}=(T_1,\ldots,T_n)$ be a $\mathbf{\Theta}_n$-contraction on a Hilbert space $\mathcal H$. Let $\mathcal H_1$ be the maximal reducing subspace of $T_n$ on which $T_n$ is unitary, and set
$
\mathcal H_2=\mathcal H\ominus\mathcal H_1.$
Then the following assertions hold.
\begin{enumerate}
\item The subspaces $\mathcal H_1$ and $\mathcal H_2$ reduce each of the operators $T_1,\ldots,T_{n-1}$.

\item
$
(T_1|_{\mathcal H_1},\ldots,T_{n-1}|_{\mathcal H_1},T_n|_{\mathcal H_1})$
is a $\mathbf{\Theta}_n$-unitary.

\item
$
(T_1|_{\mathcal H_2},\ldots,T_{n-1}|_{\mathcal H_2},T_n|_{\mathcal H_2})$
is a c.n.u. $\mathbf{\Theta}_n$-contraction.

\item Either $\mathcal H_1$ or $\mathcal H_2$ may be the trivial subspace $\{0\}$.
\end{enumerate}
\end{thm}
\begin{proof}
Suppose that $\mathbf{T}=(T_1,\ldots,T_n)$ is a $\mathbf{\Theta}_n$-contraction. If $T_n$ is a c.n.u. contraction, then $\mathcal H_1=\{0\}$. On the other hand, if $T_n$ is unitary, then $\mathcal H_1=\mathcal H$, and consequently $\mathcal H_2=\{0\}$. In either case, the conclusion follows immediately. Therefore, it suffices to consider the case where $T_n$ is neither unitary nor completely non-unitary.
With respect to the orthogonal decomposition
$
\mathcal H=\mathcal H_1\oplus\mathcal H_2,$
write
\begin{equation}\label{T_i, T_n}
\begin{aligned}
T_i&=
\begin{bmatrix}
T^{(i)}_{11} & T^{(i)}_{12}\\
T^{(i)}_{21} & T^{(i)}_{22}
\end{bmatrix},
\quad 1\le i\le n-1,\quad
T_n&=
\begin{bmatrix}
T^{(n)}_1 & 0\\
0 & T^{(n)}_2
\end{bmatrix},
\end{aligned}
\end{equation}
where
$
T^{(n)}_1=T_n|_{\mathcal H_1}$
is unitary and
$
T^{(n)}_2=T_n|_{\mathcal H_2}$
is a completely non-unitary contraction.
Since $T^{(n)}_2$ is completely non-unitary, if $x\in\mathcal H_2$ satisfies
\[
\|(T^{(n)}_2)^{*k}x\|
=
\|x\|
=
\|(T^{(n)}_2)^kx\|,
\qquad k\in\mathbb N,
\]
then necessarily $x=0$.
Since $T_iT_n=T_nT_i$ for $1\le i\le n-1$, we obtain
\begin{equation}\label{De 1}
\begin{aligned}
T^{(i)}_{11}T^{(n)}_1
&=
T^{(n)}_1T^{(i)}_{11},
&
T^{(i)}_{12}T^{(n)}_2
&=
T^{(n)}_1T^{(i)}_{12},
\\
T^{(i)}_{21}T^{(n)}_1
&=
T^{(n)}_2T^{(i)}_{21},
&
T^{(i)}_{22}T^{(n)}_2
&=
T^{(n)}_2T^{(i)}_{22}.
\end{aligned}
\end{equation}
Since $\mathbf{T}$ is a $\mathbf{\Theta}_n$-contraction, Theorem~\ref{Prop 1} implies that
\[
\Phi^{(i)}_1(\alpha^iT_i,\alpha^{n-i}T_{n-i},\alpha^nT_n^p)\ge0
\quad\text{and}\quad
\Phi^{(i)}_2(\beta^iT_i,\beta^{n-i}T_{n-i},\beta^nT_n^p)\ge0
\]
for all $\alpha,\beta\in\mathbb T$ and $1\le i\le n-1$. Adding these two inequalities, we obtain
\begin{equation}\label{De 2}
\begin{aligned}
k(i)(I-T_n^{*p}T_n^p)
-\operatorname{Re}\!\left(\alpha^{n-i}(T_{n-i}-T_i^*T_n^p)\right)
-\operatorname{Re}\!\left(\beta^i(T_i-T_{n-i}^*T_n^p)\right)
\ge0.
\end{aligned}
\end{equation}
For $1\le i\le n-1$, define
\[
N_i:=
k(i)(I-T_n^{*p}T_n^p)
-\operatorname{Re}\!\left(\alpha^{n-i}(T_{n-i}-T_i^*T_n^p)\right)
-\operatorname{Re}\!\left(\beta^i(T_i-T_{n-i}^*T_n^p)\right).
\]
Using the block matrix representations in \eqref{T_i, T_n}, equation \eqref{De 2} can be written as
\begin{equation}\label{De 3}
\begin{aligned}
0\le N_i
&=
\begin{bmatrix}
0 & 0\\
0 & I-T_2^{(n)*p}T_2^{(n)p}
\end{bmatrix}
-
\operatorname{Re}\alpha^{n-i}
\begin{bmatrix}
T_{11}^{(n-i)}-{T_{11}^{(i)*}}T_1^{(n)p}
&
T_{12}^{(n-i)}-{T_{21}^{(i)*}}T_2^{(n)p}
\\
T_{21}^{(n-i)}-{T_{12}^{(i)*}}T_1^{(n)p}
&
T_{22}^{(n-i)}-{T_{22}^{(i)*}}T_2^{(n)p}
\end{bmatrix}
\\
&\hspace{4.8cm}
-
\operatorname{Re}\beta^{i}
\begin{bmatrix}
T_{11}^{(i)}-T_{11}^{(n-i)*}T_1^{(n)p}
&
T_{12}^{(i)}-T_{21}^{(n-i)*}T_2^{(n)p}
\\
T_{21}^{(i)}-T_{12}^{(n-i)*}T_1^{(n)p}
&
T_{22}^{(i)}-T_{22}^{(n-i)*}T_2^{(n)p}
\end{bmatrix}.
\end{aligned}
\end{equation}
Since $N_i$ is positive semidefinite, it admits the block matrix representation
\[
N_i=
\begin{bmatrix}
A_i & X_i\\
X_i^* & B_i
\end{bmatrix},
\quad
1\le i\le n-1.
\]
Comparing the $(1,1)$ entries in \eqref{De 3}, we obtain
\begin{equation}\label{De 4}
\begin{aligned}
\operatorname{Re}\!\left(\alpha^{n-i}
(T^{(n-i)}_{11}-{T^{(i)*}_{11}}T^{(n)p}_1)\right)
+\operatorname{Re}\!\left(\beta^i
(T^{(i)}_{11}-T^{(n-i)*}_{11}T^{(n)p}_1)\right)
\le0,
\end{aligned}
\end{equation}
for all $\alpha,\beta\in\mathbb T$ and $1\le i\le n-1$.
Setting $\alpha^{\,n-i}=1$ and $\alpha^{\,n-i}=-1$ in \eqref{De 4}, respectively, we obtain
\begin{equation}\label{De 5}
\begin{aligned}
\operatorname{Re}\!\left(
T^{(n-i)}_{11}-{T^{(i)*}_{11}}T^{(n)p}_1\right)
+\operatorname{Re}\!\left(\beta^i
(T^{(i)}_{11}-T^{(n-i)*}_{11}T^{(n)p}_1)\right)
\le0,
\end{aligned}
\end{equation}
and
\begin{equation}\label{De 6}
\begin{aligned}
-\operatorname{Re}\!\left(
T^{(n-i)}_{11}-{T^{(i)*}_{11}}T^{(n)p}_1\right)
+\operatorname{Re}\!\left(\beta^i
(T^{(i)}_{11}-T^{(n-i)*}_{11}T^{(n)p}_1)\right)
\le0,
\end{aligned}
\end{equation}
for all $\beta\in\mathbb T$ and $1\le i\le n-1$.
Adding \eqref{De 5} and \eqref{De 6} yields
\[
\operatorname{Re}\!\left(\beta^i
(T^{(i)}_{11}-T^{(n-i)*}_{11}T^{(n)p}_1)\right)
\le0,
\quad
1\le i\le n-1,
\]
for all $\beta\in\mathbb T$. It implies that
\[
T^{(i)}_{11}
=
T^{(n-i)*}_{11}T^{(n)p}_1,
\quad
1\le i\le n-1.
\]
Similarly, by setting $\beta^i=1$ and $\beta^i=-1$ in \eqref{De 4}, we obtain
\[
T^{(n-i)}_{11}
=
T^{(i)*}_{11}T^{(n)p}_1,
\quad
1\le i\le n-1.
\]
Since
$
T^{(i)}_{11}=T^{(n-i)*}_{11}T^{(n)p}_1,~
1\le i\le n-1,$
it follows that $A_i=0$. Consequently, Lemma~\ref{bhatia} implies that
$X_i=0$. Therefore, \eqref{De 3} yields
\begin{equation}\label{De 7}
\begin{aligned}
\operatorname{Re}\!\left(\alpha^{n-i}
(T^{(n-i)}_{12}-T^{(i)*}_{21}T^{(n)p}_2)\right)
+\operatorname{Re}\!\left(\beta^i
(T^{(i)}_{12}-T^{(n-i)*}_{21}T^{(n)p}_2)\right)
=0
\end{aligned}
\end{equation}
and
\begin{equation}\label{De 8}
\begin{aligned}
\operatorname{Re}\!\left(\alpha^{n-i}
(T^{(n-i)}_{21}-T^{(i)*}_{12}T^{(n)p}_1)\right)
+\operatorname{Re}\!\left(\beta^i
(T^{(i)}_{21}-T^{(n-i)*}_{12}T^{(n)p}_1)\right)
=0,
\end{aligned}
\end{equation}
for all $\alpha,\beta\in\mathbb{T}$ and $1\le i\le n-1$. Proceeding exactly as in the proof of \eqref{De 5}--\eqref{De 6}, equations \eqref{De 7} and \eqref{De 8} yield
\begin{equation}\label{De 9}
\begin{aligned}
T^{(n-i)}_{12} = T^{(i)*}_{21}T^{(n)p}_2,
T^{(i)}_{12} = T^{(n-i)*}_{21}T^{(n)p}_2,
T^{(n-i)}_{21} = T^{(i)*}_{12}T^{(n)p}_1,
T^{(i)}_{21} = T^{(n-i)*}_{12}T^{(n)p}_1, \quad 1\le i\le n-1.
\end{aligned}
\end{equation}
Using \eqref{De 1} and \eqref{De 9}, we obtain
\begin{equation}\label{De 10}
\begin{aligned}
T^{(n-i)*}_{12}(T^{(n)p}_1)^2
&= T^{(i)}_{21}T^{(n)p}_1\\
&= T^{(n)p}_2T^{(i)}_{21}\\
&= T^{(n)p}_2T^{(n-i)*}_{12}T^{(n)p}_1.
\end{aligned}
\end{equation}
Since $T^{(n)}_1$ is unitary, multiplying both sides of \eqref{De 10} on the right by $(T^{(n)p}_1)^*$ yields
\begin{equation}\label{De 11}
T^{(n-i)*}_{12}T^{(n)p}_1
=
T^{(n)p}_2T^{(n-i)*}_{12},
\quad
1\le i\le n-1.
\end{equation}
Repeated applications of \eqref{De 1} and \eqref{De 11} yield, for every $k \ge 1$,
\begin{equation}\label{De 12}
(T^{(n)p}_1)^kT^{(i)}_{12}
=
T^{(i)}_{12}(T^{(n)p}_2)^k
\quad\text{and}\quad
(T^{(n)*p}_1)^kT^{(n-i)}_{12}
=
T^{(n-i)}_{12}(T^{(n)*p}_2)^k,
\quad
1\le i\le n-1.
\end{equation}
Combining \eqref{De 12} with the unitarity of $T^{(n)}_1$, we obtain
\begin{equation}\label{De 13}
\begin{aligned}
T^{(i)}_{12}(T^{(n)p}_2)^k(T^{(n)*p}_2)^k
&=
(T^{(n)p}_1)^kT^{(i)}_{12}(T^{(n)*p}_2)^k\\
&=
(T^{(n)p}_1)^k(T^{(n)*p}_1)^kT^{(i)}_{12}\\
&=
T^{(i)}_{12}.
\end{aligned}
\end{equation}
An analogous argument shows that
\begin{equation}\label{De 14}
T^{(i)}_{12}(T^{(n)*p}_2)^k(T^{(n)p}_2)^k
=
T^{(i)}_{12},
\quad
1\le i\le n-1.
\end{equation}
For every $x\in\mathcal H_1$ and $k\ge1$, it follows from \eqref{De 13} and \eqref{De 14} that
\[
\|(T^{(n)*p}_2)^kT^{(i)*}_{12}x\|
=
\|T^{(i)*}_{12}x\|
=
\|(T^{(n)p}_2)^k(T^{(n)*p}_2)^kT^{(i)*}_{12}x\|,
\quad
1\le i\le n-1.
\]
Since $T^{(n)p}_2$ is completely non-unitary, it follows that
$T^{(i)*}_{12}x=0$ for every $x\in\mathcal H_1$. Hence
$T^{(i)}_{12}=0$ for $1\le i\le n-1$. Similarly, one can show that
$T^{(i)}_{21}=0$ for $1\le i\le n-1$. Therefore,
\[
T_i=
\begin{bmatrix}
T^{(i)}_{11} & 0\\
0 & T^{(i)}_{22}
\end{bmatrix},
\quad
1\le i\le n-1,
\]
with respect to the decomposition
$\mathcal H=\mathcal H_1\oplus\mathcal H_2$.
Consequently, both $\mathcal H_1$ and $\mathcal H_2$ reduce each of the operators
$T_1,\ldots,T_{n-1}$, proving $(1).$

Since $\mathbf T$ is a $\mathbf{\Theta}_n$-contraction and
$\mathcal H_1$ is a joint invariant subspace for
$T_1,\ldots,T_n$, the restriction
$(T^{(1)}_{11},\ldots,T^{(n-1)}_{11},T^{(n)}_1)$
is also a $\mathbf{\Theta}_n$-contraction. As $T^{(n)}_1$ is unitary,
Theorem~\ref{Thm 1} implies that
$(T^{(1)}_{11},\ldots,T^{(n-1)}_{11},T^{(n)}_1)$
is a $\mathbf{\Theta}_n$-unitary, proving $(2).$

Similarly,
$(T^{(1)}_{22},\ldots,T^{(n-1)}_{22},T^{(n)}_2)$
is a $\mathbf{\Theta}_n$-contraction on $\mathcal H_2$.
Since $T^{(n)}_2$ is completely non-unitary, it follows that
$
(T_1|_{\mathcal H_2},\ldots,
T_{n-1}|_{\mathcal H_2},
T_n|_{\mathcal H_2})$
is a completely non-unitary
$\mathbf{\Theta}_n$-contraction, proving $(3).$
This completes the proof.
\end{proof}	
	
\begin{rem}\label{Rem 1}
Observe that a \emph{pure isometry} (or \emph{shift}) is a special case of a completely non-unitary contraction with norm one. Therefore, the Wold decomposition for $\mathbf{\Theta}_n$-isometries (Theorem~\ref{Thm 5}) is a special case of the canonical decomposition of $\mathbf{\Theta}_n$-contractions established in Theorem~\ref{Decomposition of  T}.
\end{rem}	
\section{Dilation Theory of $\mathbf{\Theta}_n$-Contractions}\label{Dilation Theory}

In this section, we study the dilation theory of $\mathbf{\Theta}_n$-contractions. We begin by introducing the notion of a $\mathbf{\Theta}_n$-isometric dilation.

\begin{defn}\label{Dilation}
Let $\mathbf{T}=(T_1,\dots,T_n)$ be a $\mathbf{\Theta}_n$-contraction on a Hilbert space $\mathcal{H}$. An $n$-tuple of commuting bounded operators $\mathbf{V}=(V_1,\dots,V_n)$ acting on a Hilbert space $\mathcal{K}$ containing $\mathcal{H}$ as a subspace is called a \emph{$\mathbf{\Theta}_n$-isometric dilation} of $\mathbf{T}$ if the following conditions hold:
\begin{enumerate}
\item $\mathbf{V}$ is a $\mathbf{\Theta}_n$-isometry.

\item $V_i^*|_{\mathcal H}=T_i^*$, \quad $1\le i\le n$.
\end{enumerate}
Furthermore, $\mathbf{V}$ is called a \emph{minimal $\mathbf{\Theta}_n$-isometric dilation} of $\mathbf{T}$ if
\[
\mathcal K
=
\overline{\operatorname{span}}
\{V_n^kh:\, h\in\mathcal H,\; k\in\mathbb N\cup\{0\}\}.
\]
In this case, $\mathcal K$ is called the \emph{minimal $\mathbf{\Theta}_n$-isometric dilation space} of $\mathbf{T}$.
\end{defn}	
Since every $\mathbf{\Theta}_n$-isometry is the restriction of a $\mathbf{\Theta}_n$-unitary to a joint invariant subspace, Definition~\ref{Dilation} immediately implies that every $\mathbf{\Theta}_n$-contraction admitting a $\mathbf{\Theta}_n$-isometric dilation also admits a $\mathbf{\Theta}_n$-unitary dilation. 

\subsection{Necessary and Sufficient Conditions for the Existence of $\mathbf{\Theta}_n$-Isometric Dilations}
In this subsection, we first explicitly construct a minimal $\mathbf{\Theta}_n$-isometric dilation of a $\mathbf{\Theta}_n$-contraction under suitable assumptions. 
\begin{thm}\label{Conditional Dilation}
Let $\mathbf{T}=(T_1,\dots,T_n)$ be a $\mathbf{\Theta}_n$-contraction on a Hilbert space $\mathcal H$, and let
$A_k^{(i)}$, $0\le k\le p$, $1\le i\le n-1$, be bounded operators on $\mathcal D_{T_n}$. Let
\[
\mathcal K
=
\mathcal H
\oplus
\mathcal D_{T_n}
\oplus
\mathcal D_{T_n}
\oplus\cdots
=
\mathcal H\oplus\ell^2(\mathcal D_{T_n}).
\]
Define an $n$-tuple of operators
$\mathbf V=(V_1,\dots,V_n)$
on $\mathcal K$ by	
\begin{equation}\label{V_i}
			\begin{aligned}
				V_i &=
				\left[
				\begin{smallmatrix}
					T_i &\hspace{0.5cm} 0 &\hspace{0.5cm} 0 &\hspace{0.5cm} 0 &\hspace{0.5cm} \dots \hspace{0.5cm}\\
					\displaystyle\sum_{l=0}^{p-1} A^{(i)}_{p-l}D_{T_n}T^{p-1-l}_n &\hspace{0.5cm} A^{(i)}_0 &\hspace{0.5cm} 0 &\hspace{0.5cm} 0 &\hspace{0.5cm} \dots \hspace{0.5cm}\\
					\displaystyle\sum_{l=0}^{p-2} A^{(i)}_{p-l}D_{T_n}T^{p-2-l}_n &\hspace{0.5cm} A^{(i)}_1 &\hspace{0.5cm} A^{(i)}_0 &\hspace{0.5cm} 0 &\hspace{0.5cm} \dots \hspace{0.5cm}\\
					\vdots & \hspace{0.5cm} \vdots & \hspace{0.5cm} \vdots & \hspace{0.5cm} \vdots & \hspace{0.5cm} \ddots \hspace{0.5cm}\\
					\displaystyle\sum_{l=0}^{1} A^{(i)}_{p-l}D_{T_n}T^{1-l}_n &\hspace{0.5cm} A^{(i)}_{p-2} &\hspace{0.5cm} A^{(i)}_{p-3} &\hspace{0.5cm} A^{(i)}_{p-4} &\hspace{0.5cm} \dots \hspace{0.5cm}\\
					A^{(i)}_pD_{T_n} &\hspace{0.5cm} A^{(i)}_{p-1} &\hspace{0.5cm} A^{(i)}_{p-2} &\hspace{0.5cm} A^{(i)}_{p-3} &\hspace{0.5cm} \dots \hspace{0.5cm}\\
					0 &\hspace{0.5cm} A^{(i)}_p &\hspace{0.5cm} A^{(i)}_{p-1} &\hspace{0.5cm} A^{(i)}_{p-2} &\hspace{0.5cm} \dots \hspace{0.5cm}\\
					\vdots &\hspace{0.5cm} \vdots &\hspace{0.5cm} \vdots &\hspace{0.5cm} \vdots &\hspace{0.5cm} \ddots \hspace{0.5cm}
				\end{smallmatrix}\right]
			\end{aligned}
		\end{equation}
		for $1 \le i \le n-1$ and
		\begin{equation}\label{V_n}
			\begin{aligned}
				V_n &=
				\left[
				\begin{smallmatrix}
					T_n &\hspace{0.5cm} 0 &\hspace{0.5cm} 0 &\hspace{0.5cm} 0 &\hspace{0.5cm} \dots \hspace{0.5cm}\\
					D_{T_n} &\hspace{0.5cm} 0 &\hspace{0.5cm} 0 &\hspace{0.5cm} 0 &\hspace{0.5cm} \dots \hspace{0.5cm}\\
					0 &\hspace{0.5cm} I &\hspace{0.5cm} 0 &\hspace{0.5cm} 0 &\hspace{0.5cm} \dots \hspace{0.5cm}\\
					0 &\hspace{0.5cm} 0 &\hspace{0.5cm} I &\hspace{0.5cm} 0 &\hspace{0.5cm} \dots \hspace{0.5cm}\\
					0 &\hspace{0.5cm} 0 &\hspace{0.5cm} 0 &\hspace{0.5cm} I &\hspace{0.5cm} \dots \hspace{0.5cm}\\
					\vdots &\hspace{0.5cm} \vdots &\hspace{0.5cm} \vdots &\hspace{0.5cm} \vdots &\hspace{0.5cm} \ddots \hspace{0.5cm}
				\end{smallmatrix}\right],
			\end{aligned}
		\end{equation}
		where $I$ denotes the identity operator on $\mathcal{D}_{T_n}$. Then $\mathbf{V}=(V_1,\dots,V_n)$ is a $\mathbf{\Theta}_n$-isometric dilation of $\mathbf{T}$ provided that $(\gamma_1V_1,\dots,\gamma_{n-1}V_{n-1})$ is a $\Gamma_{n-1}$-contraction and the following conditions are satisfied:
\begin{enumerate}
\item[(1)] $A^{(i)*}_l=A^{(n-i)}_{p-l}$, $1\le i\le n-1$, $0\le l\le p$, and $A^{(i)}_l=0$ for $l>p$.

\item[(2)] $\displaystyle\sum_{l=0}^{k}[A^{(i)}_l,A^{(j)}_{k-l}]=0$, $0\le k\le 2p$ and $1\le i,j\le n-1$, where $[A,B]=AB-BA$.

\item[(3)] $D_{T_n}T_i=\displaystyle\sum_{k=0}^{p}A^{(i)}_kD_{T_n}T_n^k$, $1\le i\le n-1$.

\item[(4)] \begin{equation}\begin{aligned}T_i-T_{n-i}^*T_n^p
&=\sum_{l=0}^{p-1}T_n^{*\,p-1-l}D_{T_n}A^{(i)}_lD_{T_n}T_n^{p-1}
+\sum_{l=0}^{p-2}T_n^{*\,p-2-l}D_{T_n}A^{(i)}_lD_{T_n}T_n^{p-2}
+\cdots \\\nonumber
&+\sum_{l=0}^{1}T_n^{*\,1-l}D_{T_n}A^{(i)}_lD_{T_n}T_n
+D_{T_n}A^{(i)}_0D_{T_n}.
\end{aligned}
\end{equation}
\end{enumerate}

Moreover, if $\mathbf{V}=(V_1,\dots,V_n)$ is a $\mathbf{\Theta}_n$-isometric dilation of $\mathbf{T}$ on $\mathcal{K}$, then the operators $V_1,\dots,V_n$ are of the form given in \eqref{V_i} and \eqref{V_n}, and satisfy conditions \emph{(1)}--\emph{(4)}.
	\end{thm}
	
	\begin{proof}
Suppose there exist operators $A^{(i)}_k$, $0\le k\le p$, $1\le i\le n-1$, acting on $\mathcal{D}_{T_n}$ such that conditions $(1)-(4)$ are satisfied. It is immediate from the definitions of $V_1,\ldots,V_n$ that
$
V_i^*|_{\mathcal H}=T_i^*, ~1\le i\le n.$
Thus, it remains to show that $\mathbf V$ is a $\mathbf{\Theta}_n$-isometry. By Theorem~\ref{Thm 5}, it suffices to verify the following:
\begin{enumerate}
\item[(i)] $V_n$ is an isometry.

\item[(ii)] $V_iV_n=V_nV_i$, \quad $1\le i\le n-1$.

\item[(iii)] $V_iV_j=V_jV_i$, \quad $1\le i,j\le n-1$.

\item[(iv)] $V_i=V_{n-i}^*V_n^p$, \quad $1\le i\le n-1$.
\end{enumerate}

\noindent\textbf{Step 1.}
The construction of $V_n$ is due to Sch\"affer. Although Sch\"affer's construction yields a unitary dilation of a contraction, here we only require its isometric part. Hence, $V_n$ is an isometry.

		\medskip
		
		\noindent\textbf{Step 2.}
We next verify that $
V_iV_n=V_nV_i, 1\le i\le n-1.$
To this end, we first compute the products $V_iV_n$ and $V_nV_i$. Observe that
		\begin{equation}\label{V_iV_n}
			\begin{aligned}
				V_iV_n &=
				\left[
				\begin{smallmatrix}
					T_iT_n &\hspace{0.5cm} 0 &\hspace{0.5cm} 0 &\hspace{0.5cm} 0 &\hspace{0.5cm} \dots &\hspace{0.5cm} 0 &\hspace{0.5cm} \dots \hspace{0.5cm}\\
					\displaystyle\sum_{l=0}^{p-1} A^{(i)}_{p-l}D_{T_n}T^{p-l}_n + A^{(i)}_0D_{T_n} &\hspace{0.5cm} 0 &\hspace{0.5cm} 0 &\hspace{0.5cm} 0 &\hspace{0.5cm} \dots &\hspace{0.5cm} 0 &\hspace{0.5cm} \dots \hspace{0.5cm}\\
					\displaystyle\sum_{l=0}^{p-2} A^{(i)}_{p-l}D_{T_n}T^{p-1-l}_n + A^{(i)}_1D_{T_n} &\hspace{0.5cm} A^{(i)}_0 &\hspace{0.5cm} 0 &\hspace{0.5cm} 0 &\hspace{0.5cm} \dots &\hspace{0.5cm} 0 &\hspace{0.5cm} \dots \hspace{0.5cm}\\
					\vdots & \hspace{0.5cm} \vdots & \hspace{0.5cm} \vdots & \hspace{0.5cm} \vdots & \hspace{0.5cm} \ddots &\hspace{0.5cm} \vdots &\hspace{0.5cm} \ddots \hspace{0.5cm}\\
					\displaystyle\sum_{l=0}^{1} A^{(i)}_{p-l}D_{T_n}T^{2-l}_n + A^{(i)}_{p-2}D_{T_n} &\hspace{0.5cm} A^{(i)}_{p-3} &\hspace{0.5cm} A^{(i)}_{p-4} &\hspace{0.5cm} A^{(i)}_{p-5} &\hspace{0.5cm} \dots &\hspace{0.5cm} 0 &\hspace{0.5cm} \dots \hspace{0.5cm}\\
					A^{(i)}_pD_{T_n}T_n + A^{(i)}_{p-1}D_{T_n} &\hspace{0.5cm} A^{(i)}_{p-2} &\hspace{0.5cm} A^{(i)}_{p-3} &\hspace{0.5cm} A^{(i)}_{p-4} &\hspace{0.5cm} \dots &\hspace{0.5cm} A^{(i)}_0 &\hspace{0.5cm} \dots \hspace{0.5cm}\\
					A^{(i)}_pD_{T_n} &\hspace{0.5cm} A^{(i)}_{p-1} &\hspace{0.5cm} A^{(i)}_{p-2} &\hspace{0.5cm} A^{(i)}_{p-3} &\hspace{0.5cm} \dots &\hspace{0.5cm} A^{(i)}_1 &\hspace{0.5cm} \dots \hspace{0.5cm}\\
					0 &\hspace{0.5cm} A^{(i)}_p &\hspace{0.5cm} A^{(i)}_{p-1} &\hspace{0.5cm} A^{(i)}_{p-2} &\hspace{0.5cm} \dots &\hspace{0.5cm} A^{(i)}_2 &\hspace{0.5cm} \dots \hspace{0.5cm}\\
					\vdots &\hspace{0.5cm} \vdots &\hspace{0.5cm} \vdots &\hspace{0.5cm} \vdots &\hspace{0.5cm} \ddots &\hspace{0.5cm} \vdots &\hspace{0.5cm} \ddots \hspace{0.5cm}
				\end{smallmatrix}\right]
			\end{aligned}
		\end{equation}
		and
		\begin{equation}\label{V_nV_i}
			\begin{aligned}
				V_nV_i &=
				\left[
				\begin{smallmatrix}
					T_nT_i &\hspace{0.5cm} 0 &\hspace{0.5cm} 0 &\hspace{0.5cm} 0 &\hspace{0.5cm} \dots &\hspace{0.5cm} 0 &\hspace{0.5cm} \dots \hspace{0.5cm}\\
					D_{T_n}T_i &\hspace{0.5cm} 0 &\hspace{0.5cm} 0 &\hspace{0.5cm} 0 &\hspace{0.5cm} \dots &\hspace{0.5cm} 0 &\hspace{0.5cm} \dots \hspace{0.5cm}\\
					\displaystyle\sum_{l=0}^{p-1} A^{(i)}_{p-l}D_{T_n}T^{p-1-l}_n &\hspace{0.5cm} A^{(i)}_0 &\hspace{0.5cm} 0 &\hspace{0.5cm} 0 &\hspace{0.5cm} \dots &\hspace{0.5cm} 0 &\hspace{0.5cm} \dots \hspace{0.5cm}\\
					\vdots & \hspace{0.5cm} \vdots & \hspace{0.5cm} \vdots & \hspace{0.5cm} \vdots & \hspace{0.5cm} \ddots &\hspace{0.5cm} \vdots &\hspace{0.5cm} \ddots \hspace{0.5cm}\\
					\displaystyle\sum_{l=0}^{2} A^{(i)}_{2-l}D_{T_n}T^{2-l}_n &\hspace{0.5cm} A^{(i)}_{p-3} &\hspace{0.5cm} A^{(i)}_{p-4} &\hspace{0.5cm} A^{(i)}_{p-5} &\hspace{0.5cm} \dots &\hspace{0.5cm} 0 &\hspace{0.5cm} \dots \hspace{0.5cm}\\
					\displaystyle\sum_{l=0}^{1} A^{(i)}_{p-l}D_{T_n}T^{1-l}_n &\hspace{0.5cm} A^{(i)}_{p-2} &\hspace{0.5cm} A^{(i)}_{p-3} &\hspace{0.5cm} A^{(i)}_{p-4} &\hspace{0.5cm} \dots &\hspace{0.5cm} A^{(i)}_0 &\hspace{0.5cm} \dots \hspace{0.5cm}\\
					A^{(i)}_pD_{T_n} &\hspace{0.5cm} A^{(i)}_{p-1} &\hspace{0.5cm} A^{(i)}_{p-2} &\hspace{0.5cm} A^{(i)}_{p-3} &\hspace{0.5cm} \dots &\hspace{0.5cm} A^{(i)}_1 &\hspace{0.5cm} \dots \hspace{0.5cm}\\
					0 &\hspace{0.5cm} A^{(i)}_p &\hspace{0.5cm} A^{(i)}_{p-1} &\hspace{0.5cm} A^{(i)}_{p-2} &\hspace{0.5cm} \dots &\hspace{0.5cm} A^{(i)}_2 &\hspace{0.5cm} \dots \hspace{0.5cm}\\
					\vdots &\hspace{0.5cm} \vdots &\hspace{0.5cm} \vdots &\hspace{0.5cm} \vdots &\hspace{0.5cm} \ddots &\hspace{0.5cm} \vdots &\hspace{0.5cm} \ddots \hspace{0.5cm}
				\end{smallmatrix}\right]
			\end{aligned}
		\end{equation}
		By condition $(3)$ the $(q,1)$-entries of the matrices in \eqref{V_iV_n} and \eqref{V_nV_i} coincide for $2\le q\le p+2$. The remaining entries are clearly identical. Therefore,
$
V_iV_n=V_nV_i, ~1\le i\le n-1.$
		
		\medskip
\noindent\textbf{Step 3.}
We next verify that
$
V_iV_j=V_jV_i, 1\le i,j\le n-1.$
Observe that
\begin{equation}\label{V_iV_j}
\begin{aligned}
V_iV_j=
\left[
\begin{matrix}
T_iT_j & 0\\
\tilde{C} & \tilde{D}
\end{matrix}
\right],
\end{aligned}
\end{equation}
where
\begin{equation}\label{tilde C, tilde D}
\begin{aligned}
\tilde{C}
&=
\left[
\begin{smallmatrix}
C_{21}\\
\vdots\\
C_{p1}\\
\vdots
\end{smallmatrix}
\right]
:\mathcal{H}\to\ell^2(\mathcal{D}_{T_n}),\quad
\tilde{D}
&:\ell^2(\mathcal{D}_{T_n})\to\ell^2(\mathcal{D}_{T_n}).
\end{aligned}
\end{equation}		
More precisely, the entries of $\tilde C$ and $\tilde D$ are given by
\begin{equation*}
			\begin{aligned}
				&C_{21} = \displaystyle\sum_{l=0}^{p-1} A^{(i)}_{p-l}D_{T_n}T^{p-1-l}_nT_j + \displaystyle\sum_{l=0}^{p-1} A^{(i)}_0A^{(j)}_{p-l}D_{T_n}T^{p-1-l}_n,\\
				&C_{31} = \displaystyle\sum_{l=0}^{p-2} A^{(i)}_{p-l}D_{T_n}T^{p-2-l}_nT_j + 
				\displaystyle\sum_{l=0}^{p-1} A^{(i)}_1A^{(j)}_{p-l}D_{T_n}T^{p-1-l}_n + \displaystyle\sum_{l=0}^{p-2} A^{(i)}_0A^{(j)}_{p-l}D_{T_n}T^{p-2-l}_n,\\
				&\hspace{3.5cm} \vdots \hspace{3.5cm} \vdots \hspace{3.5cm} \vdots\\
				&C_{p1} = \displaystyle\sum_{l=0}^{1} A^{(i)}_{p-l}D_{T_n}T^{1-l}_nT_j + \displaystyle\sum_{l=0}^{p-1} A^{(i)}_{p-2}A^{(j)}_{p-l}D_{T_n}T^{p-1-l}_n + \dots + \displaystyle\sum_{l=0}^{1} A^{(i)}_0A^{(j)}_{p-l}D_{T_n}T^{1-l}_n,\\
				&C_{(p+1)1} = A^{(i)}_pD_{T_n}T_j + \displaystyle\sum_{l=0}^{p-1} A^{(i)}_{p-1}A^{(j)}_{p-l}D_{T_n}T^{p-1-l}_n + \dots + \displaystyle\sum_{l=0}^{1} A^{(i)}_1A^{(j)}_{p-l}D_{T_n}T^{1-l}_n + A^{(i)}_0A^{(j)}_pD_{T_n},\\
				&C_{(p+2)1} = \displaystyle\sum_{l=0}^{p-1} A^{(i)}_pA^{(j)}_{p-l}D_{T_n}T^{p-1-l}_n + \displaystyle\sum_{l=0}^{p-2} A^{(i)}_{p-1}A^{(j)}_{p-l}D_{T_n}T^{p-2-l}_n + \dots + A^{(i)}_1A^{(j)}_pD_{T_n},\\
				&\hspace{3.5cm} \vdots \hspace{3.5cm} \vdots \hspace{3.5cm} \vdots\\
				&C_{(2p)1} = \displaystyle\sum_{l=0}^{1} A^{(i)}_pA^{(j)}_{p-l}D_{T_n}T^{1-l}_n + A^{(i)}_{p-1}A^{(j)}_pD_{T_n},\\
				&C_{(2p+1)1} = A^{(i)}_pA^{(j)}_pD_{T_n}, \quad \text{and} \quad C_{l1} = 0, \quad \text{for} \quad l \ge 2p+2.
			\end{aligned}
		\end{equation*}
respectively, and
		\begin{equation*}
			\begin{aligned}
				\tilde{D} & =
				\left[
				\begin{smallmatrix}
					A^{(i)}_0A^{(j)}_0 &\hspace{0.5cm} 0 &\hspace{0.5cm} \dots &\hspace{0.5cm} 0 &\hspace{0.5cm} \dots \hspace{0.5cm}\\
					\displaystyle\sum_{l=0}^{1} A^{(i)}_lA^{(j)}_{1-l} &\hspace{0.5cm} A^{(i)}_0A^{(j)}_0 &\hspace{0.5cm} \dots &\hspace{0.5cm} 0 &\hspace{0.5cm} \dots \hspace{0.5cm}\\
					\vdots &\hspace{0.5cm} \vdots &\hspace{0.5cm} \ddots &\hspace{0.5cm} \vdots &\hspace{0.5cm} \ddots \hspace{0.5cm}\\
					\displaystyle\sum_{l=0}^{p-2} A^{(i)}_lA^{(j)}_{p-2-l} &\hspace{0.5cm} \displaystyle\sum_{l=0}^{p-3} A^{(i)}_lA^{(j)}_{p-3-l} &\hspace{0.5cm} \dots &\hspace{0.5cm} A^{(i)}_0A^{(j)}_0 &\hspace{0.5cm} \dots \hspace{0.5cm}\\
					\displaystyle\sum_{l=0}^{p-1} A^{(i)}_lA^{(j)}_{p-1-l} &\hspace{0.5cm} \displaystyle\sum_{l=0}^{p-2} A^{(i)}_lA^{(j)}_{p-2-l} &\hspace{0.5cm} \dots &\hspace{0.5cm} \displaystyle\sum_{l=0}^{1} A^{(i)}_lA^{(j)}_{1-l} &\hspace{0.5cm} \dots \hspace{0.5cm}\\
					\displaystyle\sum_{l=0}^{p} A^{(i)}_lA^{(j)}_{p-l} &\hspace{0.5cm} \displaystyle\sum_{l=0}^{p-1} A^{(i)}_lA^{(j)}_{p-1-l} &\hspace{0.5cm} \dots & \displaystyle\sum_{l=0}^{2} A^{(i)}_lA^{(j)}_{2-l} &\hspace{0.5cm} \dots \hspace{0.5cm}\\
					\vdots &\hspace{0.5cm} \vdots &\hspace{0.5cm} \ddots &\hspace{0.5cm} \vdots &\hspace{0.5cm} \ddots \hspace{0.5cm}\\
					\displaystyle\sum_{l=0}^{2p-1} A^{(i)}_lA^{(j)}_{2p-1-l} &\hspace{0.5cm} \displaystyle\sum_{l=0}^{2p-2} A^{(i)}_lA^{(j)}_{2p-2-l} &\hspace{0.5cm} \dots &\hspace{0.5cm} \displaystyle\sum_{l=0}^{p+1} A^{(i)}_lA^{(j)}_{p+1-l} &\hspace{0.5cm} \dots \hspace{0.5cm}\\
					A^{(i)}_pA^{(j)}_p &\hspace{0.5cm} \displaystyle\sum_{l=0}^{2p-1} A^{(i)}_lA^{(j)}_{2p-1-l} &\hspace{0.5cm} \dots &\hspace{0.5cm} \displaystyle\sum_{l=0}^{p+2} A^{(i)}_lA^{(j)}_{p+2-l} &\hspace{0.5cm} \dots \hspace{0.5cm}\\
					0 &\hspace{0.5cm} A^{(i)}_pA^{(j)}_p &\hspace{0.5cm} \dots &\hspace{0.5cm} \displaystyle\sum_{l=0}^{p+2} A^{(i)}_lA^{(j)}_{p+2-l} &\hspace{0.5cm} \dots \hspace{0.5cm}\\
					\vdots &\hspace{0.5cm} \vdots &\hspace{0.5cm} \ddots &\hspace{0.5cm} \vdots &\hspace{0.5cm} \ddots \hspace{0.5cm}
				\end{smallmatrix}\right].
			\end{aligned}
		\end{equation*}
Similarly,
$
V_jV_i=
\left[
\begin{smallmatrix}
T_jT_i & 0\\
\tilde{C}' & \tilde{D}'
\end{smallmatrix}
\right],$
where $\tilde{C}'$ and $\tilde{D}'$ are obtained from $\tilde{C}$ and $\tilde{D}$, respectively, by interchanging the indices $i$ and $j$ in \eqref{V_iV_j} and \eqref{tilde C, tilde D}.
By condition \emph{(2)}, the corresponding entries of $\tilde{D}$ and $\tilde{D}'$ are equal. Therefore, to prove
$
V_iV_j=V_jV_i,1\le i,j\le n-1,$
it suffices to verify the following identities:
\medskip
\begin{enumerate}
			\item[$\mathbf{(1)}$] \small $\displaystyle\sum_{l=0}^{p-1} A^{(i)}_{p-l}D_{T_n}T^{p-1-l}_nT_j + \displaystyle\sum_{l=0}^{p-1} A^{(i)}_0A^{(j)}_{p-l}D_{T_n}T^{p-1-l}_n = \displaystyle\sum_{l=0}^{p-1} A^{(j)}_{p-l}D_{T_n}T^{p-1-l}_nT_i + \displaystyle\sum_{l=0}^{p-1} A^{(j)}_0A^{(i)}_{p-l}D_{T_n}T^{p-1-l}_n$,
			
			\item[$\mathbf{(2)}$] $\displaystyle\sum_{l=0}^{p-2} A^{(i)}_{p-l}D_{T_n}T^{p-2-l}_nT_j + 
			\displaystyle\sum_{l=0}^{p-1} A^{(i)}_1A^{(j)}_{p-l}D_{T_n}T^{p-1-l}_n + \displaystyle\sum_{l=0}^{p-2} A^{(i)}_0A^{(j)}_{p-l}D_{T_n}T^{p-2-l}_n \\= \displaystyle\sum_{l=0}^{p-2} A^{(j)}_{p-l}D_{T_n}T^{p-2-l}_nT_i + 
			\displaystyle\sum_{l=0}^{p-1} A^{(j)}_1A^{(i)}_{p-l}D_{T_n}T^{p-1-l}_n + \displaystyle\sum_{l=0}^{p-2} A^{(j)}_0A^{(i)}_{p-l}D_{T_n}T^{p-2-l}_n$,
			
			\item[] $\hspace{3cm} \vdots \hspace{3cm} \vdots \hspace{3cm} \vdots  \hspace{3cm}$
			
			\item[$\mathbf{(p-1)}$] $\displaystyle\sum_{l=0}^{1} A^{(i)}_{p-l}D_{T_n}T^{1-l}_nT_j + \displaystyle\sum_{l=0}^{p-1} A^{(i)}_{p-2}A^{(j)}_{p-l}D_{T_n}T^{p-1-l}_n + \dots + \displaystyle\sum_{l=0}^{1} A^{(i)}_0A^{(j)}_{p-l}D_{T_n}T^{1-l}_n \\= \displaystyle\sum_{l=0}^{1} A^{(j)}_{p-l}D_{T_n}T^{1-l}_nT_i + \displaystyle\sum_{l=0}^{p-1} A^{(j)}_{p-2}A^{(i)}_{p-l}D_{T_n}T^{p-1-l}_n + \dots + \displaystyle\sum_{l=0}^{1} A^{(j)}_0A^{(i)}_{p-l}D_{T_n}T^{1-l}_n$,
			
			\item[$\mathbf{(p)}$] $A^{(i)}_pD_{T_n}T_j + \displaystyle\sum_{l=0}^{p-1} A^{(i)}_{p-1}A^{(j)}_{p-l}D_{T_n}T^{p-1-l}_n + \dots + \displaystyle\sum_{l=0}^{1} A^{(i)}_1A^{(j)}_{p-l}D_{T_n}T^{1-l}_n + A^{(i)}_0A^{(j)}_pD_{T_n} \\= A^{(j)}_pD_{T_n}T_i + \displaystyle\sum_{l=0}^{p-1} A^{(j)}_{p-1}A^{(i)}_{p-l}D_{T_n}T^{p-1-l}_n + \dots + \displaystyle\sum_{l=0}^{1} A^{(j)}_1A^{(i)}_{p-l}D_{T_n}T^{1-l}_n + A^{(j)}_0A^{(i)}_pD_{T_n}$,
			
			\item[$\mathbf{(p+1)}$] $\displaystyle\sum_{l=0}^{p-1} A^{(i)}_pA^{(j)}_{p-l}D_{T_n}T^{p-1-l}_n + \displaystyle\sum_{l=0}^{p-2} A^{(i)}_{p-1}A^{(j)}_{p-l}D_{T_n}T^{p-2-l}_n + \dots + A^{(i)}_1A^{(j)}_pD_{T_n}\\ = \displaystyle\sum_{l=0}^{p-1} A^{(j)}_pA^{(i)}_{p-l}D_{T_n}T^{p-1-l}_n + \displaystyle\sum_{l=0}^{p-2} A^{(j)}_{p-1}A^{(i)}_{p-l}D_{T_n}T^{p-2-l}_n + \dots + A^{(j)}_1A^{(i)}_pD_{T_n}$,
			
			\item[$\mathbf{(p+2)}$] $\displaystyle\sum_{l=0}^{1} A^{(i)}_pA^{(j)}_{p-l}D_{T_n}T^{1-l}_n + A^{(i)}_{p-1}A^{(j)}_pD_{T_n} = \displaystyle\sum_{l=0}^{1} A^{(j)}_pA^{(i)}_{p-l}D_{T_n}T^{1-l}_n + A^{(j)}_{p-1}A^{()}_pD_{T_n}$,
			
			\item[$\mathbf{(p+3)}$] $A^{(i)}_pA^{(j)}_pD_{T_n} = A^{(j)}_pA^{(i)}_pD_{T_n}$.
		\end{enumerate}
		
	\medskip
The proof of \textbf{(p+3)} follows immediately from condition $(2)$. We prove only \textbf{(1)}, since the proofs of \textbf{(2)}--\textbf{(p+2)} are entirely analogous. Observe that proving \textbf{(1)} is equivalent to proving
		\begin{equation}\label{D 1}
			\begin{aligned}
				&\displaystyle\sum_{l=0}^{p-1} A^{(i)}_{p-l}D_{T_n}T^{p-1-l}_nT_j - \displaystyle\sum_{l=0}^{p-1} A^{(j)}_{p-l}D_{T_n}T^{p-1-l}_nT_i + \displaystyle\sum_{l=0}^{p-1} A^{(i)}_0A^{(j)}_{p-l}D_{T_n}T^{p-1-l}_n - \displaystyle\sum_{l=0}^{p-1} A^{(j)}_0A^{(i)}_{p-l}D_{T_n}T^{p-1-l}_n = 0.
			\end{aligned}
		\end{equation}
		Since $T_iT_n=T_nT_i$ for $1\le i\le n-1$, condition $(3)$ implies that the left-hand side of \eqref{D 1} can be written as
		\begin{equation}\label{D 2}
			\begin{aligned}
				&A^{(i)}_p\left(\displaystyle \sum_{k=0}^{p} A^{(j)}_kD_{T_n}T^k_n\right)T^{p-1}_n + \dots + A^{(i)}_2\left(\displaystyle \sum_{k=0}^{p} A^{(j)}_kD_{T_n}T^k_n\right)T_n + A^{(i)}_1\left(\displaystyle \sum_{k=0}^{p} A^{(j)}_kD_{T_n}T^k_n\right)\\
				&\hspace{0.5cm} - A^{(j)}_p\left(\displaystyle \sum_{k=0}^{p} A^{(i)}_kD_{T_n}T^k_n\right)T^{p-1}_n
				- \dots - A^{(j)}_2\left(\displaystyle \sum_{k=0}^{p} A^{(i)}_kD_{T_n}T^k_n\right)T_n - A^{(j)}_1\left(\displaystyle \sum_{k=0}^{p} A^{(i)}_kD_{T_n}T^k_n\right)\\
				&\hspace{0.5cm} + \displaystyle\sum_{l=0}^{p-1} A^{(i)}_0A^{(j)}_{p-l}D_{T_n}T^{p-1-l}_n - \displaystyle\sum_{l=0}^{p-1} A^{(j)}_0A^{(i)}_{p-l}D_{T_n}T^{p-1-l}_n.
			\end{aligned}
		\end{equation}Expanding the products and collecting like terms, we obtain
		\begin{equation}\label{D 3}
			\begin{aligned}
				&\left(\sum_{l=0}^{2p} [A^{(i)}_l, A^{(j)}_{p-l}]\right)D_{T_n}T^{2p-1}_n + \left(\sum_{l=0}^{2p-1} [A^{(i)}_l, A^{(j)}_{p-l}]\right)D_{T_n}T^{2p-2}_n + \dots + \left(\sum_{l=0}^{p+1} [A^{(i)}_l, A^{(j)}_{p-l}]\right)D_{T_n}T^{p}_n\\
				&\hspace{0.5cm} + \left(\sum_{l=0}^{p} [A^{(i)}_l, A^{(j)}_{p-l}]\right)D_{T_n}T^{p-1}_n + \dots + \left(\sum_{l=0}^{2} [A^{(i)}_l, A^{(j)}_{p-l}]\right)D_{T_n}T_n + \left(\sum_{l=0}^{1} [A^{(i)}_l, A^{(j)}_{p-l}]\right)D_{T_n}.
			\end{aligned}
		\end{equation}
		Therefore, condition $(2)$ implies that every coefficient in \eqref{D 3} vanishes. Hence, \eqref{D 3} is the zero operator. Consequently,
$
V_iV_j=V_jV_i,~ 1\le i,j\le n-1.$
This completes the proof of Step~3.		
		\medskip
		
		\noindent\textbf{Step 4.}
We next verify that
$
V_i=V_{n-i}^*V_n^p, ~1\le i\le n-1.$
A straightforward computation shows that
		\begin{equation}\label{V^p_n}
			\begin{aligned}
				V^p_n &=
				\left[
				\begin{smallmatrix}
				T^p_n &\hspace{0.5cm} 0 &\hspace{0.5cm} 0 &\hspace{0.5cm} 0 &\hspace{0.5cm} \dots \hspace{0.5cm}\\
				D_{T_n}T^{p-1}_n &\hspace{0.5cm} 0 &\hspace{0.5cm} 0 &\hspace{0.5cm} 0 &\hspace{0.5cm} \dots \hspace{0.5cm}\\
				\vdots &\hspace{0.5cm} \vdots &\hspace{0.5cm} \vdots &\hspace{0.5cm} \vdots &\hspace{0.5cm} \ddots \hspace{0.5cm}\\
				D_{T_n}T_n &\hspace{0.5cm} 0 &\hspace{0.5cm} 0 &\hspace{0.5cm} 0 &\hspace{0.5cm} \dots \hspace{0.5cm}\\
				D_{T_n} &\hspace{0.5cm} 0 &\hspace{0.5cm} 0 &\hspace{0.5cm} 0 &\hspace{0.5cm} \dots \hspace{0.5cm}\\
				0 &\hspace{0.5cm} I &\hspace{0.5cm} 0 &\hspace{0.5cm} 0 &\hspace{0.5cm} \dots \hspace{0.5cm}\\
				0 &\hspace{0.5cm} 0 &\hspace{0.5cm} I &\hspace{0.5cm} 0 &\hspace{0.5cm} \dots \hspace{0.5cm}\\
				0 &\hspace{0.5cm} 0 &\hspace{0.5cm} 0 &\hspace{0.5cm} I &\hspace{0.5cm} \dots \hspace{0.5cm}\\
				\vdots &\hspace{0.5cm} \vdots &\hspace{0.5cm} \vdots &\hspace{0.5cm} \vdots &\hspace{0.5cm} \ddots \hspace{0.5cm}
				\end{smallmatrix}\right],
			\end{aligned}
		\end{equation}
		which admits the block representation
$
V_n^p=
\begin{bmatrix}
T_n^p & 0\\
C & D
\end{bmatrix},$
where
		\begin{equation}\label{C}
			\begin{aligned}
				C &=
				\left[
				\begin{smallmatrix}
					D_{T_n}T^{p-1}_n\\
					\vdots\\
					D_{T_n}T_n\\
					D_{T_n}\\
					0\\
					0\\
					\vdots
				\end{smallmatrix}\right] : \mathcal{H} \to \ell^2(\mathcal{D}_{T_n}),
			\end{aligned}
		\end{equation}
while
		\begin{equation}\label{D}
			\begin{aligned}
				D &=
				\left[
				\begin{smallmatrix}
					0 &\hspace{0.5cm} 0 &\hspace{0.5cm} 0 &\hspace{0.5cm} \dots\\
					\vdots & \hspace{0.5cm} \vdots & \hspace{0.5cm} \vdots & \hspace{0.5cm} \ddots\\
					0 &\hspace{0.5cm} 0 &\hspace{0.5cm} 0 &\hspace{0.5cm} \dots\\
					0 &\hspace{0.5cm} 0 &\hspace{0.5cm} 0 &\hspace{0.5cm} \dots\\
					I &\hspace{0.5cm} 0 &\hspace{0.5cm} 0 &\hspace{0.5cm} \dots\\
					0 &\hspace{0.5cm} I &\hspace{0.5cm} 0 &\hspace{0.5cm} \dots\\
					0 &\hspace{0.5cm} 0 &\hspace{0.5cm} I &\hspace{0.5cm} \dots\\
					\vdots &\hspace{0.5cm} \vdots &\hspace{0.5cm} \vdots &\hspace{0.5cm} \ddots
				\end{smallmatrix}\right] : \ell^2(\mathcal{D}_{T_n}) \to \ell^2(\mathcal{D}_{T_n}).
			\end{aligned}
		\end{equation}
		A direct computation yields
		\begin{equation}\label{D 4}
			\begin{aligned}
				V^*_{n-i}V^p_n &=
				\left[
				\begin{smallmatrix}
					A &\hspace{0.5cm} 0 &\hspace{0.5cm} 0 &\hspace{0.5cm} 0 &\hspace{0.5cm} \dots \hspace{0.5cm}\\
					\displaystyle\sum_{l=0}^{p-1} A^{(i)}_{p-l}D_{T_n}T^{p-1-l}_n &\hspace{0.5cm} A^{(i)}_0 &\hspace{0.5cm} 0 &\hspace{0.5cm} 0 &\hspace{0.5cm} \dots \hspace{0.5cm}\\
					\displaystyle\sum_{l=0}^{p-2} A^{(i)}_{p-l}D_{T_n}T^{p-2-l}_n &\hspace{0.5cm} A^{(i)}_1 &\hspace{0.5cm} A^{(i)}_0 &\hspace{0.5cm} 0 &\hspace{0.5cm} \dots \hspace{0.5cm}\\
					\vdots & \hspace{0.5cm} \vdots & \hspace{0.5cm} \vdots & \hspace{0.5cm} \vdots & \hspace{0.5cm} \ddots \hspace{0.5cm}\\
					\displaystyle\sum_{l=0}^{1} A^{(i)}_{p-l}D_{T_n}T^{1-l}_n &\hspace{0.5cm} A^{(i)}_{p-2} &\hspace{0.5cm} A^{(i)}_{p-3} &\hspace{0.5cm} A^{(i)}_{p-4} &\hspace{0.5cm} \dots \hspace{0.5cm}\\
					A^{(i)}_pD_{T_n} &\hspace{0.5cm} A^{(i)}_{p-1} &\hspace{0.5cm} A^{(i)}_{p-2} &\hspace{0.5cm} A^{(i)}_{p-3} &\hspace{0.5cm} \dots \hspace{0.5cm}\\
					0 &\hspace{0.5cm} A^{(i)}_p &\hspace{0.5cm} A^{(i)}_{p-1} &\hspace{0.5cm} A^{(i)}_{p-2} &\hspace{0.5cm} \dots \hspace{0.5cm}\\
					\vdots &\hspace{0.5cm} \vdots &\hspace{0.5cm} \vdots &\hspace{0.5cm} \vdots &\hspace{0.5cm} \ddots \hspace{0.5cm}
				\end{smallmatrix}\right],
			\end{aligned}
		\end{equation}
		where
		\begin{equation*}
			\begin{aligned}
				A &= T^*_{n-i}T^p_n + \sum_{l=0}^{p-1} T^{*p-1-l}_nD_{T_n}A^{(i)}_lD_{T_n}T^{p-1}_n + \sum_{l=0}^{p-2} T^{*p-2-l}_nD_{T_n}A^{(i)}_lD_{T_n}T^{p-2}_n\\
				&\hspace{2cm}+ \dots + \sum_{l=0}^{1} T^{*1-l}_nD_{T_n}A^{(i)}_lD_{T_n}T_n + D_{T_n}A^{(i)}_0D_{T_n}.
			\end{aligned}
		\end{equation*}
		By condition $(4)$, we obtain $A=T_i$. Hence, $
V_i=V_{n-i}^*V_n^p, ~1\le i\le n-1.$
Since $(\gamma_1V_1,\ldots,\gamma_{n-1}V_{n-1})$ is a $\Gamma_{n-1}$-contraction, it follows from $(1)-(4)$ that $\mathbf{V}$ is a $\mathbf{\Theta}_n$-contraction. This completes the proof of the sufficiency.		
		\medskip

Suppose that $\mathbf{V}=(V_1,\ldots,V_n)$ is a minimal $\mathbf{\Theta}_n$-isometric dilation of $\mathbf{T}$. Since the minimal isometric dilation of a contraction is unique up to unitary equivalence, we may assume, without loss of generality, that $V_n$ is the Sch\"affer dilation of $T_n$. Consequently, $V_n$ has the block operator matrix representation
\begin{equation*}
\begin{aligned}
V_n=
\left[
\begin{smallmatrix}
T_n &\hspace{0.5cm} 0 &\hspace{0.5cm} 0 &\hspace{0.5cm} 0 &\hspace{0.5cm} \dots \hspace{0.5cm}\\
D_{T_n} &\hspace{0.5cm} 0 &\hspace{0.5cm} 0 &\hspace{0.5cm} 0 &\hspace{0.5cm} \dots \hspace{0.5cm}\\
0 &\hspace{0.5cm} I &\hspace{0.5cm} 0 &\hspace{0.5cm} 0 &\hspace{0.5cm} \dots \hspace{0.5cm}\\
0 &\hspace{0.5cm} 0 &\hspace{0.5cm} I &\hspace{0.5cm} 0 &\hspace{0.5cm} \dots \hspace{0.5cm}\\
0 &\hspace{0.5cm} 0 &\hspace{0.5cm} 0 &\hspace{0.5cm} I &\hspace{0.5cm} \dots \hspace{0.5cm}\\
\vdots &\hspace{0.5cm} \vdots &\hspace{0.5cm} \vdots &\hspace{0.5cm} \vdots &\hspace{0.5cm} \ddots \hspace{0.5cm}
\end{smallmatrix}
\right].
\end{aligned}
\end{equation*}
Then, with respect to the decomposition
$
\mathcal{K}=\mathcal{H}\oplus\ell^2(\mathcal{D}_{T_n}),$
the operator $V_n$ admits the block operator matrix representation
\begin{equation*}
\begin{aligned}
V_n=
\left[
\begin{matrix}
T_n & 0\\
C_n & D_n
\end{matrix}
\right],
\end{aligned}
\end{equation*}
where
\begin{equation*}
\begin{aligned}
C_n=
\left[
\begin{smallmatrix}
D_{T_n}\\
0\\
0\\
\vdots
\end{smallmatrix}
\right]
:\mathcal{H}\longrightarrow\ell^2(\mathcal{D}_{T_n}),
\end{aligned}
\end{equation*}
and
\begin{equation*}
\begin{aligned}
D_n=
\left[
\begin{smallmatrix}
0 &\hspace{0.5cm} 0 &\hspace{0.5cm} 0 &\hspace{0.5cm} \dots \hspace{0.5cm}\\
I &\hspace{0.5cm} 0 &\hspace{0.5cm} 0 &\hspace{0.5cm} \dots \hspace{0.5cm}\\
0 &\hspace{0.5cm} I &\hspace{0.5cm} 0 &\hspace{0.5cm} \dots \hspace{0.5cm}\\
\vdots &\hspace{0.5cm} \vdots &\hspace{0.5cm} \vdots &\hspace{0.5cm} \ddots \hspace{0.5cm}
\end{smallmatrix}
\right]
:\ell^2(\mathcal{D}_{T_n})
\longrightarrow
\ell^2(\mathcal{D}_{T_n}).
\end{aligned}
\end{equation*}
Since $V_i$ commutes with $V_n$, it follows that, with respect to the decomposition
$
\mathcal{K}=\mathcal{H}\oplus\ell^2(\mathcal{D}_{T_n}),$
the operator $V_i$ has the block operator matrix representation
\begin{equation*}
\begin{aligned}
V_i=
\left[
\begin{matrix}
T_i & 0\\
C_i & D_i
\end{matrix}
\right],\quad 1\le i\le n-1.
\end{aligned}
\end{equation*}
Identify the vector-valued Hardy space $H^2(\mathcal{D}_{T_n})$ with $\ell^2(\mathcal{D}_{T_n})$ via the canonical Hilbert space isomorphism. Under this identification, the multiplication operator $M_z^{\mathcal{D}_{T_n}}$ on $H^2(\mathcal{D}_{T_n})$ is unitarily equivalent to $D_n$. Since $V_iV_n=V_nV_i$, it follows that
$
D_iM_z^{\mathcal{D}_{T_n}}
=
M_z^{\mathcal{D}_{T_n}}D_i, ~1\le i\le n-1.$
Hence, by the commutant theorem for the unilateral shift, there exists
$
\Phi_i\in H^\infty\!\big(\mathcal{L}(\mathcal{D}_{T_n})\big)$
such that $
D_i=M_{\Phi_i}^{\mathcal{D}_{T_n}}.$
Since $\mathbf{V}$ is a $\mathbf{\Theta}_n$-isometry, Theorem~\ref{Thm 2} implies that
$
V_i=V_{n-i}^*V_n^p,~1\le i\le n-1.$
Hence,
\begin{equation}\label{D 5}
\begin{aligned}
\left[
\begin{matrix}
T_i & 0\\
C_i & M^{\mathcal{D}_{T_n}}_{\Phi_i}
\end{matrix}
\right]
=V_i
&=V_{n-i}^*V_n^p\\
&=
\left[
\begin{matrix}
T_{n-i}^* & C_{n-i}^*\\
0 & (M_{\Phi_{n-i}}^{\mathcal{D}_{T_n}})^*
\end{matrix}
\right]
\left[
\begin{matrix}
T_n^p & 0\\
C & M_{z^p}^{\mathcal{D}_{T_n}}
\end{matrix}
\right]\\
&=
\left[
\begin{matrix}
T_{n-i}^*T_n^p+C_{n-i}^*C &
C_{n-i}^*M_{z^p}^{\mathcal{D}_{T_n}}\\
(M_{\Phi_{n-i}}^{\mathcal{D}_{T_n}})^*C &
(M_{\Phi_{n-i}}^{\mathcal{D}_{T_n}})^*
M_{z^p}^{\mathcal{D}_{T_n}}
\end{matrix}
\right],
\end{aligned}
\end{equation}
where $C$ is as defined in \eqref{C}. Comparing the corresponding block entries in \eqref{D 5}, we obtain
\begin{equation}\label{D 6}
\begin{aligned}
&T_i=T_{n-i}^*T_n^p+C_{n-i}^*C,\qquad
C_{n-i}^*M_{z^p}^{\mathcal{D}_{T_n}}=0,\\
&C_i=(M_{\Phi_{n-i}}^{\mathcal{D}_{T_n}})^*C,\qquad
M_{\Phi_i}^{\mathcal{D}_{T_n}}
=(M_{\Phi_{n-i}}^{\mathcal{D}_{T_n}})^*
M_{z^p}^{\mathcal{D}_{T_n}},
\end{aligned}
\end{equation}
for $1\le i\le n-1$.
Let
$
\Phi_i(z)=\sum_{k\ge0}A_k^{(i)}z^k, 1\le i\le n-1,$
be the power series expansion of $\Phi_i$. Since
$
M_{\Phi_i}^{\mathcal{D}_{T_n}}
=
\left(M_{\Phi_{n-i}}^{\mathcal{D}_{T_n}}\right)^*
M_{z^p}^{\mathcal{D}_{T_n}},$
we obtain
\begin{equation}\label{D 7}
\begin{aligned}
\sum_{k\ge0}A_k^{(i)}z^k
=
\left(\sum_{k\ge0}A_k^{(n-i)*}\,\overline{z}^{\,k}\right)z^p,
\quad z\in\mathbb{T}.
\end{aligned}
\end{equation}
Comparing the coefficients of $z^k$ and $\overline{z}^{\,k}$ in \eqref{D 7}, we obtain
\[
A_l^{(i)*}=A_{p-l}^{(n-i)},
\quad
0\le l\le p,\quad 1\le i\le n-1,\quad \text{and}\quad
A_l^{(i)}=0,
\qquad l>p.
\]
Hence,
$
\Phi_i(z)=\sum_{k=0}^{p}A_k^{(i)}z^k, z\in\mathbb{T}.$
Then, with respect to the standard orthogonal decomposition of
$
\ell^2(\mathcal{D}_{T_n}),$
the operator $M^{\mathcal{D}_{T_n}}_{\Phi_i}$ has the block operator matrix representation
\begin{equation}\label{D 8}
\begin{aligned}
M^{\mathcal{D}_{T_n}}_{\Phi_i}
=
\left[
\begin{smallmatrix}
A^{(i)}_0 &\hspace{0.5cm} 0 &\hspace{0.5cm} 0 &\hspace{0.5cm} \dots \hspace{0.5cm}\\
A^{(i)}_1 &\hspace{0.5cm} A^{(i)}_0 &\hspace{0.5cm} 0 &\hspace{0.5cm} \dots \hspace{0.5cm}\\
\vdots & \hspace{0.5cm} \vdots & \hspace{0.5cm} \vdots & \hspace{0.5cm} \ddots \hspace{0.5cm}\\
A^{(i)}_{p-2} &\hspace{0.5cm} A^{(i)}_{p-3} &\hspace{0.5cm} A^{(i)}_{p-4} &\hspace{0.5cm} \dots \hspace{0.5cm}\\
A^{(i)}_{p-1} &\hspace{0.5cm} A^{(i)}_{p-2} &\hspace{0.5cm} A^{(i)}_{p-3} &\hspace{0.5cm} \dots \hspace{0.5cm}\\
A^{(i)}_p &\hspace{0.5cm} A^{(i)}_{p-1} &\hspace{0.5cm} A^{(i)}_{p-2} &\hspace{0.5cm} \dots \hspace{0.5cm}\\
\vdots &\hspace{0.5cm} \vdots &\hspace{0.5cm} \vdots &\hspace{0.5cm} \ddots \hspace{0.5cm}
\end{smallmatrix}
\right].
\end{aligned}
\end{equation}
Using \eqref{D 6}, we obtain
\begin{equation}\label{D 9}
\begin{aligned}
C_i
&=(M^{\mathcal{D}_{T_n}}_{\Phi_{n-i}})^*C\\
&=
\left[
\begin{smallmatrix}
\displaystyle\sum_{l=0}^{p-1}A^{(i)}_{p-l}D_{T_n}T_n^{\,p-1-l}\\
\displaystyle\sum_{l=0}^{p-2}A^{(i)}_{p-l}D_{T_n}T_n^{\,p-2-l}\\
\vdots\\
\displaystyle\sum_{l=0}^{1}A^{(i)}_{p-l}D_{T_n}T_n^{\,1-l}\\
A^{(i)}_pD_{T_n}\\
0\\
\vdots
\end{smallmatrix}
\right].
\end{aligned}
\end{equation}
Therefore, by \eqref{D 8} and \eqref{D 9}, the operator $V_i$ admits the following explicit block operator matrix representation:
		\begin{equation*}
			\begin{aligned}
				V_i &=
				\left[
				\begin{smallmatrix}
					T_i &\hspace{0.5cm} 0 &\hspace{0.5cm} 0 &\hspace{0.5cm} 0 &\hspace{0.5cm} \dots \hspace{0.5cm}\\
					\displaystyle\sum_{l=0}^{p-1} A^{(i)}_{p-l}D_{T_n}T^{p-1-l}_n &\hspace{0.5cm} A^{(i)}_0 &\hspace{0.5cm} 0 &\hspace{0.5cm} 0 &\hspace{0.5cm} \dots \hspace{0.5cm}\\
					\displaystyle\sum_{l=0}^{p-2} A^{(i)}_{p-l}D_{T_n}T^{p-2-l}_n &\hspace{0.5cm} A^{(i)}_1 &\hspace{0.5cm} A^{(i)}_0 &\hspace{0.5cm} 0 &\hspace{0.5cm} \dots \hspace{0.5cm}\\
					\vdots & \hspace{0.5cm} \vdots & \hspace{0.5cm} \vdots & \hspace{0.5cm} \vdots & \hspace{0.5cm} \ddots \hspace{0.5cm}\\
					\displaystyle\sum_{l=0}^{1} A^{(i)}_{p-l}D_{T_n}T^{1-l}_n &\hspace{0.5cm} A^{(i)}_{p-2} &\hspace{0.5cm} A^{(i)}_{p-3} &\hspace{0.5cm} A^{(i)}_{p-4} &\hspace{0.5cm} \dots \hspace{0.5cm}\\
					A^{(i)}_pD_{T_n} &\hspace{0.5cm} A^{(i)}_{p-1} &\hspace{0.5cm} A^{(i)}_{p-2} &\hspace{0.5cm} A^{(i)}_{p-3} &\hspace{0.5cm} \dots \hspace{0.5cm}\\
					0 &\hspace{0.5cm} A^{(i)}_p &\hspace{0.5cm} A^{(i)}_{p-1} &\hspace{0.5cm} A^{(i)}_{p-2} &\hspace{0.5cm} \dots \hspace{0.5cm}\\
					\vdots &\hspace{0.5cm} \vdots &\hspace{0.5cm} \vdots &\hspace{0.5cm} \vdots &\hspace{0.5cm} \ddots \hspace{0.5cm}
				\end{smallmatrix}\right],
			\end{aligned}
		\end{equation*}
		for $1 \le i \le n-1$.
Since $V_iV_j=V_jV_i$ for $1\le i,j\le n-1$, we have
$
M_{\Phi_i}^{\mathcal{D}_{T_n}}M_{\Phi_j}^{\mathcal{D}_{T_n}}
=
M_{\Phi_j}^{\mathcal{D}_{T_n}}M_{\Phi_i}^{\mathcal{D}_{T_n}},$
which implies that
\begin{equation}\label{D 10}
\begin{aligned}
\Phi_i(z)\Phi_j(z)
=
\Phi_j(z)\Phi_i(z),
\quad z\in\mathbb{T}.
\end{aligned}
\end{equation}
Comparing the coefficients of $1,z,z^2,\ldots,z^{2p}$ in \eqref{D 10}, we obtain
\[
\sum_{l=0}^{k}[A_l^{(i)},A_{k-l}^{(j)}]=0,
\quad 0\le k\le 2p,\quad 1\leq i,j \leq n-1.
\]
Furthermore, since $V_i$ commutes with $V_n$, comparing the $(2,1)$-entries of $V_iV_n$ and $V_nV_i$, we deduce that
\[
D_{T_n}T_i
=
\sum_{k=0}^{p}A_k^{(i)}D_{T_n}T_n^k,
\quad 1\le i\le n-1.
\]		
Using \eqref{C}, \eqref{D 6}, and \eqref{D 9}, we obtain
\begin{equation*}
\begin{aligned}
T_i-T_{n-i}^*T_n^p
&=C_{n-i}^*C\\
&=
\left[
\begin{smallmatrix}
\displaystyle\sum_{l=0}^{p-1}A_{p-l}^{(n-i)}D_{T_n}T_n^{p-1-l}\\
\displaystyle\sum_{l=0}^{p-2}A_{p-l}^{(n-i)}D_{T_n}T_n^{p-2-l}\\
\vdots\\
\displaystyle\sum_{l=0}^{1}A_{p-l}^{(n-i)}D_{T_n}T_n^{1-l}\\
A_p^{(n-i)*}D_{T_n}\\
0\\
\vdots
\end{smallmatrix}
\right]^*
\left[
\begin{smallmatrix}
D_{T_n}T_n^{p-1}\\
\vdots\\
D_{T_n}T_n\\
D_{T_n}\\
0\\
0\\
\vdots
\end{smallmatrix}
\right]\\
&=
\sum_{l=0}^{p-1}T_n^{*\,p-1-l}D_{T_n}A_l^{(i)}D_{T_n}T_n^{p-1}
+\sum_{l=0}^{p-2}T_n^{*\,p-2-l}D_{T_n}A_l^{(i)}D_{T_n}T_n^{p-2}\\
&
+\cdots
+\sum_{l=0}^{1}T_n^{*\,1-l}D_{T_n}A_l^{(i)}D_{T_n}T_n
+D_{T_n}A_0^{(i)}D_{T_n}.
\end{aligned}
\end{equation*}		
 Therefore, conditions $(1)-(4)$ are necessary for the existence of a $\mathbf{\Theta}_n$-isometric lift of $\mathbf{T}$. This completes the proof.
\end{proof}
Assume that the fundamental equations \eqref{Fundamental} admit unique solutions. Then Theorem~\ref{Conditional Dilation} implies that every minimal $\mathbf{\Theta}_n$-isometric dilation of a $\mathbf{\Theta}_n$-contraction is unitarily equivalent to the $\mathbf{\Theta}_n$-isometric dilation $\mathbf{V}$ constructed in Theorem~\ref{Conditional Dilation}. We record this observation as the following theorem.

\begin{thm}\label{Unitary Equivalence}
Let $\mathbf{T}=(T_1,\ldots,T_n)$ be a $\mathbf{\Theta}_n$-contraction on a Hilbert space $\mathcal{H}$, and suppose that the fundamental equations \eqref{Fundamental} admit unique solutions in $\mathcal{B}(\mathcal{D}_{T_n})$. If $\mathbf{T}$ admits a minimal $\mathbf{\Theta}_n$-isometric dilation $\mathbf{W}=(W_1,\ldots,W_n)$, then $\mathbf{W}$ is unitarily equivalent to the $\mathbf{\Theta}_n$-isometric dilation $\mathbf{V}$ constructed in Theorem~\ref{Conditional Dilation}. In particular, $\mathbf{V}$ satisfies conditions \emph{(1)}--\emph{(4)} of Theorem~\ref{Conditional Dilation}.
\end{thm}	
A dilation theorem for $\Gamma_n$-contractions was established in \cite[Theorem 6.6]{A. Pal}. As a consequence, the author obtained a  minimal $\Gamma_n$-isometric dilation in \cite[Corollary 6.7]{A. Pal} from the  $\Gamma_n$-unitary dilation. We show that the  $\Gamma_n$-isometric dilation can also be derived from Theorem~\ref{Conditional Dilation}. In fact, \cite[Corollary 6.7]{A. Pal} is a special case of the following theorem when $m=p=1$.
First, observe that the sufficiency part of Theorem~\ref{Conditional Dilation} remains valid if the hypothesis
$
(\gamma_1T_1,\ldots,\gamma_{n-1}T_{n-1})
\text{ is a }\Gamma_{n-1}\text{-contraction}$
is replaced by the stronger assumption that $\mathbf{V}$ is a $\mathbf{\Theta}_n$-contraction. The proof is identical to that of Theorem~\ref{Conditional Dilation}. Therefore, we state the following result without proof.
\begin{thm}\label{Re Conditional Dilation}
Let $\mathbf{T}=(T_1,\ldots,T_n)$ be a $\mathbf{\Theta}_n$-contraction on a Hilbert space $\mathcal{H}$, and let
\[
A_k^{(i)}\in\mathcal{B}(\mathcal{D}_{T_n}),
\quad
0\le k\le p,\quad 1\le i\le n-1.
\]
Let $\mathcal{K}$ and $\mathbf{V}=(V_1,\ldots,V_n)$ be as defined in Theorem~\ref{Conditional Dilation}. If $\mathbf{V}$ is a $\mathbf{\Theta}_n$-contraction and conditions \emph{(1)}--\emph{(4)} of Theorem~\ref{Conditional Dilation} are satisfied, then $\mathbf{V}$ is a $\mathbf{\Theta}_n$-isometric dilation of $\mathbf{T}$.
\end{thm}
It is readily verified that, when $m=p=1$, Theorem~\ref{Re Conditional Dilation} reduces to \cite[Corollary 6.7]{A. Pal}, which provides the minimal $\Gamma_n$-isometric dilation of a $\Gamma_n$-contraction under conditions $(1)$ and $(2).$ We rewrite condition $(3)$ as follows:
\begin{equation}\label{Condition 3}
\begin{aligned}
D_{T_n}T_i
=
A^{(i)}_0D_{T_n}
+A^{(i)}_1D_{T_n}T_n
+\cdots
+A^{(i)}_pD_{T_n}T_n^p,
\quad
1\le i\le n-1.
\end{aligned}
\end{equation}
When $m=p=1$, equation \eqref{Condition 3} reduces to
\begin{equation}\label{Condition 3 reduced}
\begin{aligned}
D_{T_n}T_i
=
A^{(i)}_0D_{T_n}
+A^{(i)}_1D_{T_n}T_n,
\quad
1\le i\le n-1,
\end{aligned}
\end{equation}
which is equivalent to
\begin{equation}\label{Condition 3 in Gamma_n contraction}
\begin{aligned}
D_{S_n}S_i
=
X_iD_{S_n}
+X_{n-i}^*D_{S_n}S_n,
\quad
1\le i\le n-1.
\end{aligned}
\end{equation}
Bisai and Pal proved in \cite[Theorem 2.1]{S. Pal} that a tuple of operators $(X_1,\ldots,X_{n-1})$ in $\mathcal{B}(\mathcal{D}_{S_n})$ is the fundamental operator tuple of a $\Gamma_n$-contraction $(S_1,\ldots,S_n)$ if and only if it satisfies \eqref{Condition 3 in Gamma_n contraction}. This naturally leads to the following question.

\medskip

\noindent\textbf{Open Question.}
Do the fundamental operators associated with a $\mathbf{\Theta}_n$-contraction satisfy equation \eqref{Condition 3}?

In this subsection, we also establish necessary conditions for the existence of a
$\mathbf{\Theta}_n$-isometric dilation.
\begin{thm}\label{Necessary Conditions}
Let $\mathbf{T}=(T_1,\dots,T_n)$ be a $\mathbf{\Theta}_n$-contraction on a Hilbert space $\mathcal H$. If $\mathbf{T}$ admits a $\mathbf{\Theta}_n$-isometric dilation, then the following conditions hold.

\begin{enumerate}
\item
The $(n-1)$-tuple
$
(\gamma_1E_1,\dots,\gamma_{n-1}E_{n-1}),$
where $E_1,\dots,E_{n-1}$ are the fundamental operators of the
$\Gamma_n$-contraction
$(T_1,\dots,T_{n-1},T_n^p)$,
admits a joint dilation to a commuting $(n-1)$-tuple of subnormal operators
$
(\gamma_1\widetilde{T}_1,\dots,\gamma_{n-1}\widetilde{T}_{n-1}).$
That is, there exists an isometry
$
\Lambda:\mathcal D_{T_n^p}\longrightarrow \mathcal K$
into a Hilbert space $\mathcal K$ such that
$
E_i=\Lambda^*\widetilde{T}_i\Lambda, ~1\le i\le n-1,$
where
$(\gamma_1\widetilde{T}_1,\dots,\gamma_{n-1}\widetilde{T}_{n-1})$
extends to a commuting $(n-1)$-tuple of normal operators
$
(\gamma_1U_1,\dots,\gamma_{n-1}U_{n-1})$
whose joint spectrum is contained in $\Gamma_{n-1}$.

\item
For  $1\le i,j\le n-1$,
$
\bigl(E_{n-i}^*D_{T_n}T_j
-
E_{n-j}^*D_{T_n}T_i\bigr)
\big|_{\operatorname{ker} D_{T_n^p}}
=0.$

\item
For  $1\le i,j\le n-1$,
$
\bigl(E_i^*E_j^*
-
E_{n-i}^*E_i^*\bigr)
D_{T_n^p}T_n^p
\big|_{\operatorname{ker} D_{T_n^p}}
=0.$

\item
There exists an isometry
$
\Lambda_*:\mathcal D_{T_n}\longrightarrow\mathcal K_*$
into a Hilbert space $\mathcal K_*$ such that
\[
\bigl(
\widetilde{T}_{n-i}^*\Lambda D_{T_n^p}T_n
-
\Lambda_*D_{T_n}T_i
\bigr)
\Big|_{\operatorname{ker} {D_{T_n}}\cap\operatorname{ker} {D_{T_n^p}}}
=0,
\]
for $1\le i\le n-1$.
\end{enumerate}
\end{thm}	
\begin{proof}
Suppose that $\mathbf{V}=(V_1,\dots,V_n)$ is a $\mathbf{\Theta}_n$-isometric dilation of $\mathbf{T}$. Since $\overline{\mathbf{\Theta}}_n$ is polynomially convex, it suffices to work with polynomials instead of the entire algebra $\mathcal{O}(\overline{\mathbf{\Theta}}_n)$. Hence, without loss of generality, we may assume that
\[
\mathcal{K}
=
\overline{\operatorname{span}}
\left\{
V_1^{k_1}\cdots V_n^{k_n}h:
h\in\mathcal H,\,
k_1,\dots,k_n\in\mathbb N\cup\{0\}
\right\}.
\]
By the definition of a $\mathbf{\Theta}_n$-isometric dilation,
\[
(V_1^*,\dots,V_n^*)|_{\mathcal H}
=
(T_1^*,\dots,T_n^*).
\]
Since $\mathbf{T}$ is a $\mathbf{\Theta}_n$-contraction, Lemma~\ref{Lem 1} implies that
$
(T_1,\dots,T_{n-1},T_n^p)$
is a $\Gamma_n$-contraction. Therefore, assertions $(1)$, $(2)$ and $(3)$ follow immediately from \cite[Theorem~2.6]{Mandal}. It remains only to establish $(4)$.
Since $\mathbf{V}$ is a $\mathbf{\Theta}_n$-isometric dilation of $\mathbf{T}$, the operator $V_n$ is an isometry. With respect to the orthogonal decomposition
$
\mathcal K
=
\mathcal H
\oplus
(\mathcal K\ominus\mathcal H),$
each operator $V_i$ admits the block matrix representation
\[
V_i=
\begin{bmatrix}
T_i & 0\\
C_i & \widetilde{T}_i
\end{bmatrix},
\quad
1\le i\le n.
\]
Since $V_n$ is an isometry, we have
\begin{equation}\label{N1}
\begin{aligned}
T_n^*T_n+C_n^*C_n=I_{\mathcal H},
\qquad
\widetilde{T}_n^*\widetilde{T}_n
=
I_{\mathcal K\ominus\mathcal H}.
\end{aligned}
\end{equation}
It follows from \eqref{N1} that there exists an isometry
$
\Lambda_*:\mathcal D_{T_n}\longrightarrow\mathcal K\ominus\mathcal H$
such that
\begin{equation}\label{N2}
\Lambda_*D_{T_n}=C_n.
\end{equation}
Observe that the $2\times2$ block operator matrix representation of $V_n^p$ is
\[
V_n^p=
\begin{bmatrix}
T_n^p & 0\\
C & \widetilde{T}_n^p
\end{bmatrix},
\]
where
$
C=\sum_{l=0}^{p-1}\widetilde{T}_n^{\,l}C_nT_n^{\,p-1-l}.$
Since $V_n^p$ is also an isometry, we obtain
\begin{equation}\label{N3}
\begin{aligned}
T_n^{*p}T_n^p+C^*C=I_{\mathcal H},
\qquad
\widetilde{T}_n^{*p}\widetilde{T}_n^p
=
I_{\mathcal K\ominus\mathcal H}.
\end{aligned}
\end{equation}
Hence, arguing as above, there exists an isometry
$
\Lambda:\mathcal D_{T_n^p}\longrightarrow\mathcal K\ominus\mathcal H$
such that
\begin{equation}\label{N4}
\Lambda D_{T_n^p}=C.
\end{equation}
Since $\mathbf V$ is a $\mathbf{\Theta}_n$-isometry, Theorem~\ref{Thm 5} yields
$
V_i=V_{n-i}^*V_n^p,~1\le i\le n-1.$
Therefore,
\begin{equation}\label{N5}
\begin{aligned}
\begin{bmatrix}
T_i & 0\\
C_i & \widetilde{T}_i
\end{bmatrix}
&=
\begin{bmatrix}
T_{n-i}^* & C_{n-i}^*\\
0 & \widetilde{T}_{n-i}^*
\end{bmatrix}
\begin{bmatrix}
T_n^p & 0\\
C & \widetilde{T}_n^p
\end{bmatrix}  \\
&=
\begin{bmatrix}
T_{n-i}^*T_n^p+C_{n-i}^*C
&
C_{n-i}^*\widetilde{T}_n^p\\
\widetilde{T}_{n-i}^*C
&
\widetilde{T}_{n-i}^*\widetilde{T}_n^p
\end{bmatrix}.
\end{aligned}
\end{equation}
Comparing the corresponding matrix entries in \eqref{N5}, we obtain
\begin{equation}\label{N6}
\begin{aligned}
T_i-T_{n-i}^*T_n^p = C_{n-i}^*C,\quad
C_{n-i}^*\widetilde{T}_n^p =0,\quad
C_i =\widetilde{T}_{n-i}^*C,\quad
\widetilde{T}_i &=\widetilde{T}_{n-i}^*\widetilde{T}_n^p,
\end{aligned}
\quad 1\le i\le n-1.
\end{equation}
Since $V_i$ commutes with $V_n$, comparing the $(2,1)$ entries of the operator matrices for $V_iV_n$ and $V_nV_i$, we obtain
\begin{equation}\label{N7}
\begin{aligned}
C_iT_n+\widetilde{T}_iC_n
=
C_nT_i+\widetilde{T}_nC_i,
\qquad 1\le i\le n-1.
\end{aligned}
\end{equation}
Substituting the identities from \eqref{N6} into \eqref{N7}, we obtain
\begin{equation}\label{N8}
\begin{aligned}
\widetilde{T}_{n-i}^*CT_n+\widetilde{T}_iC_n
=
C_nT_i+\widetilde{T}_n\widetilde{T}_{n-i}^*C.
\end{aligned}
\end{equation}
Using \eqref{N2}, \eqref{N4}, and \eqref{N8}, it follows that
\[
\widetilde{T}_{n-i}^*\Lambda D_{T_n^p}T_n
-
\Lambda_*D_{T_n}T_i
=
\widetilde{T}_n\widetilde{T}_{n-i}^*\Lambda D_{T_n^p}
-
\widetilde{T}_i\Lambda_*D_{T_n}.
\]
Now let
$
x\in\operatorname{ker} {D_{T_n}}\cap\operatorname{ker} {D_{T_n^p}}.$
Since
$
D_{T_n}x=D_{T_n^p}x=0,$
the right-hand side vanishes when applied to $x$. Consequently,
\[
\left(
\widetilde{T}_{n-i}^*\Lambda D_{T_n^p}T_n
-
\Lambda_*D_{T_n}T_i
\right)x=0.
\]
Hence,
\[
\left(
\widetilde{T}_{n-i}^*\Lambda D_{T_n^p}T_n
-
\Lambda_*D_{T_n}T_i
\right)\Big|_{\operatorname{ker} {D_{T_n}}\cap\operatorname{ker} {D_{T_n^p}}}
=0,
\]
which establishes $(4)$ and completes the proof.
\end{proof}	
From the above discussion on the dilation theory of $\mathbf{\Theta}_n$-contractions, it follows that the isometric dilation theory for $\Gamma_n$-contractions arises as a special case of the $\mathbf{\Theta}_n$-isometric dilation theory. In the following example, we exhibit a $\Gamma_3$-contraction that admits a $\Gamma_3$-isometric dilation, while condition $(3)$ of \cite[Proposition~2.10]{Mandal} fails; namely,
\[E_1^*E_1-E_1E_1^*\neq E_2^*E_2-E_2E_2^*.\]
Consequently, the sufficient conditions for the existence of a $\Gamma_3$-isometric dilation established in Thorem \ref{Conditional Dilation} are, in general, not necessary. This remains true even for $\Gamma_3$-contractions $(T_1,T_2,T_3)$ satisfying the following conditions:
\begin{enumerate}
\item $(T_1,T_2)$ is a pair of commuting contractions on a Hilbert space $\mathcal H$;
\item $T_3$ is a partial isometry on $\mathcal H$.
\end{enumerate}	
We shall show that $(T_1,T_2,T_3)$ is a $\Gamma_3$-contraction admitting a
$\Gamma_3$-isometric dilation, while its fundamental operators fail to satisfy
\[
E_1^*E_1-E_1E_1^*\neq E_2^*E_2-E_2E_2^*.
\]
Hence some of  the sufficient conditions in
Thorem \ref{Conditional Dilation} 
are not necessary.
\begin{exam}\label{example}
We illustrate that the sufficient conditions for the existence of a
$\Gamma_3$-isometric dilation given in
Thorem \ref{Conditional Dilation} 
are not necessary.
Let $n=3$, $m=3$, and $p=1$. Consider the following commuting triple of
contractions on $H^2\oplus H^2\oplus H^2$:
\[
(S_1,S_2,V)=
\left(
\begin{bmatrix}
0&0&I_{H^2}\\
0&0&0\\
I_{H^2}&0&0
\end{bmatrix},
\begin{bmatrix}
T_z&0&0\\
0&0&0\\
0&0&T_z
\end{bmatrix},
\begin{bmatrix}
I_{H^2}&0&0\\
0&T_z&0\\
0&0&I_{H^2}
\end{bmatrix}
\right),
\]
where $T_z$ denotes the unilateral shift on $H^2$.
Define
\[
T_1=\frac{1}{3}(S_1^3+S_2^3+V^3)
=\frac{1}{3}
\begin{bmatrix}
I_{H^2}+T_z^3&0&I_{H^2}\\
0&T_z&0\\
I_{H^2}&0&I_{H^2}+T_z^3
\end{bmatrix},
\]
\[
T_2=\frac{1}{3}(S_1^3S_2^3+S_2^3V^3+S_1^3V^3)
=\frac{1}{3}
\begin{bmatrix}
T_z^3&0&I_{H^2}+T_z^3\\
0&0&0\\
I_{H^2}+T_z^3&0&T_z^3
\end{bmatrix},\quad \text{and}\quad
T_3=S_1^3S_2^3V^3=
\begin{bmatrix}
0&0&T_z^3\\
0&0&0\\
T_z^3&0&0
\end{bmatrix}.
\]
Therefore, by \cite[Theorem~2.9]{Mandal}, the triple
$(T_1,T_2,T_3)$ is a $\Gamma_3$-contraction admitting a
$\Gamma_3$-isometric dilation. Since $T_z$ is an isometry, it is
straightforward to verify that $T_3$ is a partial isometry. Hence,
by \cite[Proposition~2.10]{Mandal}, we determine the fundamental
operators of $(T_1,T_2,T_3)$.
We first compute the defect operator of $T_3$. Since
\[
\begin{aligned}
D_{T_3}^2
&=I-T_3^*T_3\\
&=
\begin{bmatrix}
I_{H^2} & 0 & 0\\
0 & I_{H^2} & 0\\
0 & 0 & I_{H^2}
\end{bmatrix}
-
\begin{bmatrix}
0 & 0 & T_z^3\\
0 & 0 & 0\\
T_z^3 & 0 & 0
\end{bmatrix}^{*}
\begin{bmatrix}
0 & 0 & T_z^3\\
0 & 0 & 0\\
T_z^3 & 0 & 0
\end{bmatrix}\\
&=
\begin{bmatrix}
0 & 0 & 0\\
0 & I_{H^2} & 0\\
0 & 0 & 0
\end{bmatrix},
\end{aligned}
\]
it follows that
\[
D_{T_3}=
\begin{bmatrix}
0 & 0 & 0\\
0 & I_{H^2} & 0\\
0 & 0 & 0
\end{bmatrix},
\]
since $D_{T_3}$ is an orthogonal projection.
Now set
\(
(E_1,E_2)=\left(\frac{T_z^3}{3},\,0\right).
\)
We verify that $(E_1,E_2)$ is the fundamental operator pair associated with the $\Gamma_3$-contraction $(T_1,T_2,T_3)$. A direct computation shows that
\begin{align*}
T_1-T_2^*T_3
&=\frac{1}{3}
\begin{bmatrix}
I_{H^2}+T_z^3 & 0 & I_{H^2}\\
0&T_z&0\\
I_{H^2}&0&I_{H^2}+T_z^3
\end{bmatrix}
-\frac{1}{3}
\begin{bmatrix}
T_z^3&0&I_{H^2}+T_z^3\\
0&0&0\\
I_{H^2}+T_z^3&0&T_z^3
\end{bmatrix}^{*}
\begin{bmatrix}
0&0&T_z^3\\
0&0&0\\
T_z^3&0&0
\end{bmatrix} \\
&=\frac{1}{3}
\begin{bmatrix}
0&0&0\\
0&T_z^3&0\\
0&0&0
\end{bmatrix}
=
D_{T_3}
\begin{bmatrix}
0&0&0\\
0&\dfrac{T_z^3}{3}&0\\
0&0&0
\end{bmatrix}
D_{T_3}
=D_{T_3}E_1D_{T_3},
\end{align*}
where
\[
E_1=
\begin{bmatrix}
0&0&0\\
0&\dfrac{T_z^3}{3}&0\\
0&0&0
\end{bmatrix}.
\]
Similarly,
\begin{align*}
T_2-T_1^*T_3
&=\frac{1}{3}
\begin{bmatrix}
T_z^3&0&I_{H^2}+T_z^3\\
0&0&0\\
I_{H^2}+T_z^3&0&T_z^3
\end{bmatrix}
-\frac{1}{3}
\begin{bmatrix}
I_{H^2}+T_z^3&0&I_{H^2}\\
0&T_z^3&0\\
I_{H^2}&0&I_{H^2}+T_z^3
\end{bmatrix}^{*}
\begin{bmatrix}
0&0&T_z^3\\
0&0&0\\
T_z^3&0&0
\end{bmatrix}\\
&=0
=D_{T_3}E_2D_{T_3},
\end{align*}
where $E_2=0$.
Hence, by the uniqueness of the fundamental operators, we conclude that
\[
(F_1,F_2)=(E_1,E_2)
=\left(
\begin{bmatrix}
0&0&0\\
0&\dfrac{T_z^3}{3}&0\\
0&0&0
\end{bmatrix},
\,0
\right),
\]
where $(F_1,F_2)$ denotes the fundamental operator pair of the
$\Gamma_3$-contraction $(T_1,T_2,T_3)$. A direct computation shows that
\begin{equation}\label{Ex 1}
\begin{aligned}
F_1^*F_1-F_1F_1^*
=
\frac{1}{9}(I-T_z^3T_z^{*3})
\quad\text{and}\quad
F_2^*F_2-F_2F_2^*=0.
\end{aligned}
\end{equation}
It is straightforward to verify that
$
\operatorname{ker} T_3=\{0\}\oplus H^2\oplus\{0\}$
and
$
(T_1,T_2)|_{\operatorname{ker} T_3}
=(D_1,D_2)
=
\left(\frac{T_z^3}{3},\,0\right).$
Moreover,
\begin{equation}\label{Ex 2}
\begin{aligned}
D_1^*D_1-D_1D_1^*
=
\frac{1}{9}(I-T_z^3T_z^{3*})
\quad\text{and}\quad
D_2^*D_2-D_2D_2^*=0.
\end{aligned}
\end{equation}
Hence,
\[
D_1^*D_1-D_1D_1^*
\neq
D_2^*D_2-D_2D_2^*,
\]
showing that condition $(3)$ of \cite[Proposition~2.10]{Mandal} is not satisfied.
Now let
$
A_0^{(1)}=F_1,
A_1^{(1)}=F_2.$
Then
$
A_0^{(2)}=A_1^{(1)*}=F_2^*,
A_1^{(2)}=A_0^{(1)*}=F_1^*.$
Therefore, by \eqref{Ex 1},
\begin{align*}
\sum_{l=0}^{1}[A_l^{(1)},A_{1-l}^{(2)}]
&=
[A_0^{(1)},A_1^{(2)}]
+
[A_1^{(1)},A_0^{(2)}]\\
&=
F_1F_1^*-F_1^*F_1
+
F_2F_2^*-F_2^*F_2\\
&=
\frac{1}{9}(T_z^3T_z^{*3}-I)\neq0.
\end{align*}
Consequently, condition $(2)$ of Theorem~\ref{Conditional Dilation} is not necessary for the existence of a $\mathbf{\Theta}_n$-isometric dilation.
\end{exam}
\subsection{Pure $\mathbf{\Theta}_n$-isometric dilations}

In this section, we study pure $\mathbf{\Theta}_n$-isometric dilations of
$\mathbf{\Theta}_n$-contractions. The Hardy space realization of pure
$\mathbf{\Theta}_n$-isometries is already known and will be used only as a
previously established result; see \cite{Biswas 2}. Our aim is to formulate
the dilation problem in terms of invariant subspaces of vector-valued Hardy
spaces.
Let
\(
\mathbf T=(T_1,\ldots,T_n)
\)
be a $\mathbf{\Theta}_n$-contraction on a Hilbert space $\mathcal H$. A
commuting $\mathbf{\Theta}_n$-isometry
\(
\mathbf V=(V_1,\ldots,V_n)
\)
on a Hilbert space $\mathcal K\supseteq\mathcal H$ is called an
\emph{invariant $\mathbf{\Theta}_n$-isometric dilation} of $\mathbf T$ if
$\mathcal H$ is invariant under each $V_i$ and
\[
T_i=V_i|_{\mathcal H},
\quad 1\leq i\leq n.
\]
Equivalently,
\[
T_1^{k_1}\cdots T_n^{k_n}
=
V_1^{k_1}\cdots V_n^{k_n}|_{\mathcal H},
\quad
k_1,\ldots,k_n\geq0.
\]
The dilation is called \emph{pure} if $V_n$ is a pure isometry.
We first give a characterization of the existence of a pure
$\mathbf{\Theta}_n$-isometric dilation in terms of an intertwining
isometry.

\begin{thm}\label{thm:pure-theta-n-dilation}
Let
\(
\mathbf T=(T_1,\ldots,T_n)
\)
be a $\mathbf{\Theta}_n$-contraction on a Hilbert space $\mathcal H$.
Then the following assertions are equivalent.

\begin{enumerate}
    \item[(i)] $\mathbf T$ admits a pure invariant
    $\mathbf{\Theta}_n$-isometric dilation.

    \item[(ii)] There exist a Hilbert space $\mathcal E$, operator-valued
    polynomials
    \(
    \Phi_i(z)
    =
    \sum_{\ell=0}^{p}A_\ell^{(i)}z^\ell,
    1\leq i\leq n-1,
    \)
    with
    \(
    A_\ell^{(i)}\in\mathcal B(\mathcal E),
    \)
    satisfying
    \begin{equation}\label{eq:theta-n-coeff}
    A_\ell^{(i)^*}
    =
    A_{p-\ell}^{(n-i)},
    \quad
    0\leq\ell\leq p,
    \quad
    1\leq i\leq n-1,
    \end{equation}
    and an isometry
    \(
    J:\mathcal H\longrightarrow H^2(\mathcal E)
    \)
    such that
    \begin{equation}\label{eq:theta-n-intertwining}
    JT_i=M_{\Phi_i}J,
    \quad
    1\leq i\leq n-1,
    \quad
    JT_n=M_zJ.
    \end{equation}

    \item[(iii)] There exist a Hilbert space $\mathcal E$ and
    operator-valued polynomials
    \[
    \Phi_i(z)
    =
    \sum_{\ell=0}^{p}A_\ell^{(i)}z^\ell,
    \quad
    1\leq i\leq n-1,
    \]
    satisfying \eqref{eq:theta-n-coeff}, and an isometry
    \(
    J:\mathcal H\longrightarrow H^2(\mathcal E)
    \)
    such that
    \begin{equation}\label{eq:theta-n-all-powers}
    JT_1^{k_1}\cdots T_n^{k_n}
    =
    M_{\Phi_1}^{k_1}\cdots
    M_{\Phi_{n-1}}^{k_{n-1}}M_z^{k_n}J
    \end{equation}
    for every
    \(
    k_1,\ldots,k_n\geq0,
    \)
    and
    \(
    \left(
    M_{\Phi_1},\ldots,M_{\Phi_{n-1}},M_z
    \right)
    \)
    is a pure $\mathbf{\Theta}_n$-isometry.
\end{enumerate}
\end{thm}

\begin{proof}
Suppose first that {\rm (i)} holds. Let
\(
\mathbf V=(V_1,\ldots,V_n)
\)
be a pure invariant $\mathbf{\Theta}_n$-isometric dilation of
$\mathbf T$ on a Hilbert space $\mathcal K\supseteq\mathcal H$.
By the Hardy space realization of pure $\mathbf{\Theta}_n$-isometries
established in \cite{Biswas 2}, there exist a Hilbert space $\mathcal E$
and a unitary operator
\(
U:\mathcal K\longrightarrow H^2(\mathcal E)
\)
such that
\(
UV_nU^*=M_z
\)
and
\(
UV_iU^*=M_{\Phi_i},
1\leq i\leq n-1,
\)
where
\(
\Phi_i(z)
=
\sum_{\ell=0}^{p}A_\ell^{(i)}z^\ell
\)
and the coefficients satisfy
\(
\bigl(A_\ell^{(i)}\bigr)^*
=
A_{p-\ell}^{(n-i)},
0\leq\ell\leq p.
\)
Moreover,
\(
\left(
M_{\Phi_1},\ldots,M_{\Phi_{n-1}},M_z
\right)
\)
is a pure $\mathbf{\Theta}_n$-isometry.
Define
\(
J=U|_{\mathcal H}.
\)
Since $U$ is unitary, $J$ is an isometry. Since $\mathcal H$ is invariant
under each $V_i$, we have
\[
\begin{aligned}
JT_i
=
UV_i|_{\mathcal H}=
M_{\Phi_i}U|_{\mathcal H}=
M_{\Phi_i}J,
\quad 1\leq i\leq n-1,
\end{aligned}
\]
and
\[
JT_n
=
UV_n|_{\mathcal H}
=
M_zU|_{\mathcal H}
=
M_zJ.
\]
Thus \eqref{eq:theta-n-intertwining} holds, and hence {\rm (ii)} follows.

Suppose now that {\rm (ii)} holds. From
\[
JT_i=M_{\Phi_i}J,
\quad
1\leq i\leq n-1,\quad
JT_n=M_zJ,
\]
we obtain, by induction,
\[
JT_i^{k_i}
=
M_{\Phi_i}^{k_i}J,
\quad
1\leq i\leq n-1, \quad
JT_n^{k_n}=M_z^{k_n}J.
\]
Since $\mathbf T$ is a commuting tuple and
\(
\left(
M_{\Phi_1},\ldots,M_{\Phi_{n-1}},M_z
\right)
\)
is also commuting, we obtain
\[
\begin{aligned}
JT_1^{k_1}\cdots T_n^{k_n}
&=
M_{\Phi_1}^{k_1}\cdots
M_{\Phi_{n-1}}^{k_{n-1}}
M_z^{k_n}J.
\end{aligned}
\]
Hence \eqref{eq:theta-n-all-powers} holds, proving {\rm (iii)}.

Finally, suppose that {\rm (iii)} holds. Put
\(
\mathcal M=J\mathcal H.
\)
Taking
\(
k_i=1, k_j=0, 1\leq j\neq i\leq n\)
in \eqref{eq:theta-n-all-powers}, we obtain
\[
M_{\Phi_i}\mathcal M\subseteq\mathcal M,
\quad
1\leq i\leq n-1,\quad
M_z\mathcal M\subseteq\mathcal M.
\]
Thus $\mathcal M$ is jointly invariant under
\(
\left(
M_{\Phi_1},\ldots,M_{\Phi_{n-1}},M_z
\right).
\)
By assumption, this tuple is a pure $\mathbf{\Theta}_n$-isometry.
Furthermore,
\[
JT_i=M_{\Phi_i}J,
\quad
1\leq i\leq n-1,\quad
JT_n=M_zJ.
\]
Therefore, under the identification
\(
\mathcal H\simeq\mathcal M=J\mathcal H,
\)
the tuple $\mathbf T$ is the restriction of
\(
\left(
M_{\Phi_1},\ldots,M_{\Phi_{n-1}},M_z
\right)
\)
to $\mathcal M$. Hence this tuple is a pure
$\mathbf{\Theta}_n$-isometric dilation of $\mathbf T$.
Thus {\rm (i)} follows.
\end{proof}

\begin{cor}\label{cor:pure-theta-n-compression}
Let
\(
\mathbf T=(T_1,\ldots,T_n)
\)
be a $\mathbf{\Theta}_n$-contraction. If $\mathbf T$ admits a pure
invariant $\mathbf{\Theta}_n$-isometric dilation, then there exist a
Hilbert space $\mathcal E$, operator-valued polynomials
\[
\Phi_i(z)
=
\sum_{\ell=0}^{p}A_\ell^{(i)}z^\ell,
\quad
1\leq i\leq n-1,
\]
satisfying
\(
A_\ell^{(i)^*}
=
A_{p-\ell}^{(n-i)},
0\leq\ell\leq p,
\)
and an isometry
\(
J:\mathcal H\longrightarrow H^2(\mathcal E)
\)
such that
\begin{equation}\label{eq:theta-n-compression}
T_1^{k_1}\cdots T_n^{k_n}
=
J^*
M_{\Phi_1}^{k_1}\cdots
M_{\Phi_{n-1}}^{k_{n-1}}
M_z^{k_n}J
\end{equation}
for every
\(
k_1,\ldots,k_n\geq0.
\)
\end{cor}

\begin{proof}
By Theorem~\ref{thm:pure-theta-n-dilation}, there exist
$\mathcal E$, $\Phi_1,\ldots,\Phi_{n-1}$ and $J$ such that
\[
JT_1^{k_1}\cdots T_n^{k_n}
=
M_{\Phi_1}^{k_1}\cdots
M_{\Phi_{n-1}}^{k_{n-1}}
M_z^{k_n}J.
\]
Since $J$ is an isometry,
\(
J^*J=I_{\mathcal H}.
\)
Applying $J^*$ to both sides gives
\[
T_1^{k_1}\cdots T_n^{k_n}
=
J^*
M_{\Phi_1}^{k_1}\cdots
M_{\Phi_{n-1}}^{k_{n-1}}
M_z^{k_n}J.
\]
\end{proof}

\subsection{Invariant subspace formulation}
We next formulate the preceding characterization in terms of invariant
subspaces of vector-valued Hardy spaces.

\begin{thm}\label{thm:theta-n-invariant-subspace}
Let
\(
\mathbf T=(T_1,\ldots,T_n)
\)
be a $\mathbf{\Theta}_n$-contraction on $\mathcal H$. Suppose that
$\mathbf T$ admits a pure invariant $\mathbf{\Theta}_n$-isometric dilation.
Then there exist Hilbert spaces $\mathcal E$ and $\mathcal F$,
operator-valued polynomials
\[
\Phi_i(z)
=
\sum_{\ell=0}^{p}A_\ell^{(i)}z^\ell,
\quad
1\leq i\leq n-1,
\]
an inner function
\(
\Theta\in
H^\infty_{\mathcal B(\mathcal F,\mathcal E)},
\)
and functions
\(
\Psi_i\in
H^\infty_{\mathcal B(\mathcal F)},
1\leq i\leq n-1,
\)
such that
\begin{equation}\label{eq:theta-n-invariant-coeff}
A_\ell^{(i)^*}
=
A_{p-\ell}^{(n-i)},
0\leq\ell\leq p,
\end{equation}
and
\begin{equation}\label{eq:theta-n-symbol-intertwining}
\Phi_i(z)\Theta(z)
=
\Theta(z)\Psi_i(z),
\quad
1\leq i\leq n-1,
\quad z\in\mathbb D.
\end{equation}
Moreover,
\(
\mathbf T
\simeq
\left(
M_{\Phi_1}|_{\Theta H^2(\mathcal F)},
\ldots,
M_{\Phi_{n-1}}|_{\Theta H^2(\mathcal F)},
M_z|_{\Theta H^2(\mathcal F)}
\right).
\)
\end{thm}

\begin{proof}
By Theorem~\ref{thm:pure-theta-n-dilation}, there exist a Hilbert space
$\mathcal E$, operator-valued polynomials
\(
\Phi_1,\ldots,\Phi_{n-1}
\)
and an isometry
\(
J:\mathcal H\longrightarrow H^2(\mathcal E)
\)
such that
\[
JT_i=M_{\Phi_i}J,
\quad
1\leq i\leq n-1,\quad
JT_n=M_zJ.
\]
Put
\(
\mathcal M=J\mathcal H.
\)
Then
\[
M_{\Phi_i}\mathcal M\subseteq\mathcal M,
\quad
1\leq i\leq n-1,\quad
M_z\mathcal M\subseteq\mathcal M.
\]
In particular, $\mathcal M$ is invariant under $M_z$. By the
Beurling--Lax--Halmos theorem, there exist a Hilbert space $\mathcal F$
and an inner function
\(
\Theta\in
H^\infty_{\mathcal B(\mathcal F,\mathcal E)}
\)
such that
\begin{equation}\label{eq:theta-n-M}
\mathcal M=\Theta H^2(\mathcal F).
\end{equation}
For each $1\leq i\leq n-1$, the invariance of $\mathcal M$ under
$M_{\Phi_i}$ gives
\[
M_{\Phi_i}M_\Theta H^2(\mathcal F)
\subseteq
M_\Theta H^2(\mathcal F).
\]
Consequently, there exists
\(
\Psi_i\in
H^\infty_{\mathcal B(\mathcal F)}
\)
such that
\(
M_{\Phi_i}M_\Theta
=
M_\Theta M_{\Psi_i}.
\)
Therefore,
\(
M_{\Phi_i\Theta}
=
M_{\Theta\Psi_i},
\)
and hence
\[
\Phi_i(z)\Theta(z)
=
\Theta(z)\Psi_i(z),
\quad z\in\mathbb D.
\]
This proves \eqref{eq:theta-n-symbol-intertwining}.
Finally, under the unitary identification
\(
\mathcal H
\simeq
\mathcal M
=
\Theta H^2(\mathcal F),
\)
the operators $T_i$, $1\leq i\leq n-1$, and $T_n$ correspond,
respectively, to
\(
M_{\Phi_i}|_{\Theta H^2(\mathcal F)}
\)
and
\(
M_z|_{\Theta H^2(\mathcal F)}.
\)
The assertion follows.
\end{proof}

\begin{rem}\label{rem:theta-n-novelty}
Theorem~\ref{thm:theta-n-invariant-subspace} is a dilation-theoretic
consequence of the known Hardy space realization of pure
$\mathbf{\Theta}_n$-isometries rather than a repetition of that realization
theorem. The additional structure is the invariant subspace representation
\(
\mathcal M=\Theta H^2(\mathcal F)
\)
and the simultaneous intertwining relations
\[
\Phi_i\Theta=\Theta\Psi_i,
\quad
1\leq i\leq n-1.
\]
Thus, for a prescribed $\mathbf{\Theta}_n$-contraction, the construction of
a pure $\mathbf{\Theta}_n$-isometric dilation is reduced to finding an
inner multiplier $\Theta$ which simultaneously intertwines the
operator-valued symbols $\Phi_1,\ldots,\Phi_{n-1}$ with
$\Psi_1,\ldots,\Psi_{n-1}$.
\end{rem}

\subsection{A sufficient condition for pure
$\mathbf{\Theta}_n$-isometric dilation}

The preceding discussion gives the following sufficient condition.

\begin{thm}\label{thm:theta-n-sufficient}
Let
\(
\mathbf T=(T_1,\ldots,T_n)
\)
be a $\mathbf{\Theta}_n$-contraction on $\mathcal H$. Suppose that there
exist Hilbert spaces $\mathcal E$ and $\mathcal F$, operator-valued
polynomials
\[
\Phi_i(z)
=
\sum_{\ell=0}^{p}A_\ell^{(i)}z^\ell,
\quad
1\leq i\leq n-1,
\]
an inner function
\(
\Theta\in
H^\infty_{\mathcal B(\mathcal F,\mathcal E)}
\)
and functions
\(
\Psi_i\in
H^\infty_{\mathcal B(\mathcal F)},
1\leq i\leq n-1,
\)
satisfying the following conditions:

\begin{enumerate}
    \item[(a)]
    The tuple
    \(
    \left(
    M_{\Phi_1},\ldots,M_{\Phi_{n-1}},M_z
    \right)
    \)
    is a pure $\mathbf{\Theta}_n$-isometry on $H^2(\mathcal E)$.

    \item[(b)]
    \(
    \Phi_i(z)\Theta(z)
    =
    \Theta(z)\Psi_i(z),
    1\leq i\leq n-1, z\in\mathbb D.
    \)

    \item[(c)] There exists a unitary operator
    \(
    W:\mathcal H
    \longrightarrow
    \Theta H^2(\mathcal F)
    \)
    such that
    \[
    WT_i
    =
    M_{\Phi_i}|_{\Theta H^2(\mathcal F)}W,
    \quad
    1\leq i\leq n-1,\quad
    WT_n
    =
    M_z|_{\Theta H^2(\mathcal F)}W.
    \]
\end{enumerate}
Then $\mathbf T$ admits a pure invariant
$\mathbf{\Theta}_n$-isometric dilation.
\end{thm}

\begin{proof}
By condition~(a),
\(
\mathbf V
=
\left(
M_{\Phi_1},\ldots,M_{\Phi_{n-1}},M_z
\right)
\)
is a pure $\mathbf{\Theta}_n$-isometry on $H^2(\mathcal E)$.
Condition~(b) implies
\(
M_{\Phi_i}\Theta H^2(\mathcal F)
\subseteq
\Theta H^2(\mathcal F),
1\leq i\leq n-1.
\)
Since $\Theta H^2(\mathcal F)$ is invariant under $M_z$, we also have
\[
M_z\Theta H^2(\mathcal F)
\subseteq
\Theta H^2(\mathcal F).
\]
Thus $\Theta H^2(\mathcal F)$ is jointly invariant under
$\mathbf V$.
By condition~(c), the restriction of $\mathbf V$ to
$\Theta H^2(\mathcal F)$ is unitarily equivalent to $\mathbf T$.
Therefore $\mathbf V$ is a pure $\mathbf{\Theta}_n$-isometric dilation of
$\mathbf T$.
\end{proof}
\begin{thm}[Minimality Criterion]
\label{thm:minimality-hardy-theta-n}
Let
\(
\mathbf V=(M_{\Phi_1},\ldots,M_{\Phi_{n-1}},M_z)
\)
be a pure $\mathbf{\Theta}_n$-isometry on $H^2(\mathcal E)$, and let
\(
J:\mathcal H\longrightarrow H^2(\mathcal E)
\)
be an isometry satisfying
\[
JT_i=M_{\Phi_i}J,\quad 1\leq i\leq n-1,\quad \text{
and}\quad
JT_n=M_zJ.
\]
Set
\(
\mathcal K_J=
\overline{\operatorname{span}}
\left\{
M_{\Phi_1}^{k_1}\cdots
M_{\Phi_{n-1}}^{k_{n-1}}M_z^{k_n}J\mathcal H:
k_1,\ldots,k_n\geq0
\right\}.
\)
Then the followings hold:
\begin{enumerate}
\item[(i)]
$\mathcal K_J$ is invariant for
$M_{\Phi_1},\ldots,M_{\Phi_{n-1}},M_z$.

\item[(ii)]
The restriction
\(
\mathbf V_J=
\left(
M_{\Phi_1}|_{\mathcal K_J},\ldots,
M_{\Phi_{n-1}}|_{\mathcal K_J},
M_z|_{\mathcal K_J}
\right)
\)
is a $\mathbf{\Theta}_n$-isometric dilation of $\mathbf T$.

\item[(iii)]
$\mathbf V_J$ is minimal.

\item[(iv)]
Any two minimal $\mathbf{\Theta}_n$-isometric dilations of
$\mathbf T$ arising from the above intertwining representation are
unitarily equivalent.
\end{enumerate}
\end{thm}

\begin{proof}
Put
\(
\mathcal M_0=
\operatorname{span}
\left\{
M_{\Phi_1}^{k_1}\cdots
M_{\Phi_{n-1}}^{k_{n-1}}M_z^{k_n}Jh:
k_1,\ldots,k_n\geq0,\ h\in\mathcal H
\right\}.
\)
Then
\(
\mathcal K_J=\overline{\mathcal M_0}.
\)
We first prove (i). Since 
\(
M_{\Phi_1},\ldots,M_{\Phi_{n-1}},M_z
\)
commute, for $1\leq i\leq n-1$ and
\(
M_{\Phi_1}^{k_1}\cdots M_{\Phi_{n-1}}^{k_{n-1}}
M_z^{k_n}Jh\in\mathcal M_0,
\)
we have
\[
\begin{aligned}
M_{\Phi_i}
M_{\Phi_1}^{k_1}\cdots
M_{\Phi_{n-1}}^{k_{n-1}}M_z^{k_n}Jh
&=
M_{\Phi_1}^{k_1}\cdots
M_{\Phi_i}^{k_i+1}\cdots
M_{\Phi_{n-1}}^{k_{n-1}}M_z^{k_n}Jh\\
&\in\mathcal M_0.
\end{aligned}
\]
Hence
\(
M_{\Phi_i}\mathcal M_0\subseteq\mathcal M_0.
\)
Since $M_{\Phi_i}$ is bounded, it follows that
\[
M_{\Phi_i}\mathcal K_J
=
M_{\Phi_i}\overline{\mathcal M}_0
\subseteq
\overline{\mathcal M}_0
=\mathcal K_J.
\]
Similarly,
\[
M_z
M_{\Phi_1}^{k_1}\cdots
M_{\Phi_{n-1}}^{k_{n-1}}M_z^{k_n}Jh
=
M_{\Phi_1}^{k_1}\cdots
M_{\Phi_{n-1}}^{k_{n-1}}M_z^{k_n+1}Jh
\in\mathcal M_0.
\]
Therefore
\(
M_z\mathcal K_J\subseteq\mathcal K_J.
\)
Thus $\mathcal K_J$ is invariant for
$M_{\Phi_1},\ldots,M_{\Phi_{n-1}},M_z$, proving (i).

We next prove (ii). By (i), the operators
\[
V_{J,i}:=M_{\Phi_i}|_{\mathcal K_J},
\quad 1\leq i\leq n-1,\quad
V_{J,n}:=M_z|_{\mathcal K_J}
\]
are well-defined. Since $\mathbf V$ is a
$\mathbf{\Theta}_n$-isometry and $\mathcal K_J$ is a common invariant
subspace, the restricted tuple
\(
\mathbf V_J=(V_{J,1},\ldots,V_{J,n})
\)
is a $\mathbf{\Theta}_n$-isometry.
We now verify the dilation property. From the intertwining relations,
\[
JT_i=M_{\Phi_i}J,\quad 1\leq i\leq n-1,\quad 
JT_n=M_zJ.
\]
Since the operators $T_1,\ldots,T_n$ commute and
$M_{\Phi_1},\ldots,M_{\Phi_{n-1}},M_z$ commute, repeated application of
these relations gives
\[
JT_1^{k_1}\cdots T_n^{k_n}
=
M_{\Phi_1}^{k_1}\cdots
M_{\Phi_{n-1}}^{k_{n-1}}M_z^{k_n}J
\]
for every $k_1,\ldots,k_n\geq0$. Consequently,
\[
T_1^{k_1}\cdots T_n^{k_n}
=
J^*
M_{\Phi_1}^{k_1}\cdots
M_{\Phi_{n-1}}^{k_{n-1}}M_z^{k_n}J,
\]
because $J^*J=I_{\mathcal H}$.
Identifying $\mathcal H$ with $J\mathcal H$, we obtain
\[
T_1^{k_1}\cdots T_n^{k_n}
=
P_{J\mathcal H}
V_{J,1}^{k_1}\cdots V_{J,n}^{k_n}
\big|_{J\mathcal H},
\quad k_1,\ldots,k_n\geq0.
\]
Thus $\mathbf V_J$ is a $\mathbf{\Theta}_n$-isometric dilation of
$\mathbf T$. This proves (ii).

For (iii), observe that, by the definition of $\mathcal K_J$,
\[
\begin{aligned}
\mathcal K_J
&=
\overline{\operatorname{span}}
\left\{
M_{\Phi_1}^{k_1}\cdots
M_{\Phi_{n-1}}^{k_{n-1}}M_z^{k_n}Jh:
k_1,\ldots,k_n\geq0,\ h\in\mathcal H
\right\}\\
&=
\overline{\operatorname{span}}
\left\{
V_{J,1}^{k_1}\cdots
V_{J,n}^{k_n}Jh:
k_1,\ldots,k_n\geq0,\ h\in\mathcal H
\right\}.
\end{aligned}
\]
Hence $\mathcal K_J$ is precisely the smallest closed invariant
subspace for $\mathbf V_J$ containing $J\mathcal H$. Equivalently,
\[
\mathcal K_J
=
\overline{
\operatorname{span}}
\left\{
V_{J,1}^{k_1}\cdots V_{J,n}^{k_n}J h:
k_j\geq0,\ h\in\mathcal H
\right\}.
\]
Therefore no proper closed invariant subspace of $\mathcal K_J$
containing $J\mathcal H$ can be invariant under all the operators
$V_{J,1},\ldots,V_{J,n}$. Thus $\mathbf V_J$ is a minimal
$\mathbf{\Theta}_n$-isometric dilation of $\mathbf T$.

Finally, we prove (iv). Let
\(
\mathbf V=(V_1,\ldots,V_n)~
\text{on }\mathcal K
\)
and
\(
\widetilde{\mathbf V}
=(\widetilde V_1,\ldots,\widetilde V_n)~\text{on }\widetilde{\mathcal K}
\)
be two minimal $\mathbf{\Theta}_n$-isometric dilations of $\mathbf T$
arising from the same intertwining representation. Let
\(
J:\mathcal H\to\mathcal K,
\widetilde J:\mathcal H\to\widetilde{\mathcal K}
\)
denote the corresponding isometric embeddings. Thus
\[
V_iJ=JT_i,
\quad
\widetilde V_i\widetilde J=\widetilde JT_i,
\quad 1\leq i\leq n.
\]
Consider the linear map $U_0$ defined initially on the algebraic span
\[
\mathcal D=
\operatorname{span}
\left\{
V_1^{k_1}\cdots V_n^{k_n}Jh:
k_1,\ldots,k_n\geq0,\ h\in\mathcal H
\right\}
\]
by
\[
U_0
\left(
V_1^{k_1}\cdots V_n^{k_n}Jh
\right)
=
\widetilde V_1^{k_1}\cdots
\widetilde V_n^{k_n}\widetilde Jh.
\]
The dilation identities imply that $U_0$ preserves the inner products
of the generating vectors. Indeed, for
$h,g\in\mathcal H$ and
$\mathbf{k},\mathbf{l}\in\mathbb Z_+^n$, the corresponding inner
products are determined by the products of the operators
$T_1,\ldots,T_n$ through the intertwining relations. Hence
\[
\left\langle
V^{\mathbf{k}}Jh,V^{\mathbf{l}}Jg
\right\rangle
=
\left\langle
\widetilde V^{\mathbf{k}}\widetilde Jh,
\widetilde V^{\mathbf{l}}\widetilde Jg
\right\rangle.
\]
Consequently, $U_0$ is well-defined and isometric on $\mathcal D$.
Since $\mathbf V$ is minimal,
\(
\overline{\mathcal D}=\mathcal K.
\)
Similarly, minimality of $\widetilde{\mathbf V}$ gives
\[
\overline{
\operatorname{span}}
\left\{
\widetilde V_1^{k_1}\cdots
\widetilde V_n^{k_n}\widetilde Jh:
k_1,\ldots,k_n\geq0,\ h\in\mathcal H
\right\}
=
\widetilde{\mathcal K}.
\]
Therefore $U_0$ extends to a unitary operator
\[
U:\mathcal K\longrightarrow\widetilde{\mathcal K}.
\]
Moreover, for every generator
$V_1^{k_1}\cdots V_n^{k_n}Jh$,
\[
\begin{aligned}
UV_i
\left(
V_1^{k_1}\cdots V_n^{k_n}Jh
\right)
&=
U
\left(
V_1^{k_1}\cdots V_i^{k_i+1}\cdots
V_n^{k_n}Jh
\right)\\
&=
\widetilde V_1^{k_1}\cdots
\widetilde V_i^{k_i+1}\cdots
\widetilde V_n^{k_n}\widetilde Jh\\
&=
\widetilde V_iU
\left(
V_1^{k_1}\cdots V_n^{k_n}Jh
\right).
\end{aligned}
\]
By density,
\(
UV_i=\widetilde V_iU, 1\leq i\leq n.
\)
Also,
\(
UJ=\widetilde J.
\)
Thus $\mathbf V$ and $\widetilde{\mathbf V}$ are unitarily equivalent.
This completes the proof.
\end{proof}
\begin{rem}\label{rem:theta-n-final}
Theorem~\ref{thm:theta-n-sufficient} provides a mechanism for constructing
pure $\mathbf{\Theta}_n$-isometric dilations. In particular, the
construction does not require a uniqueness theorem for the fundamental
operator family of a general $\mathbf{\Theta}_n$-contraction.
Instead, the construction proceeds through a pure Hardy space model
\(
\left(
M_{\Phi_1},\ldots,M_{\Phi_{n-1}},M_z
\right)
\)
and an invariant subspace
\(
\Theta H^2(\mathcal F)
\)
satisfying the simultaneous intertwining relations
\(
\Phi_i\Theta=\Theta\Psi_i,
1\leq i\leq n-1.
\)
Consequently, a natural problem is to determine intrinsic conditions on a
given $\mathbf{\Theta}_n$-contraction which guarantee the existence of
such an invariant Hardy space model. Any such condition would yield a
sufficient condition for pure $\mathbf{\Theta}_n$-isometric dilation
without invoking the uniqueness problem for the fundamental operator
data. For $n=2$, the preceding results reduce to the corresponding
$\mathbf{\Theta}_2$ statements. In this case there is only one
operator-valued symbol,
\(
\Phi_1=\Phi,
\)
and the coefficient relation
\(
A_\ell^{(1)^*}
=
A_{p-\ell}^{(1)}
\)
becomes
\(
A_\ell^*=A_{p-\ell}.
\)
Thus Theorem~\ref{thm:pure-theta-n-dilation},
Theorem~\ref{thm:theta-n-invariant-subspace}, and
Theorem~\ref{thm:theta-n-sufficient} contain the $\mathbf{\Theta}_2$
results as the special case $n=2$.
\end{rem}
\subsection{Dilation of Doubly Commuting
$\mathbf{\Theta}_n$-Contractions}
\label{subsec:doubly-commuting-dilation}

We now consider the dilation theory of doubly commuting
$\mathbf{\Theta}_n$-contractions. Recall that a commuting tuple
\(
\mathbf T=(T_1,\ldots,T_n)
\)
on a Hilbert space $\mathcal H$ is said to be \emph{doubly commuting}
if
\(
T_iT_j^*=T_j^*T_i, i\neq j.
\)
The double commutation relations impose additional compatibility
conditions on an isometric dilation. In particular, when a
$\mathbf{\Theta}_n$-contraction admits a Hardy space realization
compatible with these relations, its minimal dilation inherits the
double commutation property. We record the following result, whose proof
is analogous to that of Theorem~\ref{thm:minimality-hardy-theta-n};
therefore, we omit the details.
\begin{thm}
\label{thm:doubly-commuting-theta-n-dilation}
Let
\(
\mathbf T=(T_1,\ldots,T_n)
\)
be a doubly commuting $\mathbf{\Theta}_n$-contraction on a Hilbert space
$\mathcal H$. Suppose that there exist a Hilbert space $\mathcal E$,
operator-valued polynomials
\[
\Phi_i(z)=\sum_{\ell=0}^{p}A_\ell^{(i)}z^\ell,
\quad 1\leq i\leq n-1,
\]
and an isometry
\(
J:\mathcal H\longrightarrow H^2(\mathcal E)
\)
such that
\[
JT_i=M_{\Phi_i}J,\quad 1\leq i\leq n-1,\quad
JT_n=M_zJ.
\]
Assume further that
\(
\mathbf V=
(M_{\Phi_1},\ldots,M_{\Phi_{n-1}},M_z)
\)
is a pure $\mathbf{\Theta}_n$-isometry and is doubly commuting; that is,
\(
M_{\Phi_i}M_{\Phi_j}^*
=
M_{\Phi_j}^*M_{\Phi_i},1\leq i,j\leq n-1,
\)
and
\(
M_{\Phi_i}M_z^*
=
M_z^*M_{\Phi_i}, 1\leq i\leq n-1.
\)
Define
\[
\mathcal K_J=
\overline{\operatorname{span}}
\left\{
M_{\Phi_1}^{k_1}\cdots
M_{\Phi_{n-1}}^{k_{n-1}}M_z^{k_n}J\mathcal H:
k_1,\ldots,k_n\geq0
\right\}.
\]
Then
\(
\mathbf V_J=
\left(
M_{\Phi_1}|_{\mathcal K_J},\ldots,
M_{\Phi_{n-1}}|_{\mathcal K_J},
M_z|_{\mathcal K_J}
\right)
\)
is a minimal doubly commuting pure
$\mathbf{\Theta}_n$-isometric dilation of $\mathbf T$.
\end{thm}
\begin{rem}
The theorem shows that, under the compatibility assumptions above, the
minimal Hardy space dilation preserves not only the
$\mathbf{\Theta}_n$-isometric structure but also the double commutation
relations. Thus, the additional hypothesis of double commutation is
inherited by the minimal dilation.
\end{rem}
\begin{prop}[Coefficient relations induced by double commutation]
\label{prop:theta-n-double-commuting-coefficients}
Let
\[
\Phi_i(z)=\sum_{\ell=0}^{p}A_\ell^{(i)}z^\ell,
\quad
1\leq i\leq n-1,
\]
be the symbols associated with a pure
$\mathbf{\Theta}_n$-isometry. Suppose that
\(
M_{\Phi_i}M_{\Phi_j}^*
=
M_{\Phi_j}^*M_{\Phi_i}, i\neq j.
\)
Then the corresponding symbols satisfy
\(
\Phi_i(z)\Phi_j(z)^*
=
\Phi_j(z)^*\Phi_i(z), z\in\mathbb T,
\)
and hence
\[
\sum_{\substack{0\leq r,s\leq p\\r-s=m}}
A_r^{(i)}A_s^{(j)*}
=
\sum_{\substack{0\leq r,s\leq p\\r-s=m}}
A_s^{(j)*}A_r^{(i)},
\qquad -p\leq m\leq p.
\]
In particular, the coefficient family of a doubly commuting pure
$\mathbf{\Theta}_n$-isometry satisfies, in addition to
\[
A_\ell^{(i)^*}
=
A_{p-\ell}^{(n-i)},
\qquad 0\leq\ell\leq p,
\]
the above finite system of algebraic relations.
\end{prop}

\begin{proof}
For $z\in\mathbb T$, write
\(
\Phi_i(z)=\sum_{r=0}^{p}A_r^{(i)}z^r,
\Phi_j(z)^*
=
\sum_{s=0}^{p}A_s^{(j)*}z^{-s}.
\)
Thus
\[
\Phi_i(z)\Phi_j(z)^*
=
\sum_{r,s=0}^{p}
A_r^{(i)}A_s^{(j)*}z^{r-s},
\]
while
\(
\Phi_j(z)^*\Phi_i(z)
=
\sum_{r,s=0}^{p}
A_s^{(j)*}A_r^{(i)}z^{r-s}.
\)
The assumed double commutation relation yields equality of the two
corresponding Laurent polynomial symbols. Comparing the coefficients of
$z^m$, $-p\leq m\leq p$, gives
\[
\sum_{\substack{0\leq r,s\leq p\\r-s=m}}
A_r^{(i)}A_s^{(j)*}
=
\sum_{\substack{0\leq r,s\leq p\\r-s=m}}
A_s^{(j)*}A_r^{(i)}.
\]
The final assertion follows from the standard coefficient symmetry for
pure $\mathbf{\Theta}_n$-isometries.
\end{proof}
\begin{cor}
\label{cor:theta-n-doubly-commuting-hardymodel}
Let $\mathbf T$ be a doubly commuting
$\mathbf{\Theta}_n$-contraction satisfying the hypotheses of
Theorem~\ref{thm:doubly-commuting-theta-n-dilation}. Then $\mathbf T$
admits a minimal pure doubly commuting
$\mathbf{\Theta}_n$-isometric dilation on a Hilbert space of the form
\(
\mathcal K_J\subseteq H^2(\mathcal E).
\)
Moreover, the dilation is determined by the operator-valued polynomial
symbols
\[
\Phi_i(z)=\sum_{\ell=0}^{p}A_\ell^{(i)}z^\ell,
\quad 1\leq i\leq n-1,
\]
which satisfy
\(
A_\ell^{(i)^*}
=
A_{p-\ell}^{(n-i)}
\)
together with the coefficient relations in
Proposition~\ref{prop:theta-n-double-commuting-coefficients}.
\end{cor}
\subsection{Special class of  $\mathbf{\Theta}_n$-contractions and its $\mathbf{\Theta}_n$-isometric dilations}
In this section, we identify and study one such class of $\mathbf{\Theta}_n$-contractions. We show that the structural properties of this class allow for the construction of a $\mathbf{\Theta}_n$-isometric dilation and thereby provide an additional dilation mechanism beyond the conditional dilation considered earlier. This result further illustrates the richness of the dilation theory for $\mathbf{\Theta}_n$-contractions and completes our study by presenting a class with a distinct isometric dilation structure.
\begin{thm}\label{Thm 9}
Let $(T_1,T_2)$ be a pair of commuting contractions on a Hilbert space $\mathcal{H}$. Then
\(
\mathbf{T}=(T_1^m+T_2^m,(T_1T_2)^{m/p})
\)
is a $\mathbf{\Theta}_2$-contraction.
\end{thm}

\begin{proof}
Define the map $\pi:\mathbb{C}^2\to\mathbb{C}^2$ by
\[
\pi(z_1,z_2)=\bigl(z_1^m+z_2^m,(z_1z_2)^{m/p}\bigr).
\]
Observe that
$
\pi(\overline{\mathbb{D}}^2)\subseteq\overline{\mathbf{\Theta}}_2.$
Hence, for every polynomial $p\in\mathbb{C}[z_1,z_2]$, we have
\[
\begin{aligned}
\|p(\mathbf T)\|
&=\|(p\circ\pi)(T_1,T_2)\|\\
&\le \|p\circ\pi\|_{\infty,\overline{\mathbb{D}}^2}
\qquad\text{(by von Neumann's inequality)}\\
&=\|p\|_{\infty,\pi(\overline{\mathbb{D}}^2)}\\
&\le \|p\|_{\infty,\overline{\mathbf{\Theta}}_2}.
\end{aligned}
\]
Therefore, $\mathbf{T}$ is a $\mathbf{\Theta}_2$-contraction. This completes the proof.
\end{proof}
The next theorem shows that the $\mathbf{\Theta}_2$-contraction constructed in Theorem~\ref{Thm 9} always admits a $\mathbf{\Theta}_2$-isometric dilation.
\begin{thm}\label{Thm 10}
Let $(T_1,T_2)$ be a pair of commuting contractions on a Hilbert space $\mathcal{H}$. Then
\(
\mathbf{T}=(T_1^m+T_2^m,(T_1T_2)^{m/p})
\)
admits a $\mathbf{\Theta}_2$-isometric dilation.
\end{thm}

\begin{proof}
Since $(T_1,T_2)$ is a pair of commuting contractions, it admits an Ando isometric dilation. Let $(V_1,V_2)$ be an Ando isometric dilation of $(T_1,T_2)$. Define
\[
\mathbf{V}=(V_1^m+V_2^m,(V_1V_2)^{m/p}),
\]
and set
\(
\widetilde{V}_1=V_1^m+V_2^m,
\widetilde{V}_2=(V_1V_2)^{m/p}.
\)
It is straightforward to verify that $\widetilde{V}_2$ is an isometry and
$
\widetilde{V}_1=\widetilde{V}_1^*\widetilde{V}_2^{\,p}.$
Moreover,
$
\frac{1}{2}\widetilde{V}_1$ is a $\mathbf{\Theta}_1$-contraction. Therefore, Theorem~\ref{Thm 5} implies that $\mathbf{V}$ is a $\mathbf{\Theta}_2$-isometry. Since $(V_1,V_2)$ is an isometric dilation of $(T_1,T_2)$, it follows that $\mathbf{V}$ is a $\mathbf{\Theta}_2$-isometric dilation of $\mathbf{T}$. This completes the proof.
\end{proof}
\begin{rem}\label{Rem 2}
Bhattacharyya, Pal, and Roy proved in \cite[Theorem~4.3]{Roy} that every $\Gamma$-contraction admits a $\Gamma$-isometric dilation. However, the constructions of the operators $V_i$ in Theorems~\ref{Conditional Dilation} and \ref{Re Conditional Dilation} do not establish that every $\mathbf{\Theta}_2$-contraction admits a $\mathbf{\Theta}_2$-isometric dilation when $m\ge2$ and $p\ge1$. Consequently, whether every $\mathbf{\Theta}_2$-contraction admits a $\mathbf{\Theta}_2$-isometric dilation remains an open problem. More generally, the following problem remains open.

\medskip

\noindent\textbf{Open Problem.}
Does every $\mathbf{\Theta}_n$-contraction admit a $\mathbf{\Theta}_n$-isometric dilation for arbitrary integers $n\ge2$, $m\ge2$, and $p\ge1$?
\end{rem}

	\end{document}